\documentclass[11pt]{article}
\usepackage[letterpaper, textwidth=17cm, centering, top=2.5cm, bottom=2.5cm]{geometry}
\usepackage[utf8]{inputenc}
\usepackage{amsthm}
\usepackage{mathtools}
\usepackage{amsmath}
\usepackage{bm}
\usepackage{yhmath}
\usepackage{enumerate}
\usepackage{subfigure}
\usepackage{graphicx}
\usepackage{algorithm}
\usepackage{multirow}
\usepackage{dcolumn}
\usepackage{array}
\usepackage{algorithmicx}  
\usepackage{algpseudocode} 
\usepackage{amssymb} 
\usepackage{mathrsfs}
\def\diff{\mathrm{d}}
\def\P{\mathbb{P}}
\def\R{\mathbb{R}}

\def\S{\mathbb{S}}

\def\E{\mathbb{E}}
\def\la{\langle}

\def\PSD{{\S^k_{++}}}
\def\ra{\rangle}

\def\trans{\mathsf{T}}

\usepackage{natbib} 

\usepackage{amsmath}
\usepackage{amsfonts}
\usepackage{amssymb}
\usepackage{xcolor}
\usepackage{epsfig}

\usepackage{amsmath}
\usepackage{amsfonts}
\usepackage{amssymb}
\usepackage{dsfont}
\usepackage{mathrsfs}
\usepackage{bbm}
\usepackage{amsmath}
\usepackage{amsfonts}
\usepackage{amssymb}
\usepackage{multirow}
\usepackage{mathrsfs}
\let\amssymbboxplus\boxplus
\let\amssymbboxminus\boxminus

\renewcommand{\boxplus}{\mathbin{\mathop\amssymbboxplus}}
\renewcommand{\boxminus}{\mathbin{\mathop\amssymbboxminus}}

\usepackage{makecell} 
\usepackage{boldline}
\usepackage{diagbox}
\usepackage{threeparttable}
\usepackage{adjustbox}

\usepackage[title]{appendix}

\usepackage[colorlinks=true]{hyperref}

\renewcommand{\Box}{\framebox{\rule{0.3em}{0.0em}}}

\providecommand{\keywords}[1]{\textbf{\textit{Keywords:}} #1}

\newtheorem{theorem}{Theorem}[section]
\newtheorem{lemma}{Lemma}[section]
\newtheorem{proposition}{Proposition}[section]

\newtheorem{remark}{Remark}[section]

\newcommand{\setd}{{ d \kern -.15em l}}
\newcommand{\hatsetd}{ d \hat{\kern -.15em l }}
\newcommand{\dd}{\mathsf {d\kern -0.07em l}}

\newcommand{\bgeqn}{\begin{eqnarray}}
\newcommand{\edeqn}{\end{eqnarray}}
\newcommand{\bgeq}{\begin{eqnarray*}}
\newcommand{\edeq}{\end{eqnarray*}}

\newcommand{\inmat}[1]{\mbox{\rm {#1}}}

\def\diff{\mathrm{d}}
\def\P{\mathbb{P}}
\def\R{\mathbb{R}}
\def\N{\mathbb{N}}

\def\K{\mathbb{K}}
\def\S{\mathbb{S}}

\def\E{\mathbb{E}}
\def\la{\langle}

\def\PSD{{\S^n_{+}}}
\def\PD{{\S^n_{++}}}
\def\ra{\rangle}
\def\trans{\mathsf{T}}

\usepackage[colorlinks=true]{hyperref}
 \hypersetup{urlcolor=blue,
 citecolor=blue,linkcolor=blue}
   
\makeatletter
\g@addto@macro{\UrlBreaks}{\UrlOrds}
\makeatother

\begin{document}

\allowdisplaybreaks
\title{
Distributional stability of sparse inverse covariance matrix estimators\footnote{Dedicated to the memory of Professor Werner R\"omisch, a pioneer in stability analysis of stochastic optimization.}
}

 \author{
Renjie Chen\footnote{
 Department of Systems Engineering and Engineering Management, 
 The Chinese University of Hong Kong. Email: rchen@se.cuhk.edu.hk.}\and  Huifu Xu\footnote{ Department of Systems Engineering and Engineering Management, The Chinese University of Hong Kong. Email: hfxu@se.cuhk.edu.hk.}
\and Henryk Z\"ahle\footnote{Department of Mathematics, Saarland University, Saarbrücken, Germany. Email: zaehle@math.uni-sb.de }
}

\maketitle

\begin{abstract}
Finding an approximation of the inverse of the covariance matrix, also known as precision matrix, of a random vector with empirical data is widely discussed in finance and engineering. In data-driven problems, empirical data may be ``contaminated''. This raises the question as to whether the approximate precision matrix is reliable from a statistical point of view. In this paper, we concentrate on a much-noticed sparse estimator of the precision matrix and investigate the issue from the perspective of distributional stability. Specifically, we derive an explicit local Lipschitz bound for the distance between the distributions of the sparse estimator under two different distributions (regarded as the true data distribution and the distribution of ``contaminated'' data). The distance is measured by the Kantorovich metric on the set of all probability measures on a matrix space. We also present analogous results for the standard estimators of the covariance matrix and its eigenvalues. Furthermore, we discuss {several} 
applications and conduct some numerical experiments.
\end{abstract}


\allowdisplaybreaks


\section{Introduction}\label{Sec:Introduction}

Let $\xi$ be an $\R^n$-valued random variable on a probability space $(\Omega,{\cal F},\P)$. The statistical estimation of its covariance matrix $\Sigma=\E[(\xi-\E[\xi])(\xi-\E[\xi])^\trans]$ and the inverse $\Sigma^{-1}$ thereof, if it exists, is a standard and long standing problem in multivariate statistics with wide applications in finance and engineering. The inverse matrix $\Sigma^{-1}$, which is also known as precision matrix, is needed, for example, for optimal decision or model selection such as linear discriminant analysis \cite{fisher1936use}, portfolio optimization \cite{fabozzi2008portfolio} and graphical model selection \cite{Venkat2012latent,drton--jmlr-2008,yuan2007model}. 
Given independent copies $\xi^1,\ldots,\xi^N$ of $\xi$, the classical nonparametric estimators of $\Sigma$ and $\Sigma^{-1}$ are the sample covariance matrix $\widehat\Sigma_N:= \frac{1}{N}\sum_{i=1}^N\xi^i{\xi^i}^\trans -(\frac{1}{N}\sum_{i=1}^N\xi^i)(\frac{1}{N}\sum_{i=1}^N\xi^i)^\trans$ and its inverse $\widehat\Sigma_N^{-1}$, if it exists, respectively. In the case where it exists, $\widehat\Sigma_N^{-1}$ is called the sample precision matrix.

In this article, our main focus is on the estimation of the precision matrix $\Sigma^{-1}$, which is subject to some problems. Two of these problems are the following. First, the sample precision matrix $\widehat\Sigma_N^{-1}$, i.e.\ the inverse of the sample covariance matrix, can fail to exist even if the precision matrix $\Sigma^{-1}$ exists. Second, the sample precision matrix $\widehat\Sigma_N^{-1}$ can fail to have a sparse structure even if the precision matrix $\Sigma^{-1}$ has a sparse structure. Recall that sparsity of covariance and precision matrices is a basic  requirement for many applications such as model selection problems, see e.g.\ \cite{bien2011sparse,dempster1972covariance,lam2009sparsistency,meinshausen-aos-2006,rothman2008sparse,yuan2007model}. To overcome these two problems, Banerjee et.~al.\ \cite{10.5555/1390681.1390696} 
introduced, in the scope of a graphical model selection problem, the estimator
\begin{align}\label{eq:L1-MLE}
    \widehat S_N := \arg\min_{S\in\PD}\big(\la\widehat{\Sigma}_N, S \ra - \log(\det S) + \lambda \|S\|_{1}\big)
\end{align}
of $\Sigma^{-1}$, where $\PD$ is the set of all symmetric and positive definite matrices in $\R^{n\times n}$, $\la\widehat{\Sigma}_N, S \ra$ is defined to be the trace of the matrix $\widehat{\Sigma}_N^\trans S$, $\|S\|_{1}:=\sum_{1\leq i,j\leq n} |S_{i,j}|$ and $\lambda\in\R_+:=[0,\infty)$. 
For $\lambda=0$, the estimator $\widehat S_N$ is well defined only for samples $\xi^1,\ldots,\xi^N$ for which the symmetric matrix $\widehat\Sigma_N$ is positive definite and in this case it boils down to the sample precision matrix $\widehat\Sigma_N^{-1}$ (see Appendix \ref{Appendix:motivation SN}). Analogous to the classical Lasso estimator, Banerjee et.~al.\ \cite{10.5555/1390681.1390696} introduce the regularization term to overcome such rank deficient circumstance. For $\lambda>0$, the last term on the right-hand side in (\ref{eq:L1-MLE}), i.e.\ $\lambda \|S\|_{1}$, is an $\ell_1$-penalty term which penalizes deviations from sparsity, i.e., more precisely, a large number of nonzero entries. The parameter $\lambda$ controls the intensity of the penalty, which is consistent with techniques used in regression problems, such as the lasso. The choice $\lambda>0$ has the advantage over $\lambda=0$ not only that the minimizer becomes sparser, but also that the estimator $\widehat S_N$ is well defined for {\em each} sample $\xi^1,\ldots,\xi^N$ (this is ensured by Proposition \ref{prop:boundns-of-sln}). Moreover, it is proven in~\cite{ravikumar2011high,rothman2008sparse} that when $\lambda$ declines to $0$ with suitable order, the proposed estimator $\widehat S_N$ is consistent. In \cite{ravikumar2011high}, it is pointed out that the right-hand side in \eqref{eq:L1-MLE} can also be viewed as an $\ell_1$-penalized Bregman divergence minimization problem.

The estimator $\widehat{S}_N$ and variants of it have also been considered in several other papers, such as \cite{d2008first,friedman2008sparse,ravikumar2011high,yuan2007model}. Convergence analysis (in terms of Frobenius and spectral norms) and model selection consistency can be found in \cite{rothman2008sparse} for Gaussian random variables, in \cite{ravikumar2011high} for general random variables and in \cite{lam2009sparsistency} for an estimator that is based on a version of \eqref{eq:L1-MLE} with more general (non-convex) regularization term. In \cite{nguyen2022distributionally}, the authors proposed a so-called Wasserstein shrinkage estimator of the precision matrix based on a distributionally robust version of \eqref{eq:L1-MLE}.
In \cite{wang2010solving,yang2013proximal}, the authors discuss numerical methods to solve \eqref{eq:L1-MLE}.



%


In this paper, we intend to investigate the extent to which the estimator $\widehat S_N$ is distributional stable to deviations in the underlying distribution. By underlying distribution we mean the distribution $\P\circ\xi^{-1}$ of the random variable $\xi$, 
i.e.\ the image measure 
of $\P$ w.r.t.\ $\xi$, which is a Borel probability measure on $\R^n$. 
The set of all admissible underlying distributions is the set $\mathscr{P}_2(\R^n)$ of all Borel probability measures $P$ on $\R^n$ with $\int_{\R^n}\|x\|^2\,P(dx)<\infty$. To properly frame the question we want to raise and answer, we need to change over to the canonical setting. So, set
\begin{align*}
    (\Omega,{\cal F}):=\big((\R^n)^N,{\cal B}(\R^n)^{\otimes N}\big)\quad\mbox{ and }\quad \P^P:=P^{\otimes N}\mbox{ for any } P\in\mathscr{P}_2(\R^n)
\end{align*}
and let $\boldsymbol{\xi}=(\xi^1,\ldots,\xi^N)$ be the identity on $\Omega$. Then $\xi^1,\ldots,\xi^N$ are i.i.d.\ according to $P$ under $\P^P$ for any $P\in\mathscr{P}_2(\R^n)$. In particular, $(\Omega,{\cal F},\{\P^P:P\in\mathscr{P}_2(\R^n)\})$ is a nonparametric statistical model and $\widehat S_N=\widehat S_N(\boldsymbol{\xi})$ is a point estimator of the precision matrix of $\xi^1$ in this model. For any $P\in\mathscr{P}_2(\R^n)$, the image measure $\P^P\circ\widehat S_N^{-1}$ is a probability measure on $(\PD,{\cal B}(\PD))$ and specifies the distribution of the estimator $\widehat S_N$ if the underlying distribution is $P$. Here we assume that $\PD$ is equipped with the Frobenius norm, such that the mapping $\widehat S_N:\Omega\to\PD$ is Borel measurable, since it is continuous (which follows from Theorem \ref{thm-prci-mtx-Lip-hld-cnty} and the continuity of $\widehat\Sigma_N:\Omega\to\PSD$).

We want to investigate the question of how much the distribution $\P^P\circ\widehat S_N^{-1}$ changes when the underlying distribution $P$ is replaced by another distribution $Q$. A change from $P$ to $Q$ can be due to a ``contamination'' of the sample data, caused, for example, by unsystematic outliers, random measurement errors or simply observing realizations from the wrong distribution.  
The latter refers to the phenomenon that observations are often not drawn exactly from the distribution in which one is actually interested, simply because external influences can change over time (or observations can be mislabeled).
As early as 1971, Hampel \cite{hampel1971general} introduced the concept of qualitative robustness of a nonparametric estimator $\widehat T_N$, which basically means that the distribution of the estimator, here $\P^P\circ\widehat T_N^{-1}$, is continuous in the underlying distribution $P$ (uniformly in $N$), where sets of probability measures are equipped with (metrics metrizing) the respective weak topologies (see also the monographs \cite{huber2004robust,huber2009robust}). In the $2010$th, the property of qualitative robustness has attracted some interest in the field of quantitative risk management, see, for instance, \cite{BELLINI2022270,cont2010robustness,KochMedinaMunari+2014+215+236,kratschmer2012qualitative,kratschmer2014comparative,kratschmer2017domains,https://doi.org/10.1111/sjos.12259,WEBER2018191}.  
In \cite{zahle2015qualitative,zahle2016definition}, it was extended to general estimators and general statistical models.

Hampel's concept of qualitative robustness has at least two limitations. First, the distance $\dd(\P^P\circ\widehat T_N^{-1},\P^Q\circ\widehat T_N^{-1})$ is measured w.r.t.\ metrics $\dd$ which metrize the weak topology (such as L\'evy, Prohorov or bounded Lipschitz). This is a little unsatisfactory because, for example, in the case where $\widehat T_N$ is real-valued, the means of $\P^P\circ\widehat T_N^{-1}$ and $\P^Q\circ\widehat T_N^{-1}$ can be arbitrarily far apart, even if $\dd(\P^P\circ\widehat T_N^{-1},\P^Q\circ\widehat T_N^{-1})\le\varepsilon$ for some given $\varepsilon>0$. Second, no explicit bounds for $\dd(\P^P\circ\widehat T_N^{-1},\P^Q\circ\widehat T_N^{-1})$ in terms of $\dd'(P,Q)$ are required, where $\dd'$ is a metric on the set of admissible underlying distributions. In this respect, qualitative robustness differs from  stability concepts in the fields of optimization and numerical analysis, where stability is usually interpreted as Lipschitz continuity.

For these reasons, Guo and Xu \cite{guo2021statistical} extended, in the scope of a preference optimization problem, the conventional robustness concept in two direction. First, they replaced $\dd$ by the Kantorovich metric $\dd_{\rm K}$, which ensures, for real-valued $\widehat T_N$, that the means of $\P^P\circ\widehat T_N^{-1}$ and $\P^Q\circ\widehat T_N^{-1}$ are close to each other when $\dd_{\rm K}(\P^P\circ\widehat T_N^{-1},\P^Q\circ\widehat T_N^{-1})$ is small. Second, they aimed at Lipschitz bounds in the form of $\dd_{\rm K}(\P^P\circ\widehat T_N^{-1},\P^Q\circ\widehat T_N^{-1})\le L\,\dd'(P,Q)$ for all admissible $P,Q$, for some constant $L>0$, where $\dd'$ is chosen appropriately. Such Lipschitz bounds have recently been established in other applications as well, see \cite{guo2023data,guo2023statistical,wang2021quantitative}. In this paper, we establish a corresponding (local) Lipschitz bound for the sparse precision matrix estimator $\widehat S_N$ defined in \eqref{eq:L1-MLE}. In fact, Theorem \ref{thm-SR-Sprs-Preci-mtx} shows that there exists a constant $L_\lambda>0$, depending only on $\lambda$ (i.e.\ being independent of $P$, $Q$ and $N$), such that
\begin{align}\label{dist_stable}
    \dd_{\rm K}\big(\P^P\circ\widehat S_N^{-1},\P^Q\circ\widehat S_N^{-1}\big)
    & \,\le\, L_\lambda\max\{3,2m_P,2m_Q\}\, \dd_2(P,Q)
\end{align}
where $\dd_2$ is the second order Fortet-Mourier metric on $\mathscr{P}_2(\R^n)$ and $m_P$ and $m_Q$ are the absolute means of $P$ and $Q$, respectively.

Note that the Borel probability measures
$\P^P\circ\widehat S_N^{-1}$ and 
$\P^Q\circ\widehat S_N^{-1}$
on the left-hand side in \eqref{dist_stable} are distributions of random matrices (with samples in $\PD$). For an overview on random matrix theory we refer to \cite{edelman2005random,tao2023topics}, and we point out that random matrices have wide applications in engineering \cite{tulino2004random}, physics \cite{beenakker1997random}, random graph theory \cite{janson2011random}, neural networks \cite{louart2018random} and, most related to this paper, statistics. The distribution of the sample covariance matrix $\widehat\Sigma_N$, first studied in \cite{wishart1928generalised}, has been a central research problem in multivariate statistical analysis. The analysis of spectrum behavior of random matrix founds a base for principle component analysis \cite{johnstone2001distribution} and factor analysis models. The empirical distribution of the eigenvalues of random matrices, often referred to as empirical spectral distribution, was first studied in \cite{Wigner1958} and considered in detail for sample covariance matrices in \cite{marchenko1967distribution}. The eigenvector empirical spectral distribution was studied in \cite{bai2007asymptotics,silverstein1981describing} and, more recently, in \cite{ningning2013Convergence}.

With regard to the right-hand side of (\ref{dist_stable}), it is worth noting that obviously $m_Q\le m_P+|m_P-m_Q|\le m_P+\dd_1(P,Q)$, where $\dd_1$ is the first order Fortet-Mourier metric on $\mathscr{P}_1(\R^n)$. Bounds analogous to (\ref{dist_stable}) are also obtained for the sample covariance matrix $\widehat \Sigma_N$ (see Theorem \ref{thm-SR-Covn-mtx}) and the eigenvalues of $\widehat \Sigma_N$ (see Theorem \ref{Statistical-robustness-of-eigenvalue}). In {Section \ref{sec:app-numerical}} we apply these results to a Gaussian graphical model selection problem 
and the determination of an insurance company's solvency capital requirement, among others.
{Combining (\ref{dist_stable}) with (\ref{eq:consistency-precision-matrix-1}) and (\ref{thm-SR-Covn-mtx:REMARK-EQ2}), we also obtain that
\begin{align}\label{dist_stable-2}
    \dd_{\rm K}\big(\P^Q\circ\widehat S_N^{-1},\delta_{S_P}\big) & \,\le\, L_\lambda\max\{3,2m_P,2m_Q\}\, \dd_2(P,Q)+o\big(N^{-(r-1)/r}\big)
\end{align}
for all $r\in(1,2]$ and $P,Q\in\mathscr{P}_2(\R^n)$ with $\int_{\R^n}\|x\|^{2r}\,P(\diff x)<\infty$, where $\delta_{S_P}$ is the Dirac measure at the precision matrix $S_P$ given the underlying model is $P$ and the $o$-term on the right-hand side of (\ref{dist_stable-2}) only depends on $P$  (i.e.\ it is independent of $Q$). Inequality (\ref{dist_stable-2}) shows that $\widehat S_N$ can be reasonably estimated even if the data are drawn from a ``contaminated'' distribution $Q$ that is slightly different from the target distribution $P$.

The rest of this article is organized as follows. In Section \ref{Sec:notation}, we first introduce some basic notation. In Section \ref{Sec:Acfsidogse}, we introduce a criterion for distributional stability of general point estimators, which we use in Section \ref{Sroeocmapm} to prove our main results. For one of our main results, we also need information on the optimization problem underlying the estimator $\widehat S_N$, which we provide in Section \ref{Sec: OputeSN}. 
{Some applications and numerical experiments can be found in Section \ref{sec:app-numerical}}.
All the proofs are given in Appendix \ref{Sec:Proofs}, some of which rely on the auxiliary results in Appendix \ref{Sec:Some auxiliary results}.
Some open questions for future research are listed in
Section \ref{sec:concluding-remarks}.



\section{Basic notation}
\label{Sec:notation}

{Set $\N:=\{1,2,\ldots\}$ and let $n\in\N$.} We use $\R^n$ to denote the $n$-dimensional Euclidean space and $\|x\|$ to denote the Euclidean norm of $x\in\R^n$. We set $\R_+:=[0,\infty)$ and $\R_{++}:=(0,\infty)$. 
We use $\R^{n\times n}$ to denote the set of all $n\times n$ matrices with real entries and $\mathbb{S}^n$ to denote the linear space of all symmetric matrices in $\R^{n\times n}$. Moreover, we use $\mathbb{S}^n_+$ to denote the set of all positive semi-definite matrices in $\mathbb{S}^n$, and $\mathbb{S}^n_{++}$ to denote the set of all positive definite matrices in $\mathbb{S}^n$. For a matrix $A\in\R^{n\times n}$, we write $A_{i,j}$ for its $(i,j)$-th entry.
For any $A, B\in \R^{n\times n}$, we write $\la A,B\ra$ for the {Frobenius} inner product of $A$ and $B$, which is defined as the trace $\mathrm{tr}(A^TB)$ of $A^TB$. For any matrix $A\in\R^{n\times n}$, we write $\|A\|$ for its Frobenius norm, i.e.\ $\|A\|:=\sqrt{\la A,A\ra}$ (which will be used mostly
in the paper), and $\|A\|_1:=\sum_{i=1}^n\sum_{j=1}^n |A_{i,j}|$ and $\|A\|_2:=\sup_{\|x\|=1}\|Ax\|$
for the $1$-norm and the $2$-norm (spectral norm), respectively. Unless otherwise stated, the linear spaces $\R^{n\times n}$ and 
$\S^n$ are equipped with the Frobenius norm $\|\cdot\|$.
{The subsets $\PSD$ and $\PD$ are equipped with the respective induced (Frobenius) norms, regarded as metrics. Note that $\PD$ is an open subset of $(\S^n,\|\cdot\|)$.} 
We use $\mathrm{diag}(x_1,\ldots,x_n)$ to represent a diagonal matrix with diagonal entries $x_1,\ldots,x_n\in\R$. 
Furthermore, we use $\bm 0$ to denote the zero vector or zero matrix in the space depending on the context.


For any normed linear space $(X,\|\cdot\|_X)$, we use $\mathscr{P}(X)$ to denote the set of all Borel probability measures on $X$. For any $p\in[1,\infty)$, we write $\mathscr{P}_p(X)$ for the subset of all $P\in\mathscr{P}(X)$ satisfying $\int_X \|x\|_X^p\,P(\diff x)<\infty$ and ${\cal F}_p(X)$ for the set of all functions $\psi:X\to \R$ satisfying $|\psi(\hat x)-\psi(\tilde{x})| \leq L_p(\hat x,\tilde{x})\|\hat x-\tilde{x}\|$ for all $\hat x,\tilde{x}\in X$, where $L_p(\hat x,\tilde{x}):=\max\{1,\|\hat x\|_X,\|\tilde{x}\|_X\}^{p-1}$. For any $p\in[1,\infty)$, the $p$-th order Fortet-Mourier metric on $\mathscr{P}_p(X)$ is defined by
\begin{align}\label{eq:dfn-Fortet-Mourier}
        \dd_{X,p}(P,Q) := \sup_{\psi\in \mathcal{F}_p(X)} \Big|\int_X \psi(x)\,P(\diff x) - \int_X \psi(x)\,Q(\diff x)\Big|,\quad P, Q\in \mathscr{P}_p(X).
\end{align}
In the case $p=1$, the Fortet-Mourier metric $\dd_{X,p}(P,Q)$ recovers the well-known Kantorovich metric (also known as Wasserstein distance). Note that Fortet-Mourier metrics are extensively used in stability analysis of stochastic programming, see \cite{romisch2003stability} for an overview. 



\section{A criterion for distributional stability of general estimators}\label{Sec:Acfsidogse}

The main results of this paper, i.e.,\ distributional stability of the estimators of the covariance matrix and the precision matrix, will be presented in Section \ref{Sroeocmapm}. These results rely on Theorem \ref{thm-SR-gnl-nom-space} below, which addresses the distributional stability of general statistical estimators. Indeed, applying Theorem \ref{thm-SR-gnl-nom-space} to the sample covariance $\widehat\Sigma_N$ (in the role of $\widehat T_N$) shows that $\widehat\Sigma_N$ is distributionally stable (see Theorem \ref{thm-SR-Covn-mtx}). To conclude that the sparse estimator $\widehat S_N$ of the precision matrix $\Sigma^{-1}$ (introduced in (\ref{eq:L1-MLE})) is also distributionally stable (see Theorem \ref{thm-SR-Sprs-Preci-mtx}), we need some additional arguments which are prepared in Section \ref{Sec: OputeSN}.

Let $(X,\|\cdot\|_X)$ be a normed linear space and consider, for any fixed $N\in\mathbb{N}$, the measurable space
 $   \big(\Omega,{\cal F}\big):=\big(X^N,{\cal B}(X)^{\otimes N}\big)$,
which is regarded as the sample space. Moreover, set
\begin{align*}
    \P^P:=P^{\otimes N}\quad\mbox{ for any $P\in\mathscr{P}_2(X)$}.
\end{align*}
For any $P\in\mathscr{P}_2(X)$, the $N$ coordinate projections on $\Omega=X^N$, denoted by $\xi^1,\ldots,\xi^N$, are, under $\P^P$, i.i.d.\ according to $P$. Note that $(\Omega,{\cal F},\{\P^P:P\in\mathscr{P}_2(X)\})$ is a nonparametric statistical model and ${\bm\xi}:=(\xi^1,\ldots,\xi^N)$ ($=$ identity on $\Omega$) is the sample variable on it. Let $(Y,\|\cdot\|_Y)$ be another normed linear space and $\widehat{T}_N:\Omega\to Y$
be a $({\cal B}(\Omega),{\cal B}(Y))$-measurable mapping, which is regarded as a statistical estimator.
Moreover, set $$
m_P:=\E^P[\|\xi^1\|_X]:=\int_\Omega\|\xi^1\|_X\,\diff\P^P=\int_X \|x\|_X\,\diff P
$$
for any $P\in\mathscr{P}_2(X)$ and
recall that $L_2(x,\tilde{x})=\max\{1,\|x\|_X,\|\tilde{x}\|_X\}$ for all $\hat x,\tilde x\in X$. Moreover, recall from Section \ref{Sec:notation} that $\dd_{X,2}$ and $\dd_{Y,1}$ refer to the second order Fortet-Mourier metric on $\mathscr{P}_2(X)$ and the first order Fortet-Mourier metric on $\mathscr{P}_1(Y)$, respectively.

\begin{theorem}
\label{thm-SR-gnl-nom-space}
Assume that there exist constants $\kappa_1,\kappa_2\in\R_{+}$ such that
\begin{align}\label{eq:general-SR-condition}
    \big\|\widehat{T}_N(\hat{\bm{x}})-\widehat{T}_N(\tilde{\bm{x}})\big\|_Y\nonumber
    & \,\leq\, \frac{\kappa_1}{N}\sum_{i=1}^N L_2(\hat{x}^i,\tilde{x}^i) \|\hat{x}^i-\tilde{x}^i\|_X\\
    & \qquad+\frac{\kappa_2}{N^2} \sum_{i=1}^N \big(\|\hat{x}^i\|_X+\|\tilde{x}^i \|_X\big)\sum_{j=1}^N \|\hat{x}^j-\tilde{x}^j\|_X
\end{align}
holds true for all $\hat{\bm x}=(\hat x^1,\ldots,\hat x^N),\,\tilde{\bm x}=(\tilde x^1,\ldots,\tilde x^N)\in X^{N}$. Then 
\begin{align}\label{eq:SR-gnl-nom-space}
    \dd_{Y,1} \big(\P^P \circ \widehat{T}_N^{-1}, \P^Q\circ \widehat{T}_N^{-1}\big) \leq \max\big\{\kappa_1+\kappa_2,2\kappa_2m_P,2\kappa_2m_Q\big\}\,\dd_{X,2}(P,Q)
\end{align}
for all $P,Q\in\mathscr{P}_2(X)$ and $N\in\mathbb{N}$.
\end{theorem}



The result of Theorem \ref{thm-SR-gnl-nom-space} is similar to that of Theorem 4.5 in \cite{wang2021quantitative}. However, 
they differ in that the right-hand side of (\ref{eq:SR-gnl-nom-space}) contains the terms $m_P$ and $m_Q$. This is because $\widehat{T}_N$ is assumed to satisfy condition (\ref{eq:general-SR-condition}) with the right-hand side having a term $\frac{1}{N^2} \sum_{i=1}^N (\|\hat{x}^i\|_X+\|\tilde{x}^i\|_X)\sum_{j=1}^N \|\hat{x}^j-\tilde{x}^j\|_X$. We consider this term because it allows the theorem to be applied to the sample covariance matrix. 
It might also be helpful to note that Inequality (\ref{eq:SR-gnl-nom-space}) holds for all $N\in\mathbb{N}$ rather than only for $N$ sufficiently large as in many asymptotic statistical analyses. 
This is primarily because the statistical estimator $\widehat{T}_N$ is assumed to satisfy the local Lipschitz property 
(\ref{eq:general-SR-condition})
with respect to the perturbation of sample data for all sample sizes $N$. This Lipschitz property 
controls the Kantorovich distance between  $\P^P \circ \widehat{T}_N^{-1}$ and $ \P^Q\circ \widehat{T}_N^{-1}$.
In our view, this is the unique feature of quantitative statistical robustness 
as introduced in \cite{guo2021statistical}, which effectively relates Lipschitz continuity of a statistical estimator to the Kantorovich distance of the distributions of the associated  estimators based on two different 
sampling distributions.

The following proposition is trivial, but worth noting. In the proposition, $(Z,\|\cdot\|_Z)$ is another normed linear space.

\begin{proposition}\label{prop:chain-rule}
Assume that there exists a constant $L_1\in\R_{+}$ such that 
\begin{align*}
    \dd_{Y,1} \big(\P^P \circ \widehat{T}_N^{-1}, \P^Q\circ \widehat{T}_N^{-1}\big) \,\le\, L_1\,\dd_{X,2}(P,Q)
\end{align*}
for all $P,Q\in\mathscr{P}_2(X)$ and $N\in\mathbb{N}$. Moreover, let $g:Y\to Z$ be a map that is Lipschitz continuous on a subset $Y_0\subseteq Y$ with Lipschitz constant $L_0\in\R_+$. Then
\begin{align*}
    \dd_{Z,1} \big(\P^P \circ g(\widehat{T}_N)^{-1}, \P^Q\circ g(\widehat{T}_N)^{-1}\big) \,\le\, L_0L_1\,\dd_{X,2}(P,Q)
\end{align*}
for all $P,Q\in\mathscr{P}_2(X)$ with $\widehat T_N\in Y_0$ $\P^P$-a.s.\ and $\P^Q$-a.s., for any $N\in\mathbb{N}$.
\end{proposition}


\section{Optimization problem underlying the estimator \texorpdfstring{$\widehat S_N$}{S\_N}}\label{Sec: OputeSN}

In this section, we formally formulate and analyze the optimization problem underlying the estimator $\widehat S_N$ defined in (\ref{eq:L1-MLE}).
Theorem \ref{thm-glb-Lip-spars-precin-mtx} is specifically needed to obtain distributional stability of $\widehat S_N$ in Theorem \ref{thm-SR-Sprs-Preci-mtx}.


In Section \ref{Sec:Faeoaum}, we formulate the optimization problem (see (\ref{eq:L-PrecisionM-spse})) and show that it has a unique minimizer, denoted by $S^*(\lambda,\Sigma)$, and that its feasible set can be even chosen smaller than $\PD$ without affecting its optimal solution. In Sections \ref{Sec:Sotosbotgc} and \ref{Sec:Lcotosboift}, we analyze the mapping $\PSD\to\PD$, $\Sigma\mapsto S^*(\lambda,\Sigma)$ for continuity and Lipschitz continuity, respectively. The proof of continuity is based on the growth condition (\ref{eq:L-growth-in-S}). 
The approach to prove Lipschitz continuity is to write down the first-order optimality condition at the optimal solution and then use an implicit function theorem. 
The main difference between the two approaches is that the former does not require continuous differentiability of $S\mapsto L(\lambda,\Sigma,S)$, whereas the latter does. We propose a smoothing approach to get around the non-smoothness issue in the second approach, and this is indeed the main contribution of this section.

\subsection{Formulation and existence of a unique minimizer}\label{Sec:Faeoaum}

Recall that the spaces $\PSD$ and $\PD$ are equipped with the Frobenius norm $\|\cdot\|$ and define a mapping $L:\R_{++}\times\PSD\times\PD\to\R$ by
\begin{align*}
    L(\lambda,\Sigma,S):=\la \Sigma, S\ra - \log(\det S) + \lambda \|S\|_{1},
\end{align*}
where as before $\|S\|_{1}:=\sum_{1\leq i,j\leq n} |S_{i,j}|$. For any fixed $\lambda\in\R_{++}$ and $\Sigma\in\PSD$, consider the minimization problem
\begin{align}\label{eq:L-PrecisionM-spse}
    \min_{S\in\PD}{L(\lambda,\Sigma,S)}. 
\end{align}
If the second argument of $L$, i.e.\ $\Sigma$, is chosen to be $\widehat\Sigma_N$, then a minimizer of (\ref{eq:L-PrecisionM-spse}) is just $\widehat S_N$. The following proposition ensures that (\ref{eq:L-PrecisionM-spse}) possesses a unique minimizer.

\begin{proposition}\label{prop:boundns-of-sln}
For any fixed $\lambda\in\R_{++}$ and $\Sigma\in\PSD$, the mapping $\PD\to\R$, $S\mapsto L(\lambda,\Sigma,S)$ {is continuous} and strictly convex, {and it} has a unique minimizer, denoted by $S^*(\lambda,\Sigma)$.  
{
Moreover, $\sup_{\Sigma\in\PSD}\|S^*(\lambda,\Sigma)\|\le n/\lambda$ for any $\lambda\in\R_{++}$.} 
\end{proposition}

{
The latter statement in the preceding proposition has the following implication. 
For any $\lambda\in\R_{++}$, let $C_\lambda>n/\lambda$ be (arbitrarily close to $n/\lambda$ but) fixed. Then
\begin{align}
    S^*(\lambda,\Sigma)\in \inmat{int}\; {\cal S}_\lambda
    \quad{\mbox{ for all }\Sigma\in\S_+^n},
    \label{eq:Int-optimal-sln}
\end{align}
where 
\begin{align}
    \label{eq:S-feasible-set-bnd}
    \mathcal{S}_\lambda := \big\{S\in\PD:\,\|S\| \leq C_\lambda\big\}.
\end{align}
Here and later on, 
``int'' denotes the interior of a set. Consequently,
\eqref{eq:L-PrecisionM-spse} can be equivalently written as
\begin{align}\label{problem-with-compact-feasible-set}
    \min_{S\in\mathcal{S}_\lambda}\, {L(\lambda,\Sigma,S)} 
\end{align}
for any $\Sigma\in\PSD$ without affecting its optimal solution.}


\subsection{Continuity of the minimizer in $\Sigma$}\label{Sec:Sotosbotgc}

The following proposition shows that the mapping $\PSD\to\PD$, $S\mapsto L(\lambda,\Sigma,S)$ satisfies a certain growth condition at the minimizer $S^*(\lambda,\Sigma)$.

\begin{proposition}
\label{prop-L-growth-cvn-mtx}
{For any $\Sigma\in\PSD$ and $\lambda\in\R_{++}$, there exist $\alpha_{\lambda,\Sigma},\beta_{\lambda,\Sigma}\in\R_{++}$ (depending on $\lambda$ and $\Sigma$)} such that 
\begin{align}\label{eq:L-growth-in-S}
    L(\lambda,\Sigma, S) \geq L\big(\lambda,\Sigma, S^*(\lambda,\Sigma)\big) + \phi_{\lambda,\Sigma}\big(\|S-S^*(\lambda,\Sigma)\|\big)\quad\mbox{ for all $S\in\PD$ }
\end{align}
for the function $\phi_{\lambda,\Sigma}:\R_+\to\R_+$ defined by $\phi_{\lambda,\Sigma}(x):=\min\{\alpha_{\lambda,\Sigma} x^2, \beta_{\lambda,\Sigma} x\}$. 
\end{proposition}


Using the growth condition (\ref{eq:L-growth-in-S}), we can derive the following result. 

\begin{theorem}
\label{thm-prci-mtx-Lip-hld-cnty}
Let $\lambda\in\R_{++}$ and $\Sigma\in\PSD$. Moreover, let $C_\lambda$ and $\alpha_{\lambda,\Sigma},\beta_{\lambda,\Sigma}$ be as in \eqref{eq:S-feasible-set-bnd} 
 and Proposition \ref{prop-L-growth-cvn-mtx}, respectively. Then 
\begin{align}\label{thm-prci-mtx-Lip-hld-cnty-EQUATION}
    & \big\|S^*(\lambda,\Sigma)-S^*(\lambda,\Sigma')\big\|\nonumber\\
    & \,\leq\, \max\Big\{\Big(\frac{3C_\lambda}{\alpha_{\lambda,\Sigma}} \|\Sigma-\Sigma'\|\Big)^{1/2},\frac{3C_\lambda}{\beta_{\lambda,\Sigma}} \|\Sigma-\Sigma'\|\Big\} \quad\mbox{ for all $\Sigma'\in \PSD$}.
\end{align}
\end{theorem}


Theorem \ref{thm-prci-mtx-Lip-hld-cnty} shows that the mapping $\PSD\to\PD$, $\Sigma\mapsto S^*(\lambda,\Sigma)$ is continuous. However, for the proof of our main result, Theorem \ref{thm-SR-Sprs-Preci-mtx}, we need (global) Lipschitz continuity. Therefore, in the next subsection we will use an implicit function theorem to derive (global) Lipschitz continuity (see Theorem \ref{thm-glb-Lip-spars-precin-mtx}). It is worth noting that the proof of Theorem \ref{thm-glb-Lip-spars-precin-mtx} uses Theorem \ref{thm-prci-mtx-Lip-hld-cnty}, see {Step 2 and Lemma \ref{sec:M-continuity} in Section \ref{Sec: Proof of Proof of thm-glb-Lip-spars-precin-mtx}.}
\subsection{Lipschitz continuity of the minimizer in $\Sigma$}\label{Sec:Lcotosboift}

Let us assume that the constant $C_\lambda$ in (\ref{eq:S-feasible-set-bnd}) is chosen so large that (\ref{eq:Int-optimal-sln}) holds. Then we can write down the first order condition on the minimizer of (\ref{problem-with-compact-feasible-set}) without the effect of the constraint $\|S\|\leq C_\lambda$ as
\bgeqn
    \bm 0\in \Sigma - S^{-1} + \lambda\,\partial\|S\|_{1},
\label{eq:first-order-optim-cnd}
\edeqn
where $\partial\|S\|_{1}$ is the Clarke subdifferential of $\|\cdot\|_1$ at point $S$. The latter is defined by $\partial\|S\|_1:= \mathrm{conv}\{\limsup_{S_k\to S} \nabla \|S_k\|_1:\nabla \|S_k\|_1{\inmat{ exists for all }k\in\mathbb{N}}\}$, where $\mathrm{conv}{A}$ stands for the convex hull of a set $A$ (see Section 2.1 in \cite{clarke1990optimization}) {and $\nabla \|S_k\|_1:=(\frac{\partial}{\partial S_{i,j}}\|(S_{k,\ell})_{1\le k,\ell\le n}\|_1)_{1\le i,j\le n}$.}   
For any $\Sigma\in \PSD$, the minimizer $S^*(\lambda,\Sigma)$ of (\ref{problem-with-compact-feasible-set}) satisfies (\ref{eq:first-order-optim-cnd}). However, the right-hand side of (\ref{eq:first-order-optim-cnd}) is a set-valued mapping in $S$. Therefore, we cannot {directly use an implicit} function theorem to derive Lipschitz continuity of the minimizer $S^*(\lambda,\Sigma)$ in $\Sigma$. To circumvent the difficulty, we propose to smooth the function $S\mapsto\|S\|_1$ by applying the simple smoothing function $h_\varepsilon(x):=\sqrt{x^2+\varepsilon}$ to each summand $|S_{i,j}|$ in the representation $\|S\|_1=\sum_{1\le i,j\le n}|S_{i,j}|$, where $\varepsilon\in\R_{++}$. The resulting analogue of (\ref{problem-with-compact-feasible-set}) is
\begin{align}\label{smoothObjective}
    \min_{S\in\mathcal{S}_\lambda} L_\varepsilon(\lambda,\Sigma,S),
\end{align}
where {${\cal S}_\lambda$ is as in (\ref{eq:S-feasible-set-bnd}) and} the mapping $L_\varepsilon:\R_{++}\times\PSD\times\PD\to\R$ is defined by
\begin{align*}
    L_\varepsilon(\lambda,\Sigma,S):=\la \Sigma, S\ra - \log(\det S) + \lambda H_\varepsilon(S)
\end{align*}
with $H_\varepsilon(S) = \sum_{1\le i,j\le n} h_\varepsilon(S_{i,j})$. { Note that by the definition of $h_\varepsilon$, $H_\varepsilon$ is also convex.}
The following Proposition \ref{prop-smth-appx-prci-mtx} states
that~\eqref{smoothObjective} has a unique minimizer and that this minimizer converges to $S^*(\lambda,\Sigma)$ as $\varepsilon\downarrow 0$. The proposition uses the function $R_{\lambda,\Sigma}:\R_+\to\R$ defined as follows:
\begin{align}\label{prop-smth-appx-prci-mtx-lemma-eq}
    R_{\lambda,\Sigma}(\delta):={\inf_{S\in\PD:\,\|S-S^*(\lambda,\Sigma)\|\geq \delta}} L(\lambda,\Sigma,S) - L(\lambda,\Sigma, S^*(\lambda,\Sigma)).
\end{align}
This function should be seen as a growth function, which quantifies the increment of the value of $L(\lambda,\Sigma,S)$ as $S$ deviates from $S^*(\lambda,\Sigma)$ by at least $\delta$ under the norm distance $\|\cdot\|$. For any fixed $\lambda\in\R_{++}$ and $\Sigma\in\PSD$, this function is non-decreasing { and continuous} over $\R_{+}$ and satisfies $R_{\lambda,\Sigma}(0)=0$, $R_{\lambda,\Sigma}(\delta)>0$ for all $\delta\in\R_{++}$ and $\lim_{\delta\uparrow\infty}R_{\lambda,\Sigma}(\delta)=\infty$ 
(see Appendix \ref{prop-smth-appx-prci-mtx-lemma-sec-1}). Together with (\ref{ineq:continuity-of-S_eps}), these properties justify the last statement in the following proposition.

\begin{proposition}\label{prop-smth-appx-prci-mtx}
For any fixed $\lambda\in\R_{++}$ and $\Sigma\in\PSD$, the minimization problem~\eqref{smoothObjective} {has a unique minimizer}, denoted by $S^*_\varepsilon(\lambda,\Sigma)$. Moreover, we have
\begin{align}\label{ineq:continuity-of-S_eps}
    \big\|S^*(\lambda,\Sigma)-S_\varepsilon^*(\lambda,\Sigma)\big\|\leq R_{\lambda,\Sigma}^{-1}\big({3\lambda n^2\sqrt{\varepsilon}}\big)\quad\mbox{ for all }\varepsilon\in\R_{++}, 
\end{align}
where $R_{\lambda,\Sigma}^{-1}(t):=\inf\{\delta\in\R_+: R_{\lambda,\Sigma}(\delta) \geq  t\}$ 
for the function $R_{\lambda,\Sigma}:\R_+\to\R$ defined by (\ref{prop-smth-appx-prci-mtx-lemma-eq}). {
In particular, $\lim_{\varepsilon\downarrow 0} S_\varepsilon^*(\lambda,\Sigma) = S^*(\lambda,\Sigma)$.
}
\end{proposition}

Since $S^*(\lambda,\Sigma)$ lies in the interior of $\mathcal{S}_\lambda$, Proposition \ref{prop-smth-appx-prci-mtx} ensures that we can choose $\varepsilon\in\R_{++}$ so small that $S_\varepsilon^*(\lambda,\Sigma)$ also lies in the interior of $\mathcal{S}_\lambda$. Let us assume that $\varepsilon$ is chosen in this way. Then $S_\varepsilon^*(\lambda,\Sigma)$ satisfies the following first order optimality condition:
\bgeqn
    \bm 0 =\Sigma - S^{-1} + \lambda \nabla H_\varepsilon(S),
    \label{eq:G_ep=0}
\edeqn
{where $\nabla H_\varepsilon(S):=(\frac{\partial}{\partial S_{i,j}}H_\varepsilon((S_{k,\ell})_{1\le k,\ell\le n}))_{1\le i,j\le n}$.} 

In Appendix \ref{Sec: Proof of Proof of thm-glb-Lip-spars-precin-mtx} we show that the mapping $\mathcal{S}_\lambda\to\S^n$, $S\mapsto G_\varepsilon(\lambda,\Sigma,S) := \Sigma - S^{-1} + \lambda \nabla H_\varepsilon(S)$ induced by \eqref{eq:G_ep=0} 
is a strongly monotone operator (see Section \ref{Sec:ift} for the precise meaning). Therefore, we will be able to use the implicit function theorem in the form of Theorem \ref{thm-implt-fnct-thm} to derive the following result,
%
%
%
%
which shows that the mapping $\PSD\to\PD$, $\Sigma\mapsto S^*(\lambda,\Sigma)$ is globally Lipschitz continuous. Recall that we use $S^*(\lambda,\Sigma)$ to denote the minimizer of \eqref{eq:L-PrecisionM-spse} or, equivalently, of \eqref{problem-with-compact-feasible-set}.

\begin{theorem}\label{thm-glb-Lip-spars-precin-mtx}
For any $\lambda\in\R_{++}$ 
and $\kappa>(n/\lambda)^2$,
the following inequality holds:
\bgeqn
    \big\|S^*(\lambda,\Sigma_1) - S^*(\lambda,\Sigma_2)\big\| \leq 
    \kappa\|\Sigma_1-\Sigma_2\|\quad\mbox{ for all }\Sigma_1,\Sigma_2\in\PSD. \label{eq:S^*sigma-Lip}
\edeqn
\end{theorem}



\section{Distributional stability of \texorpdfstring{$\widehat S_N$}{S\_N} and of \texorpdfstring{$\widehat\Sigma_N$}{Sigma\_N} and its eigenvalues}\label{Sroeocmapm}

In this section, we show that the sample covariance matrix $\widehat\Sigma_N$, the eigenvalues of $\widehat\Sigma_N$, and the sparse precision matrix estimator $\widehat S_N$ are distributionally stable.

Recall that the sample covariance matrix $\widehat\Sigma_N$, the simplified sample covariance matrix $\widetilde\Sigma_N$ and the sparse estimators $\widehat S_N$ and $\widetilde S_N$ of the precision matrix are defined by
\begin{subequations}
\begin{align}
    \widehat\Sigma_N(\bm{x}) & \,:=\, \frac{1}{N}\sum_{i=1}^N x^i{x^i}^\trans - \Big(\frac{1}{N}\sum_{i=1}^N x^i\Big)\Big(\frac{1}{N}\sum_{i=1}^N x^i\Big)^\trans,\quad \widetilde\Sigma_N(\bm{x}) \,:=\, \frac{1}{N}\sum_{i=1}^N x^i{x^i}^\trans,
    \\
    \widehat S_N(\bm{x}) & \,:=\, \arg\min_{S\in\PD}\big(\la\widehat{\Sigma}_N(\bm{x}), S \ra - \log(\det S) + \lambda \|S\|_{1}\big),
    \label{eq:sampl-precision-mtx}
    \\
    \widetilde{S}_N(\bm x) & \,:=\,
    \arg\min_{S\in\PD}\big(\la\widetilde{\Sigma}_N(\bm{x}), S \ra - \log(\det S) + \lambda \|S\|_{1}\big)\label{eq:sampl-precision-mtx-tilde}
\end{align}
\end{subequations}
for all $\bm{x}=(x^1,\ldots,x^N)\in\Omega$, where $\Omega:=(\R^n)^N$ and $\lambda\in\R_{++}$ is a fixed constant. Proposition \ref{prop:boundns-of-sln} ensures that $\widehat S_N(\bm{x})$ and $\widetilde{S}_N(\bm x)$ are well defined for any $\bm{x}\in\Omega$.
Recall that $\boldsymbol{\xi}=(\xi^1,\ldots,\xi^N)$ is used to denote the identity on $\Omega$. The sample covariance matrix $\widehat\Sigma_N=\widehat\Sigma_N(\boldsymbol{\xi})$ is the standard nonparametric estimation of the covariance matrix $\Sigma$ of $\xi^1$ when $\xi^1,\ldots,\xi^n$ are i.i.d.\ (i.e.\ regarded as random variables under a product measure $\P^P=P^{\otimes N}$ on $(\Omega,{\cal F})$). In the case where it is a priori known that the mean of $\xi^1$ is equal to ${\bm 0}$, then it is more reasonable to use $\widetilde\Sigma_N=\widetilde\Sigma_N(\boldsymbol{\xi})$ as an estimator of the covariance matrix $\Sigma$ of $\xi^1$. The background of the estimator $\widehat S_N=\widehat S_N(\boldsymbol{\xi})$ of the precision matrix $\Sigma^{-1}$ of $\xi^1$ 
was already discussed in Section \ref{Sec:Introduction}.

Recall from Section \ref{Sec:notation} that $\dd_{X,p}$ is used to denote the $p$-th Fortet-Mourier metric on a linear space $X$. In the following theorem, the role of $X$ is played by $\PSD$ and $\R^n$. To simplify the presentation, we will write $\dd_{\rm K}$ and $\dd_2$ instead of $\dd_{\PSD,1}$ and $\dd_{\R^n,2}$, respectively (the subscript ${\rm K}$ refers to Kantorovich). As before (see Section \ref{Sec:Acfsidogse}), $m_P$ and $m_Q$ denote the first absolute moments of $P$ and $Q$, respectively. \textcolor{black}{Moreover, we will use $\Sigma_P$ to denote the covariance matrix of $\xi^1$ under $P$.} 

%

%

{

\begin{theorem}
\label{thm-SR-Covn-mtx}
For any $P,Q\in\mathscr{P}_2(\R^n)$ and $N\in\mathbb{N}$, 
the following inequalities hold:
\begin{subequations}
\begin{align}
    \dd_{\rm K}\big(\P^P\circ \widehat\Sigma_N^{\,-1},\P^Q\circ\widehat\Sigma_N^{\,-1}\big) & \,\le\, \max\{3,2m_P,2m_Q\}\,\dd_{2}(P,Q),\label{robustnessofT}\\
    \dd_{\rm K} \big(\P^P\circ \widetilde\Sigma_N^{\,-1},\P^Q\circ\widetilde\Sigma_N^{\,-1}\big) & \,\le\, 2\,\dd_{2}(P,Q),\label{robustnessofT-1}\\
    \dd_{\rm K}\big(\P^P\circ \widehat\Sigma_N^{\,-1},\delta_{\Sigma_P}\big)  & \,\le\, \E^P\big[\|\widehat\Sigma_N-\Sigma_P\|\big],\label{consistencyofT}\\
    \dd_{\rm K}\big(\P^P\circ \widetilde\Sigma_N^{\,-1},\delta_{\Sigma_P}\big) & \,\le\, \E^P\big[\|\widetilde\Sigma_N-\Sigma_P\|\big].\label{consistencyofT-1}
\end{align}
\end{subequations}
\end{theorem}


In the following theorem we extend (\ref{robustnessofT}) and (\ref{consistencyofT}), where $\widehat\mu_N$ is used to denote the standard nonparametric estimator of the mean vector, i.e.\ $\widehat\mu_N(\bm x):=\frac{1}{N}\sum_{i=1}^Nx^i$ for all $\bm{x}=(x^1,\ldots,x^N)\in\Omega$. By $\dd_{\rm K}$ we will mean the Fortet-Mourier metric $\dd_{1,\R^n\times\PSD}$ of order $1$, where the Cartesian product $\R^n\times\PSD$ is equipped with the norm $\|(\mu,\Sigma)\|_1:=\|\mu\|+\|\Sigma\|$, 
and $\mu_P$ is used to denote the mean vector of $\xi^1$ under $P$. 

\begin{theorem}\label{thm-SR-Covn-mtx-with-mean}
For any $P,Q\in\mathscr{P}_2(\R^n)$ and $N\in\mathbb{N}$, the following inequalities hold:
\begin{subequations}
\begin{align}
    & \dd_{\rm K}\big(\P^P\circ (\widehat\mu_N,\widehat\Sigma_N)^{-1},\P^Q\circ(\widehat\mu_N,\widehat\Sigma_N)^{-1}\big)  \,\le\, \max\{4,2m_P,2m_Q\}\,\dd_{2}(P,Q),\label{robustnessofT-with-mean}\\
    & \dd_{\rm K}\big(\P^P\circ (\widehat\mu_N,\widehat\Sigma_N)^{-1},\delta_{(\mu_P,\Sigma_P)}\big)  \,\le\, \E^P\big[\|\widehat\mu_N-\mu_P\|\big]+\E^P\big[\|\widehat\Sigma_N-\Sigma_P\|\big].\label{consistencyofT-with-mean}
\end{align}
\end{subequations}
\end{theorem}


In a sense, the following theorem is the main result of this paper. In the theorem, we  write $\dd_{\rm K}$ and $\dd_2$ instead of $\dd_{\PD,1}$ and $\dd_{\R^n,2}$, respectively.

\begin{theorem}\label{thm-SR-Sprs-Preci-mtx}
For any $P,Q\in\mathscr{P}_2(\R^n)$, $N\in\mathbb{N}$ 
and $\kappa>(n/\lambda)^2$, the following inequalities hold:
\begin{subequations}
\begin{align}
    \dd_{\rm K}\big(\P^P\circ\widehat S_N^{-1},\P^Q\circ\widehat S_N^{-1}\big)
    & \,\le\, \kappa\,\max\{3,2m_P,2m_Q\}\, \dd_2(P,Q),\label{eq:statistical-robustness-precision-matrix-1}\\
    \dd_{\rm K}\big(\P^P\circ\widetilde S_N^{-1},\P^Q\circ\widetilde S_N^{-1}\big)
    & \,\le\, 2\kappa\, \dd_2(P,Q),\label{eq:statistical-robustness-precision-matrix-2}
\end{align}
\end{subequations}
{For any $P\in\mathscr{P}_2(\R^n)$, for which $S_P:=\Sigma_P^{-1}$ exists, and for any $N\in\mathbb{N}$ and $\kappa>(n/\lambda)^2$, the following inequalities hold:} 
\begin{subequations}
\begin{align}
    \dd_{\rm K}\big(\P^P\circ \widehat S_N^{-1},\delta_{S_P}\big)  & \,\le\, \kappa\,\E^P\big[\|\widehat\Sigma_N-\Sigma_P\|\big],\label{eq:consistency-precision-matrix-1}\\
    \dd_{\rm K}\big(\P^P\circ \widetilde S_N^{-1},\delta_{S_P}\big) & \,\le\, \kappa\,\E^P\big[\|\widetilde\Sigma_N-\Sigma_P\|\big].\label{eq:consistency-precision-matrix-2}
\end{align}
\end{subequations}
\end{theorem}


The following proposition tells us how quickly the right-hand sides in (\ref{consistencyofT}), (\ref{consistencyofT-1}), {(\ref{consistencyofT-with-mean}),} (\ref{eq:consistency-precision-matrix-1}), (\ref{eq:consistency-precision-matrix-2}) {(and \eqref{robustnessofeigens-3}, \eqref{robustnessofeigens-4})} converge to $0$ as $N\to\infty$.


\begin{proposition}\label{thm-SR-Covn-mtx:REMARK}
Let $r\in[1,2)$. Then 
\begin{subequations}
\begin{align}
    \lim_{N\to\infty}N^{(r-1)/r}\,\E^P\big[\|\widehat\mu_N-\mu_P\|_1\big] & \,=\, 0 \quad\mbox{for all }P\in\mathscr{P}_{r}(\R^n),\label{thm-SR-Covn-mtx:REMARK-EQ1}\\
    \lim_{N\to\infty}N^{(r-1)/r}\,\E^P\big[\|\widehat\Sigma_N-\Sigma_P\|_1\big] & \,=\, 0 \quad\mbox{for all }P\in\mathscr{P}_{2r}(\R^n),\label{thm-SR-Covn-mtx:REMARK-EQ2}\\
    \lim_{N\to\infty}N^{(r-1)/r}\,\E^P\big[\|\widetilde\Sigma_N-\Sigma_P\|_1\big] & \,=\, 0 \quad\mbox{for all }P\in\mathscr{P}_{2r}(\R^n)\mbox{ with }\mu_P={\bm 0}.\label{thm-SR-Covn-mtx:REMARK-EQ3}
\end{align}
\end{subequations}
In particular, the same statements hold true if $\|\cdot\|_1$ is replaced by any other matrix norm (such as the Frobenius norm $\|\cdot\|$). 
\end{proposition} 




Let us use $\lambda_1(\widehat\Sigma_N(\bm{x})),\ldots,\lambda_n(\widehat\Sigma_N(\bm{x}))$ to denote the eigenvalues of $\widehat\Sigma_N(\bm{x})$ in decreasing order (i.e.\ $\lambda_1(\widehat\Sigma_N(\bm{x}))\ge\cdots\ge\lambda_n(\widehat\Sigma_N(\bm{x}))$) for all $\bm{x}=(x^1,\ldots,x^N)\in\Omega$. For the sake of notational simplicity, we set
\begin{align*}
    \widehat\lambda_{1,N}(\bm{x}):=\lambda_1\big(\widehat\Sigma_N(\bm{x})\big),~\ldots,~\widehat\lambda_{n,N}(\bm{x}):=\lambda_N\big(\widehat\Sigma_N(\bm{x})\big)
\end{align*}
for all $\bm{x}=(x^1,\ldots,x^N)\in\Omega$. Thus, $\widehat\lambda_{1,N},\ldots,\widehat\lambda_{n,N}$ are the standard nonparametric estimators of the eigenvalues $\lambda_1,\ldots,\lambda_n$ of the true covariance matrix $\Sigma$ in decreasing order (i.e.\ $\lambda_1\ge\cdots\ge\lambda_n$). Moreover, let the estimators $\widetilde\lambda_{1,N},\ldots,\widetilde\lambda_{n,N}$ be defined in the same way, but with $\widehat\Sigma_N$ replaced by $\widetilde\Sigma_N$. In the following theorem, $\dd_{\rm K}$ is use to denote the first order Fortet-Mourier metric {$\dd_{\R^n,1}$ on $\mathscr{P}_1(\R^n)$,} whereas $\dd_2$ is as before the second order Fortet-Mourier metric $\dd_{\R^n,2}$ on $\mathscr{P}_2(\R^n)$. Moreover, we will use $\lambda_1^P,\ldots,\lambda_n^P$ to denote the eigenvalues of the covariance matrix $\Sigma$ of $\xi^1$ under $P$, i.e.\ of $\Sigma_P$, in decreasing order (i.e.\ $\lambda_1^P\ge\cdots\ge\lambda_n^P$).

{
\begin{theorem}\label{Statistical-robustness-of-eigenvalue}
For any $P,Q\in\mathscr{P}_2(\R^n)$ and $N\in\mathbb{N}$, the following inequalities hold:
\begin{subequations}
\begin{align}
    & \dd_{\rm K}\big(\P^P \circ (\widehat\lambda_{1,N},\ldots,\widehat\lambda_{n,N})^{-1},\P^Q \circ (\widehat\lambda_{1,N},\ldots,\widehat\lambda_{n,N})^{-1}\big)\nonumber\\
    & \qquad\qquad\qquad\qquad\qquad\qquad\qquad\quad\,\,\le\, \max\{3,2m_P,2m_Q\}\,\dd_2(P,Q),\label{robustnessofeigens}\\
    & \dd_{\rm K}\big(\P^P \circ (\widetilde\lambda_{1,N},\ldots,\widetilde\lambda_{n,N})^{-1},\P^Q \circ (\widetilde\lambda_{1,N},\ldots,\widetilde\lambda_{n,N})^{-1}\big)\,\le\, 2\,\dd_2(P,Q),\label{robustnessofeigens-2}\\
    & \dd_{\rm K}\big(\P^P \circ (\widehat\lambda_{1,N},\ldots,\widehat\lambda_{n,N})^{-1},\delta_{(\lambda_{1}^P,\ldots,\lambda_{n}^P)}\big)
    \,\le\, \E^P\big[\|\widehat\Sigma_N-\Sigma_P\|\big]\label{robustnessofeigens-3},\\
    & \dd_{\rm K}\big(\P^P \circ (\widetilde\lambda_{1,N},\ldots,\widetilde\lambda_{n,N})^{-1},\delta_{(\lambda_{1}^P,\ldots,\lambda_{n}^P)}\big)
    \,\le\, \E^P\big[\|\widetilde\Sigma_N-\Sigma_P\|\big].\label{robustnessofeigens-4}
\end{align}
\end{subequations}
\end{theorem}  

\begin{remark}\label{Statistical-robustness-of-eigenvalue-Remark}
In view of Proposition \ref{prop:chain-rule}, Theorem \ref{Statistical-robustness-of-eigenvalue} yields in particular that (\ref{robustnessofeigens}) and (\ref{robustnessofeigens-3}), each with an additional factor $L$ on the right-hand side, are still valid when $(\widehat\lambda_{1,N},\ldots,\widehat\lambda_{n,N})$ and $(\lambda_{1}^P,\ldots,\lambda_{n}^P)$ are replaced by $g(\widehat\lambda_{1,N},\ldots,\widehat\lambda_{n,N})$ and $g(\lambda_{1}^P,\ldots,\lambda_{n}^P)$, respectively, where $g:\R^n\to\R^k$ is a Lipschitz continuous function with Lipschitz constant $L$. Of course, in this case we have to choose $\dd_{\rm K}:=\dd_{1,\R^k}$. An example for $k=1$ and $L=1$ is obtained by setting $g(x_1,\ldots,x_n):=x_i$, which corresponds to the $i$-th largest eigenvalue. An example for $k=1$ and $L=\sqrt{2}$ is obtained by setting $g(x_1,\ldots,x_n):=x_1-x_2$, which corresponds to the spectral gap. The same applies to (\ref{robustnessofeigens-2}) and (\ref{robustnessofeigens-4}).
\end{remark}
}



{
Theorems \ref{thm-SR-Covn-mtx} and \ref{thm-SR-Covn-mtx-with-mean} and Proposition \ref{prop:chain-rule} imply that any estimator that is defined as a globally Lipschitz continuous transformation of the empirical mean vector and/or the empirical covariance matrix is ``distributionally stable''. For example, Theorem \ref{thm-SR-Sprs-Preci-mtx} relies on the global Lipschitz continuity of the mapping $\PSD\to\PD$, $\Sigma\mapsto S^*(\lambda,\Sigma)$, which is established in Theorem \ref{thm-glb-Lip-spars-precin-mtx}. Likewise, Theorem \ref{Statistical-robustness-of-eigenvalue} relies on the global Lipschitz continuity of the eigenvalues, see (\ref{stability-of-eigenvalues}). On the other hand, an estimator that can only be represented as a {\it non}-Lipschitz continuous transformation of $\widehat\mu_N$ and/or $\widehat\Sigma_N$ is typically less stable. For example, the inverse map $\PD\to\PD$, $\Sigma\mapsto\Sigma^{-1}$ is known not to be globally Lipschitz continuous. Therefore, one might expect the inverse $\widehat\Sigma_N^{-1}$ of the empirical covariance matrix $\widehat\Sigma_N$ to be less stable. In Section \ref{Sec-DSotIotSCM} we demonstrate that this does indeed appear to be the case. Nevertheless, the distribution of the inverse of the empirical covariance matrix depends continuously on the underlying distribution $P$, albeit not in a Lipschitz sense. 
} 



}

{

\section{Applications and numerical experiments} \label{sec:app-numerical}
In this section, we conduct four numerical experiments to (i) visualize the theoretical findings of this paper and (ii) showcase how the concept of distributional stability can be applied to practical applications or help us understand the stability of these practical models when data perturbation occurs.}

}


\subsection{Distributional stability of the eigenvalues of the sample covariance matrix}\label{sec:dsoeotscm} 

Inequality (\ref{robustnessofeigens}) (in Theorem \ref{Statistical-robustness-of-eigenvalue}) and Remark \ref{Statistical-robustness-of-eigenvalue-Remark} 
show that the Kantorovich distance $\dd_{\rm K}(\P^P \circ \widehat\lambda_{i,N}^{-1},\P^Q \circ \widehat\lambda_{i,N}^{-1})$ between the distributions of the sample eigenvalues under two different distributions $P$ and $Q$ grows at most linearly as $\dd_2(P,Q)$ deviates from zero, {regardless of the sample size $N$}. We conducted a numerical experiment, which provides support for this finding.


\begin{figure}
    \includegraphics[width=\textwidth]{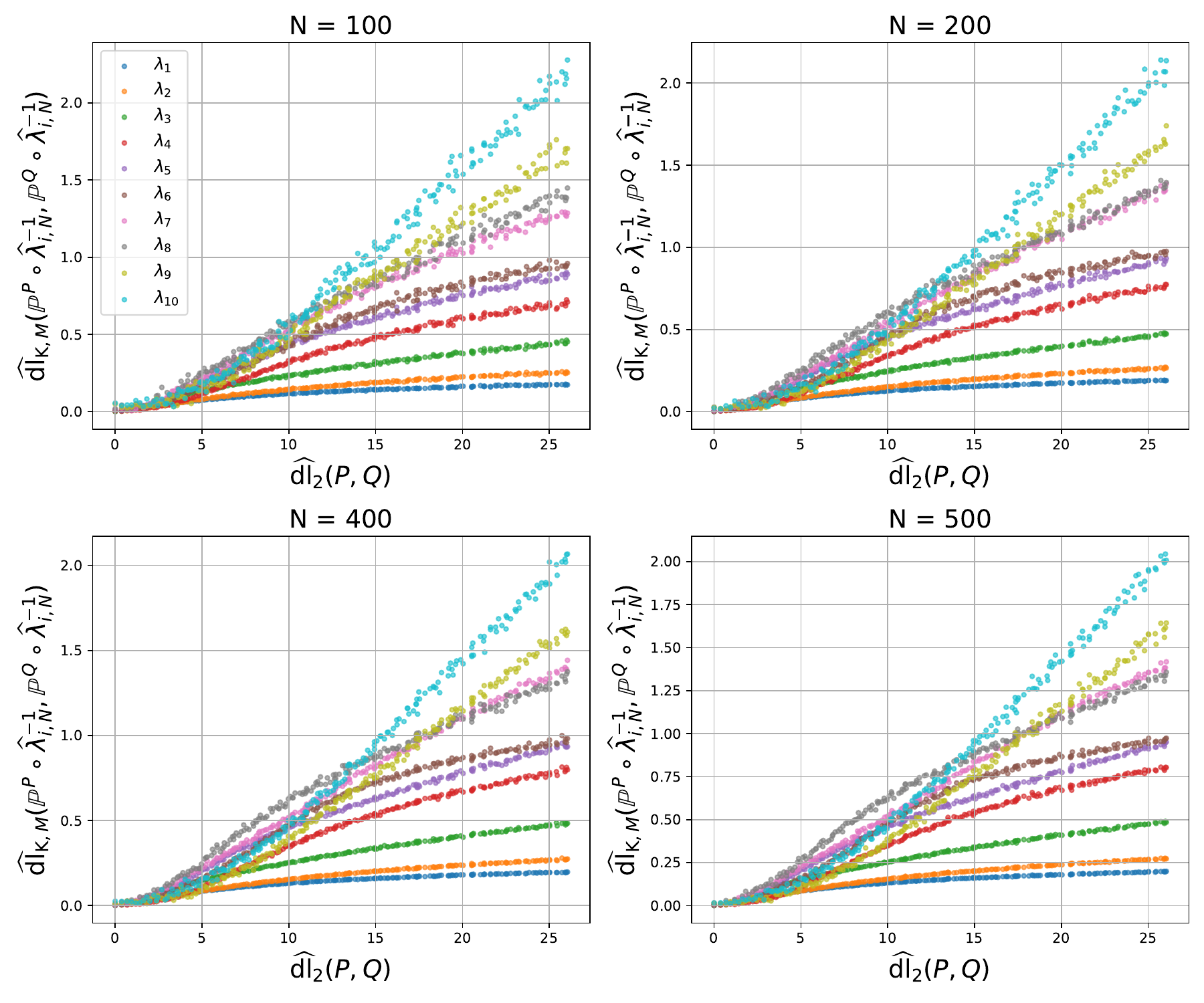}
        \caption{{$\widehat \dd_{{\rm K},M}(\P^P \circ \widehat\lambda_{i,N}^{-1},\P^Q \circ \widehat\lambda_{i,N}^{-1})$ (y-axis) as a function of $\widehat{\dd}_{2}(P,Q)$, $Q\in\mathscr{Q}$ (x-axis), for different sample sizes $N$.}}
        \label{fig:eigenvalue_dKvsdK-500}
\end{figure}

In this experiment, we choose $P:={\rm N}_{\mu,\Sigma}$ and $Q\in\mathscr{Q}:=\{{\rm N}_{\mu+\alpha\mu',\Sigma+\alpha\Sigma'}:\alpha\in[0,1]\}$, where $\mu,\mu'\in\R^n$ and $\Sigma,\Sigma'\in\PSD$ are fixed and ${\rm N}_{\widetilde\mu,\widetilde\Sigma}$ is used to denote the $n$-variate normal distribution with mean vector $\widetilde\mu$ and covariance matrix $\widetilde\Sigma$. To avoid an overly artificial setting, we have left the choice of the parameters $\mu$, $\mu'$, $\Sigma$, $\Sigma'$ to $2n+2$ independent samples $\eta,\eta^1,\ldots,\eta^n,\gamma,\gamma^1,\ldots,\gamma^n$ of an $n$-variate standard normal distribution. Specifically, we chose $n:=10$, and set $\mu:=\eta$, $\mu':=\gamma$, $\Sigma:={\frac{1}{n} }\sum_{i=1}^n\eta^i{\eta^i}^\trans$, $\Sigma':={\frac{1}{n} }\sum_{i=1}^n\gamma^i{\gamma^i}^\trans$.


{
Figure~\ref{fig:eigenvalue_dKvsdK-500} visualizes the dependence of $\dd_{\rm K}(\P^P \circ \widehat\lambda_{i,N}^{-1},\P^Q \circ \widehat\lambda_{i,N}^{-1})$ on $\dd_{2}(P,Q)$, $Q\in\mathscr{Q}$, for samples sizes $N=100$, $200$, $400$, $500$. The results indicate that $\dd_{\rm K}(\P^P \circ \widehat\lambda_{i,N}^{-1},\P^Q \circ \widehat\lambda_{i,N}^{-1})$ depends almost linearly on $\dd_{2}(P,Q)$ for all the cases.}

Since the Kantorovich distance $\dd_{\rm K}(\P^P \circ \widehat\lambda_{i,N}^{-1},\P^Q \circ \widehat\lambda_{i,N}^{-1})$ can hardly be determined explicitly, we approximated it by means of standard Monte Carlo simulations based on $M:=100$ Monte Carlo repetitions. More precisely, we used $\widehat \dd_{{\rm K},M}(\P^P \circ \widehat\lambda_{i,N}^{-1},\P^Q \circ \widehat\lambda_{i,N}^{-1}):=\dd_{\rm K}(\widehat\Lambda_{i,N;M}^{\,P},\widehat\Lambda_{i,N;M}^{\,Q})$ as an approximation of $\dd_{\rm K}(\P^P \circ \widehat\lambda_{i,N}^{-1},\P^Q \circ \widehat\lambda_{i,N}^{-1})$, where $\widehat\Lambda_{i,N;M}^{\,P}$ and $\widehat\Lambda_{i,N;M}^{\,Q}$ denote the empirical distributions of the Monte Carlo samples $\widehat\lambda_{i,N;1},\ldots,\widehat\lambda_{i,N;M}$ under $P$ and $Q$, respectively. Note that the computation of the Kantorovich (Wasserstein) distance between two discrete distributions is pre-implemented in many relevant software packages. {The distance $\dd_{2}(P,Q)$ was approximated by the $\dd_{2}$-distance of appropriately discretized versions of $P$ and $Q$, denoted by $\widehat{\dd}_{2}(P,Q)$, where the computation of $\widehat{\dd}_{2}(P,Q)$ is based on \cite[p.\,732]{heitsch-roemisch-2007}. See Appendix~\ref{sec:app-fm-metric} for details.}

{
\subsection{Distributional sensitivity of the inverse of the sample covariance matrix}\label{Sec-DSotIotSCM}

As already mentioned at the end of Section \ref{Sroeocmapm}, the inverse map $\PD\to\PD$, $\Sigma\mapsto\Sigma^{-1}$ is not globally Lipschitz continuous and therefore we cannot apply Theorem \ref{thm-SR-Covn-mtx} together with Proposition \ref{prop:chain-rule} to obtain distributional stability of the inverse $\widehat\Sigma_N^{-1}$ of the empirical covariance matrix $\widehat\Sigma_N$. Here we investigate the sensitivity of the estimator $\widehat\Sigma_N^{-1}$ to small changes in the underlying distribution $P$ and compare it to the sensitivity of the sparse estimator $\widehat S_N$. The experiment setting is the same as in Section~\ref{sec:dsoeotscm}. Figure~\ref{fig:precision-matrix} visualizes the dependence of $\dd_{\rm K}(\P^P \circ \widehat S_N^{-1},\P^Q \circ \widehat S_N^{-1})$ on $\dd_{2}(P,Q)$, $Q\in\mathscr{Q}$, for different choices of the regularization parameter $\lambda$. The results show that when $\lambda=0$ {(i.e.\ when $\widehat S_N=\widehat\Sigma_N^{-1}$)} the dependence is not Lipschitz and the distribution of the estimator $\widehat\Sigma_N^{-1}$ is {quite} sensitive to {(small)} perturbation in the underlying distribution $P$. Moreover, we can also see that a larger regularization parameter results in a more stable estimator. This result is consistent with Theorem~\ref{thm-SR-Sprs-Preci-mtx}, taking into account that this theorem indicates that a larger $\lambda$ gives rise to a smaller Lipschitz constant $\kappa$, and thus a more distributionally stable estimator.

\begin{figure}
    \includegraphics[width=\textwidth]{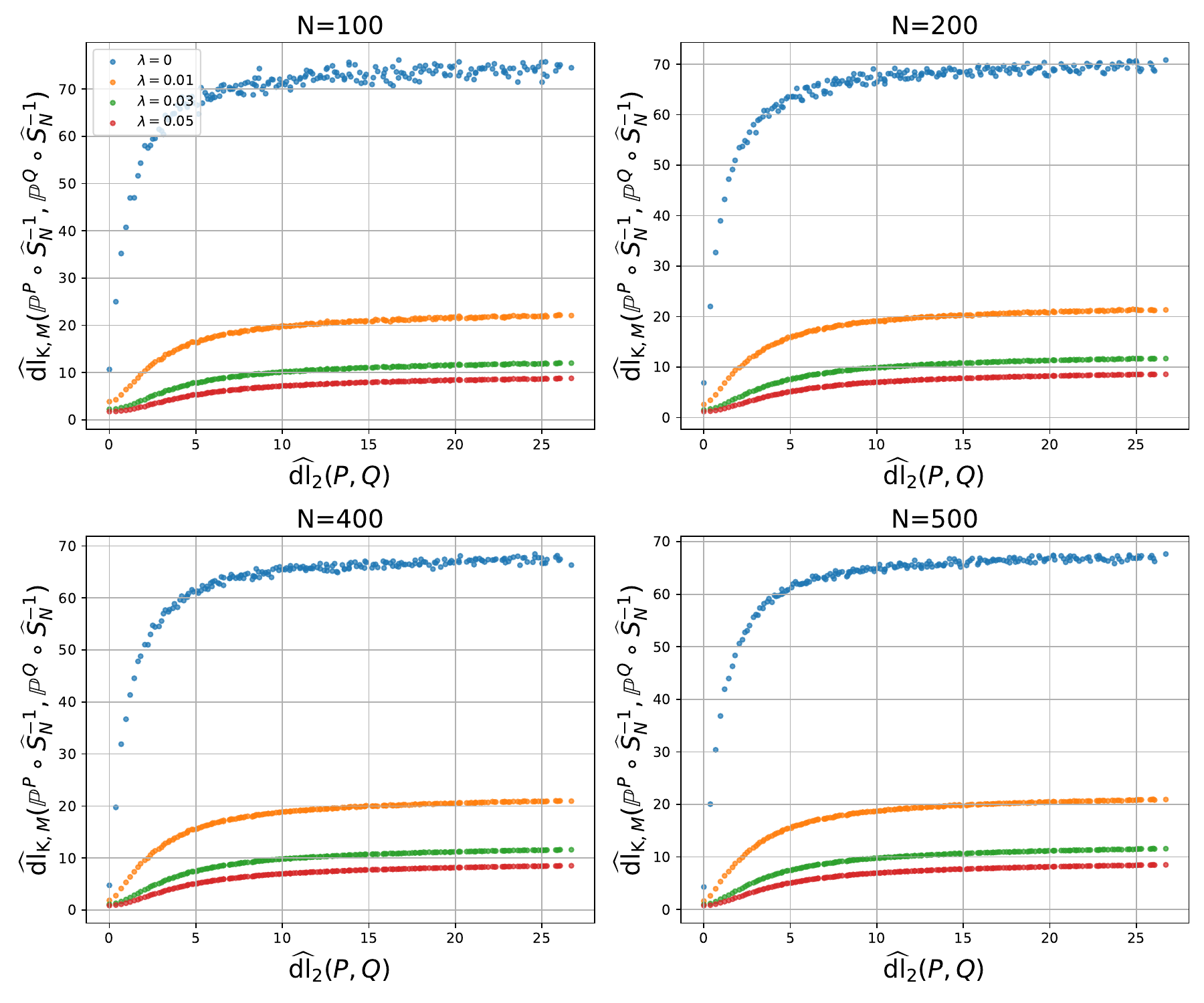}
        \caption{{$\widehat \dd_{{\rm K},M}(\P^P \circ \widehat S_N^{-1},\P^Q \circ \widehat S_N^{-1})$ (y-axis) as a function of $\widehat{\dd}_{2}(P,Q)$, $Q\in\mathscr{Q}$ (x-axis), for different sample size $N$.}}
        \label{fig:precision-matrix}
\end{figure}




%


\subsection{Gaussian graphical model selection and its application in cancer genetic network inference}

\subsubsection{Distributional stability of the sparse estimator in Gaussian graphical models}\label{sec:application-to-gaussian-graphical-model} 

One of the motivations of sparse precision matrix estimation is the {problem of} Gaussian graphical model selection, {see e.g.\ \cite{Venkat2012latent,drton--jmlr-2008,ravikumar2011high,uhler2017ggm,yuan2007model}.} 
Let $\xi$ be an $\R^n$-valued random variable and assume that its distribution is given by a non-degenerate $n$-variate normal distribution {${\mathrm N}_{m,\Sigma}$ with mean vector $m$ and covariance matrix $\Sigma$. Its precision matrix is as before denoted by $S:=\Sigma^{-1}$.}  
In a Gaussian graphical model (GGM), the Gaussian random variable $\xi$ is associated with a 
undirected graph $G=(V,E)$, where $V:=\{1,\ldots,n\}$ is the set of vertices and $E\subseteq\{1,\ldots,n\}^2$ is the set of edges, for which it is assumed that $(i,j)\notin E$ if and only if $S_{i,j} = 0$, i.e.\ $E:=\{(i,j): i,j\in V\mbox{ with }S_{i,j} \neq 0\}$. {It is important to note that $S_{i,j}=0$ if and only if the entries $\xi_i$ and $\xi_j$ are conditionally independent given the remaining entries $(\xi_\ell)_{\ell\in V\setminus\{i,j\}}$ (see e.g.\ Corollary 2.2 in \cite{uhler2017ggm}). That is, a missing edge $(i,j)$ in the graph $G$ corresponds to the conditional independence of $\xi_i$ and $\xi_j$ given $(\xi_\ell)_{\ell\in V\setminus\{i,j\}}$.}

In most practical applications, the true precision matrix $S$ and the set of edges $E$ are unknown and have to be estimated statistically, often assuming that the estimators used are based on i.i.d.\ copies $\xi^1,\ldots,\xi^N$ of $\xi$. In many situations, the graph is assumed to be sparse. In this case, the problem of the GGM selection, i.e.\ determining the edge structure of the graph $G$, reduces to finding the sparsity of the precision matrix $S$ of $\xi$. In this setting, the main concern is whether the entries of the precision matrix estimator correctly recover the \textit{zero}-behavior of its true counterpart. As already mentioned in Section \ref{Sec:Introduction}, the estimator {$\widehat S_N$ introduced in (\ref{eq:sampl-precision-mtx})} is a reasonable choice for the estimator of a sparse precision matrix $S$. As before we consider the nonparametric statistical model $(\Omega,{\cal F},\{\P^P:P\in\mathscr{P}_2(\R^n)\})$ introduced in Sections \ref{Sec:Introduction} and \ref{Sroeocmapm} and assume that $\xi^1,\ldots,\xi^N$ are the coordinate projections on $\Omega=(\R^n)^N$. Recall that $\P^P:=P^{\otimes N}$.

Let us use $E_P$ and $S_P$ to denote the set of edges and the precision matrix, respectively, when the underlying distribution of $\xi$ is given by $P$. To know that the set $E_P$ is well recovered (for large sample size $N$), one needs to make sure that the distribution $\P^P\circ\widehat{S}_N^{-1}$ of $\widehat S_N$ is close to $\delta_{S_P}$ (when $N$ is large). Inequality (\ref{eq:consistency-precision-matrix-2}) and Proposition \ref{thm-SR-Covn-mtx:REMARK} indicate that this is true. For any $P,Q\in\mathscr{P}_2(\R^n)$ and $\kappa>(n/\lambda)^2$,
we also have 
\begin{align}\label{robustness-gauusian-selection}
    & \dd_{\rm K}\big(\P^Q \circ\widehat S_N^{-1},\delta_{S_P}\big)\,\le\, {\kappa\,\max\{3,2m_P,2m_Q\}}\,\dd_2(P,Q)+\kappa\,\E^P\big[\|\widehat\Sigma_N-\Sigma_P\|\big],
\end{align}
which follows from the triangle inequality {and (\ref{eq:statistical-robustness-precision-matrix-1}) and (\ref{eq:consistency-precision-matrix-1}).} Inequality (\ref{robustness-gauusian-selection}) shows that $S_P$ (and thus $E_P$) can be reasonably recovered even if the data are drawn from a ``contaminated'' distribution $Q$ that is slightly different from $P$, i.e., the distance between $P$ and $Q$ is small. By Proposition \ref{thm-SR-Covn-mtx:REMARK}, we know that $\lim_{N\to\infty}N^{(r-1)/r}\,\E^P[\|\widehat\Sigma_N-\Sigma_P\|]=0$ for all $r\in[1,2)$; take into account that $\|\cdot\|\le\|\cdot\|_1$.


{

\subsubsection{Numerical experiment in cancer genetic network inference}

\begin{figure}
    \centering
    \includegraphics[width=\textwidth]{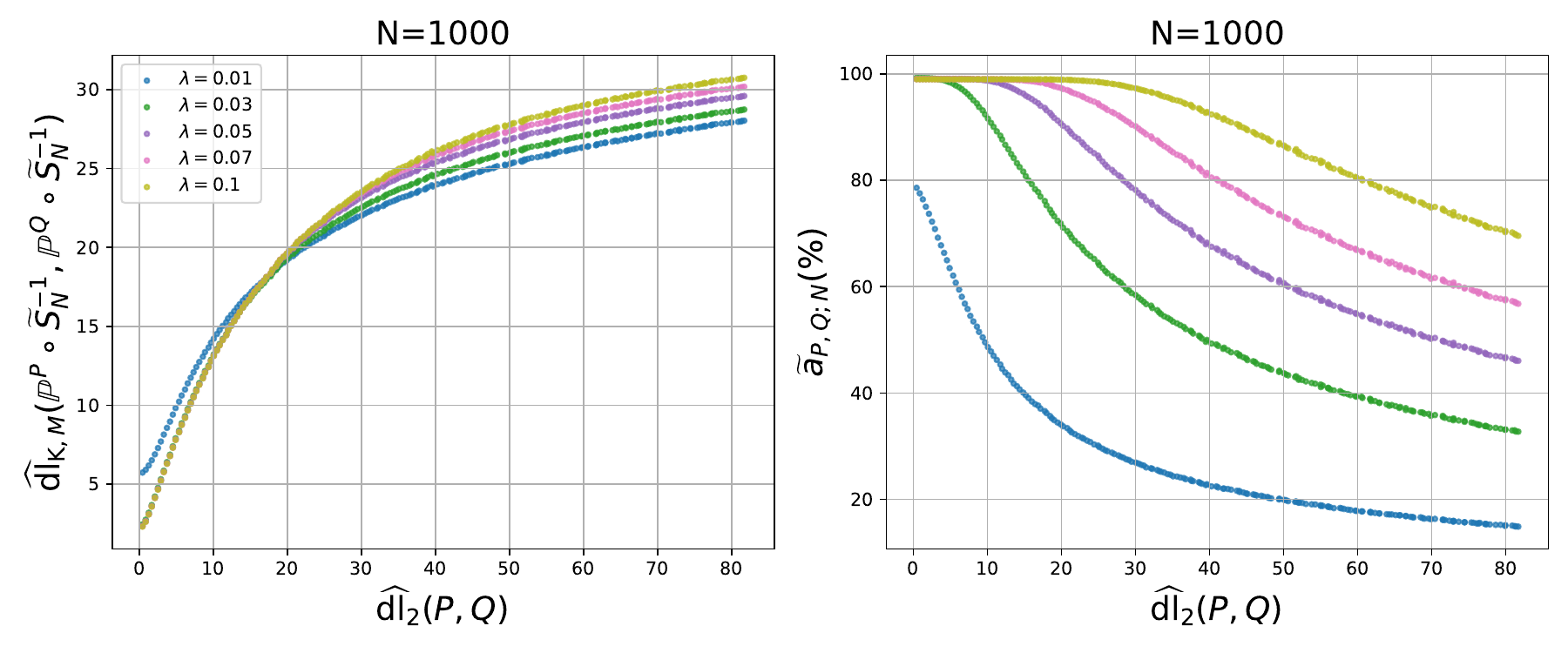}
    \includegraphics[width=\textwidth]{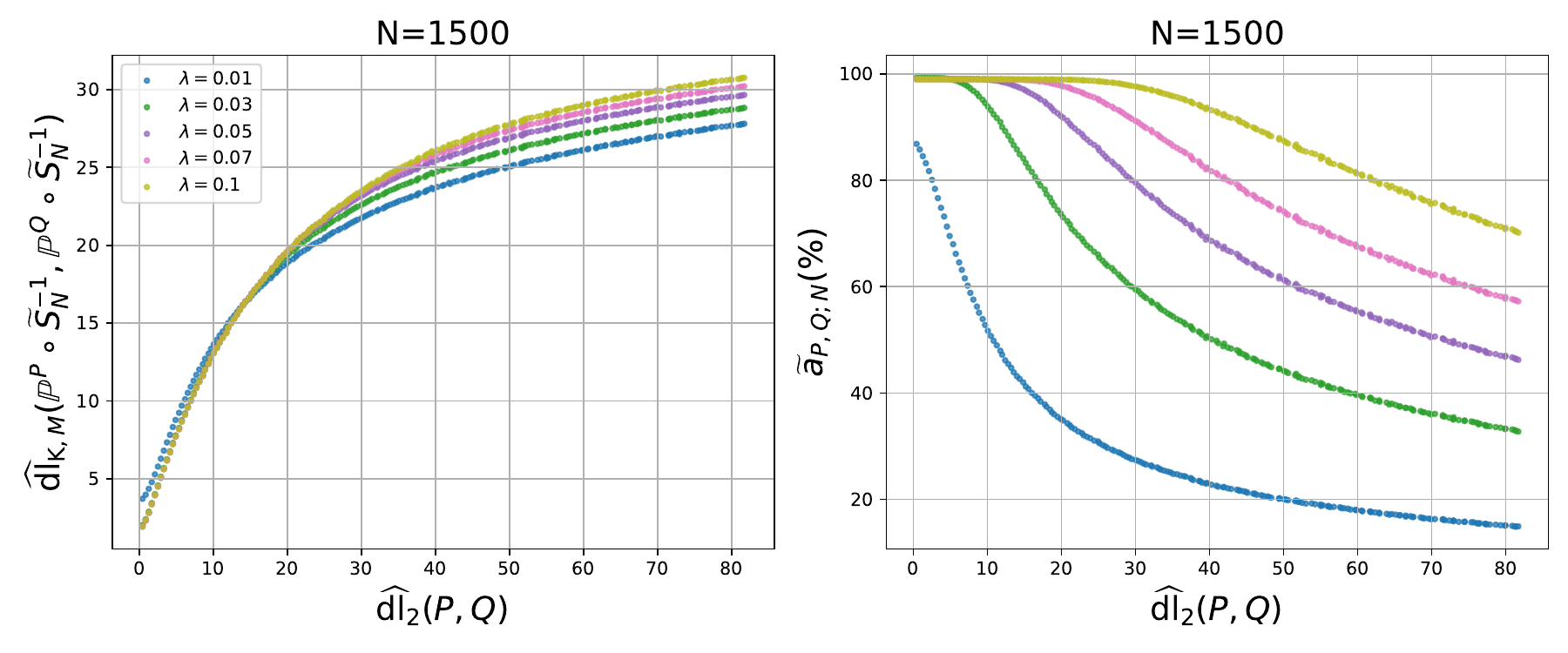}
    \includegraphics[width=\textwidth]{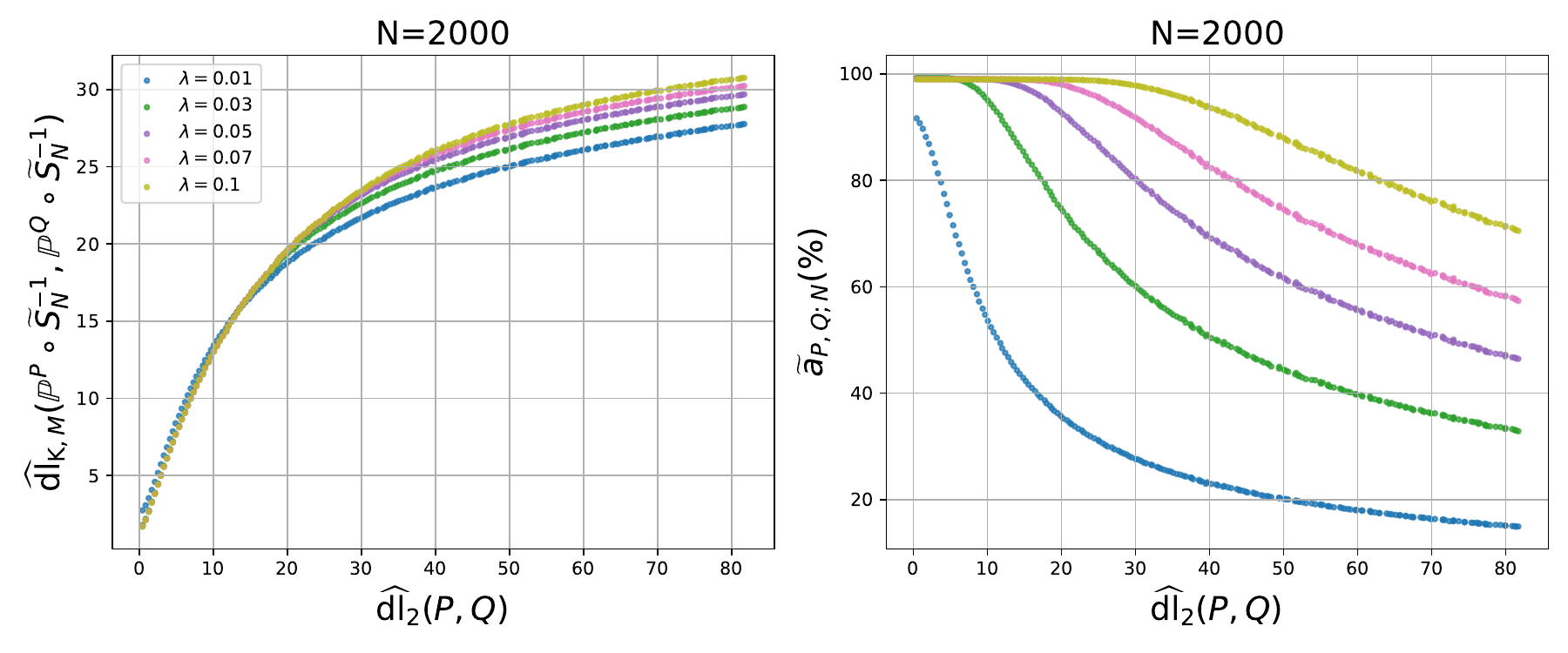}
    \caption{{$\widehat \dd_{{\rm K},M}(\P^P\circ\widetilde S_N^{-1},\P^Q\circ\widetilde S_N^{-1})$  and structure match accuracy (y-axis) as a function of $\widehat{\dd}_{2}(P,Q)$ (x-axis) for different choice of $\lambda$ and different sample sizes $N$.}} 
    \label{fig:cancer_gene}
\end{figure}

In this part, we conduct a numerical experiment to support the theoretical finding in Section~\ref{sec:application-to-gaussian-graphical-model}, and showcase its potential extension in practical applications. In cancer genetic network inference~\cite{ref:zhao2019cancer}, the 
GGM can be used to identify conditional dependencies between genes based on gene expression data. By constructing an undirected graph from a sparse precision matrix, GGM reveals regulatory relationships among genes. More precisely, if $\xi_1,\ldots,\xi_n$ represent the log-transformed gene expression data (for $n$ different genes) of a representative individual and it is assumed that $\xi=(\xi_1,\ldots,\xi_n)$ is distributed according {to ${\rm N}_{m,\Sigma}$}, then, as mentioned in Section \ref{sec:application-to-gaussian-graphical-model}, an edge $(i,j)$ is {\em not} contained in $E$ if and only if $\xi_i$ and $\xi_j$ are conditionally independent given $(\xi_\ell)_{\ell\in V\setminus\{i,j\}}$. In other words, an edge $(i,j)$ is contained in $E$ if and only if $\xi_i$ and $\xi_j$ are conditionally dependent.

In~\cite{ref:zhao2019cancer}, the authors consider this model {for $15$ different cancer types} and estimate the {$15$} unknown (sparse) precision matrices $S=\Sigma^{-1}$ using the estimator $\widehat S_N$ (defined by (\ref{eq:sampl-precision-mtx})) based on RNA-Seq data $\xi^1,\ldots,\xi^N$ from $N$ different individuals {($N$ depends on the cancer type and ranges from $102$ to $1094$)}, where the data are taken from The Cancer Genome Atlas (TCGA). They assume that the observations $\xi^1,\ldots,\xi^N$ are i.i.d.\ according to ${\rm N}_{\mathbf{0},\Sigma}$ in each case.


In our experiment, we focus on the specific cancer type BRCA and consider ${\rm N}_{\mathbf{0},\Sigma_P}$ with $\Sigma_P:=S_P^{-1}$ as the ``unperturbed truth'', where $S_P$ is chosen to be the estimated precision matrix {(based on a sample of size $N=1094$)} shown in Table s1 in the supplemental material of~\cite{ref:zhao2019cancer}. We restrict ourselves to the gene interaction structure among $n=100$ different types of genes (chosen out of 360 genes) to keep the problem dimension moderate. Furthermore, we consider ``contaminations'' $Q\in\mathscr{Q}:=\{{\rm N}_{\bm 0,\Sigma+\alpha\Sigma'}:\alpha\in[0,1]\}$. To avoid an overly artificial setting, we have left the choice of the parameter $\Sigma'$ to independent samples of an $100$-variate standard normal distribution (exactly as in Section \ref{sec:dsoeotscm}). 
We estimate the precision matrix based on $N$ samples from $P$ and $Q$, using the estimator $\widetilde S_N$ {(defined by (\ref{eq:sampl-precision-mtx-tilde}))}. In~\cite{ref:zhao2019cancer}, the regularization parameter $\lambda$ (see (\ref{eq:L1-MLE}))  is chosen as $0.03$. In our experiment, we consider more potential choices 
($\lambda\in\{0.01,0.03,0.05,0.07,0.1\}$). Figure~\ref{fig:cancer_gene} visualizes the results of the numerical experiments for sample sizes $N=1000$, $1500$, and $2000$. 


The plot on the left hand side of Figures~\ref{fig:cancer_gene} shows that $\widehat \dd_{{\rm K},M}(\P^P\circ\widetilde S_N^{-1},\P^Q\circ\widetilde S_N^{-1})$ deteriorate smoothly and slowly when the data-generation distribution deviates from $P$. This result validates the theoretical analysis as in Theorem~\ref{thm-SR-Sprs-Preci-mtx}. Note here that, analogous to Section \ref{sec:dsoeotscm}, we approximated the Kantorovich distance $\dd_{\rm K}(\P^P\circ\widetilde S_N^{-1},\P^Q\circ\widetilde S_N^{-1})$ by a Monte Carlo approximation $\widehat \dd_{{\rm K},M}(\P^P\circ\widetilde S_N^{-1},\P^Q\circ\widetilde S_N^{-1})$ based on $M:=100$ Monte Carlo repetitions. 


The goodness of the realized value of the estimator $\widetilde S_N$ of the precision matrix $S_P$ w.r.t.\ the property of sparsity can be quantified by the structure match accuracy ${\widetilde{a}}_{P,Q;N}$, which is defined as follows. Given {$Q\in\mathscr{Q}$ is the distribution of $\xi$,} we set $(\widetilde A_{P,Q;N})_{i,j}:=1$ if {$(S_P)_{i,j}$} and $(\widetilde S_N)_{i,j}$ are both equal to zero or both non-zero, and $(\widetilde A_{P,Q;N})_{i,j}:=0$ if not.
The {\it structure match accuracy} is defined by ${\widetilde{a}}_{P,Q;N}:=\frac{1}{n^2}\sum_{i=1}^n\sum_{j=1}^n(\widetilde A_{P,Q;N})_{i,j}$. Of course, in practical applications the value $\widetilde{a}_{P,Q;N}$ is unknown, since $P$ and $Q$ are unknown. In our experiment, we also investigated the behavior of the structure match accuracy. We calculate the empirical mean of $\widetilde{a}_{P,Q;N}$ based on $100$ Monte Carlo repetitions. 
Looking at the plots on the right hand side of Figures~\ref{fig:cancer_gene}, we find that although the structure match accuracy declines smoothly in {each setting,} 
the estimator induced by {the largest $\lambda$} 
is 
{most} stable when data perturbation occurs. This result is consistent with Theorem~\ref{thm-SR-Sprs-Preci-mtx}, taking into account that this theorem indicates that a larger $\lambda$ gives rise to a smaller Lipschitz constant $\kappa$, and thus a more distributionally stable estimator.

}

}

\subsection{Portfolio optimization}\label{sec-two-additional-examples}

\subsubsection{Distributional stability of the estimator of the optimal value}\label{sec:portfolio-opt}

Let $\xi=(\xi^1,\ldots,\xi^n)$ be an $\R^n$-valued random variable whose coordinates $\xi^1,\ldots,\xi^n$ model the one-period net returns of $n$ assets traded on a financial market. An agent is to invest a capital of size $1$ into the $n$ assets in such a way that the portfolio risk (i.e.\ the variance of the portfolio's return) is minimized while the expected portfolio's return is set to a given level $z\in\R_{++}$. 
More precisely, if $\mu$ and $\Sigma$ are used to denote the mean vector and the covariance matrix of $\xi$, respectively, 
the agent is to determine  
\begin{align}\label{Markowitz-model}
    v(\mu,\Sigma):=\min_{w\in \R^n}\;\;&\frac{1}{2}w^\trans \Sigma w\quad\inmat{ 
    subject to}\quad w^\trans \mu = z,~ w^\trans \bm 1=1,~ w\geq 0,
\end{align}
where $\bm 1:=(1,\ldots,1)^\trans$ and it is assumed that $\mu\in \R^n$ and $\Sigma\in\PSD$. Since the feasible set is compact, the minimum in (\ref{Markowitz-model}) is attained.

In practice, the distribution of $\xi$ is unknown, which means that $\mu$ and $\Sigma$ have to be estimated statistically, often assuming that the estimators used are based on i.i.d.\ copies $\xi^1,\ldots,\xi^N$ of $\xi$. As before we consider the nonparametric statistical model $(\Omega,{\cal F},\{\P^P:P\in\mathscr{P}_2(\R^n)\})$ introduced in Section \ref{Sroeocmapm} and assume that $\xi^1,\ldots,\xi^N$ are the coordinate projections on $\Omega=(\R^n)^N$. Recall that $\P^P:=\P^{\otimes N}$. Replacing $\mu$ and $\Sigma$ in (\ref{Markowitz-model}) with the empirical mean vector $\widehat\mu_N=\widehat\mu_N(\bm\xi)$ and the empirical covariance matrix $\widehat\Sigma_N=\widehat\Sigma_N(\bm\xi)$, respectively, leads to 
\begin{align*}
    v(\widehat\mu_N,\widehat\Sigma_N):=\min_{w\in \R^n}\;\;&\frac{1}{2}w^\trans\widehat\Sigma_N w\quad\inmat{ subject to }\quad w^\trans \widehat\mu_N = z,~ w^\trans \bm 1=1,~ w\geq 0,
\end{align*}
where $\widehat\mu_N(\bm x):=\frac{1}{N}\sum_{i=1}^Nx^i$, $\bm x=(x^1,\ldots,x^N)\in\Omega$. For notational simplicity, we set $\widehat v_N:=v(\widehat\mu_N,\widehat\Sigma_N)$. 

For any $C_1, C_2\in\R_{++}$  with $C_1<C_2$, let $\mathscr{P}_{C_1,C_2}(\R^n)$ be the set of all $P\in\mathscr{P}(\R^n)$ that have no mass outside the set $\{x\in\R^n: {(C_1^2+\la x,\bm 1\ra^2/n)^{1/2}\le\|x\|}\leq C_2\}$. Note that $\la \bm 1,x\ra/n^{1/2}\le$ ($\|x\|_1/n^{1/2}\le n^{1/2}\|x\|/n^{1/2}=$) $\|x\|$ for all $x\in\R^n$, and $\la \bm 1,x\ra/n^{1/2}=\|x\|$ if and only if $x=a\bm 1$ for some $a\in\R$. Therefore, for (very) small $C_1$, the condition $(C_1^2+\la x,\bm 1\ra^2/n)^{1/2}\le\|x\|$ in the definition of $\mathscr{P}_{C_1,C_2}$ rules out only a (very) small area around the set $\{a\bm 1:a\in\R\}$. From an application perspective, this limitation is negligible, as one would never hold a portfolio consisting of several assets with the same net returns; in this case one would only invest in the asset with the lowest variance. {The proof of the following proposition is given in Section \ref{Sec:Proof:thm:statistical-robustness-of-value-function-of-Markowitz-model}}.

\begin{proposition}\label{thm:statistical-robustness-of-value-function-of-Markowitz-model-PROPOSITION}
Let $C_1,C_2\in\R_{++}$ be such that $C_1<C_2$.
Then the function $v:\R^n\times\PSD\to\R$ is Lipschitz continuous on the subset $Y_{C_1,C_2}:=\{(\mu,\Sigma)\in\R^n\times\PSD:{(C_1^2+\la \mu,\bm 1\ra^2/n)^{1/2}\le\|\mu\|\le C_2},\,\|\Sigma\|\le 2C_2^2\}$ of $\R^n\times\PSD$ with Lipschitz constant 
$L_{C_1,C_2}:=\frac{1}{2}+\frac{32C_2^2}{9{C_1^2}}(z+\frac{C_2}{\sqrt{n}})+\frac{16C_2^2}{9\sqrt{n}{C_1^2}}$. 
\end{proposition}

Now, let $C_1,C_2\in\R_{++}$ be fixed such that $C_1<C_2$ and let us use $\dd_{\rm K}$ and $\dd_2$ to denote $\dd_{\R,1}$ and $\dd_{\R^n,2}$, respectively. Then, as a direct consequence of Theorem \ref{thm-SR-Covn-mtx-with-mean}, Proposition \ref{thm:statistical-robustness-of-value-function-of-Markowitz-model-PROPOSITION} and Proposition \ref{prop:chain-rule}, we obtain 
\begin{align}\label{eq:SR-Potflo}   
    & \dd_{\rm K}\big(\P^{P}\circ\widehat v_N^{\,-1},\P^Q\circ\widehat v_N^{\,-1}\big) \le\, L_{C_1,C_2}\max\{4,2m_P,2m_Q\}\,\dd_2(P,Q)
\end{align}
for all $P,Q\in{\mathscr{P}_{C_1,C_2}}$ ($\subseteq\mathscr{P}_2(\R^n)$), where $L_{C_1,C_2}$ is as in Proposition \ref{thm:statistical-robustness-of-value-function-of-Markowitz-model-PROPOSITION} and  $m_p$ and $m_Q$ are bounded above by $C_2$. Moreover, proceeding as in the proof of (\ref{robustnessofeigens-3}) (see Section \ref{Sec:Proof:Statistical-robustness-of-eigenvalue}), using Proposition \ref{thm:statistical-robustness-of-value-function-of-Markowitz-model-PROPOSITION} instead of \eqref{stability-of-eigenvalues}, we also obtain
\begin{align}\label{eq:SR-Potflo-2}   
    & \dd_{\rm K}\big(\P^{P}\circ\widehat v_N^{\,-1},\delta_{v(\mu_P,\Sigma_P)}\big) \le\, L_{C_1,C_2}\big(\E^P\big[\|\widehat\mu_N-\mu_P\|\big]+\E^P\big[\|\widehat\Sigma_N-\Sigma_P\|\big]\big)
\end{align}
for all $P\in{\mathscr{P}_{C_1,C_2}}$. The triangle inequality and (\ref{eq:SR-Potflo}) and (\ref{eq:SR-Potflo-2}) imply 
\begin{align}\label{eq:SR-Potflo-3}   
    \dd_{\rm K}\big(\P^{Q}\circ\widehat v_N^{\,-1},\delta_{v(\mu_P,\Sigma_P)}\big)
    & \,\le\, L_{C_1,C_2}\max\{4,2m_P,2m_Q\}\,\dd_2(P,Q)\\
    & \qquad +L_{C_1,C_2}\big(\E^P\big[\|\widehat\mu_N-\mu_P\|\big]+\E^P\big[\|\widehat\Sigma_N-\Sigma_P\|\big]\big) \nonumber
\end{align}
for all $P,Q\in{\mathscr{P}_{C_1,C_2}}$. Inequality (\ref{eq:SR-Potflo-3}) shows that $v(\mu_P,\Sigma_P)$ can be reasonably estimated even if the data are drawn from a ``contaminated'' distribution $Q$ that is slightly different from $P$. Recall from Proposition \ref{thm-SR-Covn-mtx:REMARK} that for the second summand on the right-hand side of (\ref{eq:SR-Potflo-3}) we know that $\lim_{N\to\infty}N^{(r-1)/r}(\E^P[\|\widehat\mu_N-\mu_P\|]+\E^P[\|\widehat\Sigma_N-\Sigma_P\|])=0$ for all $r\in[1,2)$, if $P,Q\in\mathscr{P}_{C_1,C_2}(\R^n)$ ($\subseteq\mathscr{P}_p$ for all $p\ge 1$). 


\subsubsection{Numerical experiment in portfolio optimization}

\begin{figure}
    \centering
    \includegraphics[width=0.5\textwidth]{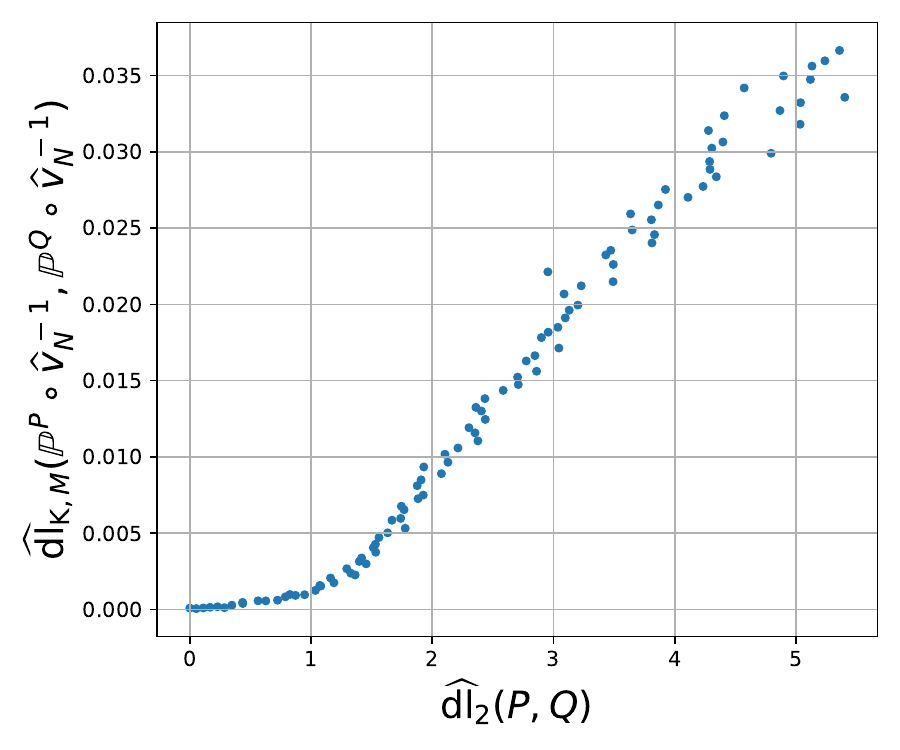}
    \caption{{$\widehat \dd_{{\rm K},M}(\P^P \circ \widehat v_N^{-1},\P^Q \circ \widehat v_N^{-1})$ as a function of $\widehat{\dd}_{2}(P,Q)$, $Q\in\mathscr{Q}$} 
    }
    \label{fig:value_function_dKvsdK}
\end{figure}

In the framework of Section \ref{sec:portfolio-opt}, inequality \eqref{eq:SR-Potflo} shows that the Kantorovich distance $\dd_{\rm K}(\P^P \circ \widehat v_{N}^{\,-1},\P^Q \circ \widehat v_{N}^{\,-1})$ between the distributions of the optimal portfolio value under two different distributions $P$ and $Q$ grows at most linearly as $\dd_2(P,Q)$ deviates from zero. We conducted a numerical experiment, which provides support for this finding.

In this experiment, we choose $P:={\rm LN}_{\mu,\Sigma}$ and $Q\in\mathscr{Q}:=\{{\rm LN}_{\mu+\alpha\mu',\Sigma+\alpha\Sigma'}:\alpha\in[0,1]\}$, where $\mu,\mu'\in\R^n$ and $\Sigma,\Sigma\in\PSD$ are fixed and ${\rm LN}_{\widetilde\mu,\widetilde\Sigma}$ is used to denote the $n$-variate log-normal distribution with parameters $\widetilde\mu$ and $\widetilde\Sigma$. To avoid an overly artificial setting, we have left the choice of the parameters $\mu$, $\mu'$, $\Sigma$ and $\Sigma'$ to two independent samples of an $n$-variate standard normal distribution (exactly as in Section \ref{sec:dsoeotscm}), where we chose $n:=5$ and $N:=1000$. Analogous to Section \ref{sec:dsoeotscm}, we approximated the Kantorovich distance $\dd_{\rm K}(\P^P \circ \widehat v_{N}^{\,-1},\P^Q \circ \widehat v_{N}^{\,-1})$ by a Monte Carlo approximation $\widehat \dd_{{\rm K},M}(\P^P \circ \widehat v_{N}^{\,-1},\P^Q \circ \widehat v_{N}^{\,-1})$ based on $M:=1000$ Monte Carlo repetitions. 

From Figure~\ref{fig:value_function_dKvsdK}, we can see that the (approximated) Kantorovich distance 
$\widehat \dd_{{\rm K},M}(\P^P \circ \widehat v_{N}^{\,-1},\P^Q \circ \widehat v_{N}^{\,-1})$ increases at most linearly  
as $\dd_{2}(P,Q)$ increases. This provides support for Inequality (\ref{eq:SR-Potflo}).



\section{Concluding remarks}
\label{sec:concluding-remarks}
Differing from the mainstream research in matrix optimization, here we study the distributional behavior of statistical estimators of the covariance matrix and the precision matrix. Specifically 
we investigate how a perturbation of the sampling distribution may affect the performance of the sample covariance matrix and its eigenvalues and of a sparse estimator of the precision matrix. We do this by looking at the distributions of the mentioned estimators based on a perceived sampling distribution compared to the distributions based on the target sampling distribution. We are able to establish a sort of stability, which means it is ``safe'' to use the estimators when the perturbation of the sampling distribution is confined to some specified 
data structure. 

There are still a number of open questions that remain to be addressed. For instance, when in the penalty term in (\ref{eq:L1-MLE}) the norm $\|\cdot\|_1$ is replaced by a different norm such as $\|\cdot\|_{1-}$, we are 
unable to assert Lipschitz continuity of the minimizer of the optimization problem underlying the estimator (in the spirit of (\ref{eq:S^*sigma-Lip}))
and subsequently unable to obtain distributional stability
(in the spirit of (\ref{eq:SR-gnl-nom-space}).
Likewise, we are unable to assert distributional stability of the empirical estimators of various ratios 
of interest in the context of portfolio theory, including the Sharpe ratio, the Omega ratio and the Rachev ratio. 
{ To address the issue, one may consider 
qualitative statistical robustness, which does not require the Lipschitz continuity, see \cite{kratschmer2014comparative,guo2021statistical}.
In that case, we can  establish a weaker version of distributional stability of the statistical estimators where 
the distance between statistical estimators 
based on contaminated data and true data
is implicitly related 
to the distance between the original probability distributions generating the data sets.
}
We leave all these to future research.


\appendix
\normalsize

\section{Auxiliary results}
\label{Sec:Some auxiliary results}

\subsection{An implicit function theorem}\label{Sec:ift}

For the proof of Theorem \ref{thm-glb-Lip-spars-precin-mtx} (see Section \ref{Sec: Proof of Proof of thm-glb-Lip-spars-precin-mtx}) we need the implicit function theorem in the form of Theorem \ref{thm-implt-fnct-thm} below. 
For its formulation, we need some terminology. Let $(X,\|\cdot\|_X)$ and $(Y,\|\cdot\|_Y)$ be normed linear spaces. A function $f:X\to Y$ is said to be Lipschitz continuous on a given subset ${X_0}\subseteq X$ if there exists a constant $L\in\R_{++}$ such that $\|f(x)-f(x')\|_Y \leq L_{\cal X} \|\hat x-\tilde x\|_X$ for all $\hat x,\tilde x\in{X_0}$. In this case, we say that $f$ is {\em L-Lipschitz continuous on ${X_0}$}. If $X$ is even a pre-Hilbert space with scalar product $\la\,\cdot\,,\,\cdot\,\ra_X$ and induced norm $\|\cdot\|_X$, then, for any convex subset ${X_0}\subseteq X$, a mapping $f:X\to X$ is said to be a monotone operator on ${X_0}$ if $\la f(\hat x)-f(\tilde x),\hat x-\tilde x\ra_X\geq 0$ for all $\hat x,\tilde x\in{X_0}$, and a strongly monotone operator on ${X_0}$ if there exists a constant $\rho\in\R_{++}$ such that $\la f(\hat x)-f(\tilde x),\hat x-\tilde x\ra_X\geq \rho\|\hat x-\tilde x\|_X^2$ for all $\hat x,\tilde x\in{X_0}$. In the latter case we say that $f$ is {\em $\rho$-strongly monotone on ${X_0}$}. The following theorem, which is known from Theorem F.1 in \cite{zhang2024statistical}, is an extension of Theorem 1H.3 in \cite{dontchev2009implicit}.

{
\begin{theorem}
\label{thm-implt-fnct-thm}
Let $(X,\|\cdot\|_X)$ be a Banach space and $(Y,\la\,\cdot\,,\,\cdot\,\ra_Y)$ a Hilbert space, and denote the topologies induced by $\|\cdot\|_X$ and $\la\,\cdot\,,\,\cdot\,\ra_Y$ by ${\cal O}_X$ and ${\cal O}_Y$, respectively. Let $Y_0$ be a convex
subset of $Y$ and consider a map $f:X\times Y_0\to Y$. Let $(\bar{x},\bar{y})$ be an interior point of $(X\times Y_0,{\cal O}_X\otimes({\cal O}_Y\cap Y_0))$ satisfying $f(\bar{x},\bar{y})={\bm 0}$. Let $\rho,L\in\R_{++}$ be constants and assume that there exists a neighborhood ${\cal N}(\bar{x})$ of $\bar{x}$ in $(X,{\cal O}_X)$ and a neighborhood ${\cal N}_0(\bar{y})$ of $\bar{y}$ in $(Y_0,{\cal O}_Y\cap Y_0)$ such that the following conditions are met:
\begin{itemize} 
    \item[{\rm (a)}] $f$ is continuous on $({\cal N}(\bar x) \times {\cal N}_0(\bar y),({\cal O}_X\cap{\cal N}(\bar x))\otimes({\cal O}_Y\cap{\cal N}_0(\bar y)))$. 
    \item[{\rm (b)}] For any $x\in{\cal N}(\bar x)$, the function $f(x,\,\cdot\,)$ is $\rho$-strongly monotone on ${\cal N}_0(\bar y)$.
    \item[{\rm (c)}] For any $y\in{\cal N}_0(\bar y)$, the function $f(\,\cdot\,,y)$ is $L$-Lipschitz continuous on ${\cal N}(\bar x)$.
\end{itemize}
Let $s:X\rightrightarrows Y_0$ be the solution mapping of $f$, which is defined by $s(x):=\{y\in Y_0:f(x,y)={\bm 0}\}$. 
Then there exists a neighborhood ${\cal N}_0(\bar x)$ ($\subseteq{\cal N}(\bar x)$) of $\bar x$ in $X$ such that 
\begin{itemize}
    \item[{\rm (i)}] $s(\cdot)$ is a single-valued function on ${\cal N}_0(\bar x)$,  
    \item[{\rm (ii)}] and this single-valued function is $L/\rho$-Lipschitz continuous on ${\cal N}_0(\bar x)$.
\end{itemize}
\end{theorem}

}

\subsection{Local Lipschitz continuity of covariance matrix and its eigenvalues and eigenvectors}\label{Sec: LLcocmaieae}

Although the main focus of this paper is on the stability of {\em estimators} for the covariance matrix and its eigenvalues (and the precision matrix), we will also discuss the stability of the {\em true} covariance matrix and its eigenvalues in this section. 
Let us use $\xi=(\xi_1,\ldots,\xi_n)$ to denote the identity on $\R^n$. If we equip the measurable space $(\R^n,{\cal B}(\R^n))$ with some $P\in\mathscr{P}(\R^n)$, we can regard $\xi$ as a random variable distributed according to $P$. For any $P\in\mathscr{P}_1(\R^n)$, we write $\mu_P$ for the expectation of $\xi$ w.r.t.\ $P$, denoted by $\E_P[\,\cdot\,]$, i.e.\ $\mu_P=\E_P[\xi]$. For any $P\in\mathscr{P}_2(\R^n)$,
we write $\Sigma_P$ for the covariance matrix of $\xi$ w.r.t.\ $P$, i.e.\ $\Sigma_P=\E_{P}[(\xi-\E_{P}[\xi])(\xi-\E_{P}[\xi])^\trans]$. The following proposition shows in particular that the mapping $\mathscr{P}_2(\R^n)\to \mathbb{S}^n$, $P\mapsto\Sigma_P$ is locally Lipschitz continuous w.r.t.\ the Fortet-Mourier metric of order $2$. The Fortet-Mourier metric $\dd_p:=\dd_{\R^n,p}$ of arbitrary order $p\ge 1$ was introduced in (\ref{eq:dfn-Fortet-Mourier}).


%

\begin{proposition}\label{prop-cvn-mtx-Lip}
For any $P, Q\in\mathscr{P}_2(\R^n)$, we have
\begin{align}
    \|\Sigma_P-\Sigma_Q\|
    & \leq \, n\,\dd_2(P,Q)+\sqrt{n}\big(\|\mu_P\|+\|\mu_Q\|\big)\dd_1(P,Q)\nonumber\\
    & \leq \, \big(n+\sqrt{n}(\|\mu_P\|+\|\mu_Q\|\big)\dd_2(P,Q),\label{eq:cvn-mtx-Lip-EQ}
\end{align}
\end{proposition}

\begin{proof}
For any $P, Q\in \mathscr{P}(\R^n)$, we obtain
\begin{align*}
    & \|\Sigma_P-\Sigma_Q\|
    =\big\|\E_P[\xi\xi^\trans]-\mu_P\mu_P^\trans-\E_Q[\xi\xi^\trans]+\mu_Q\mu_Q^\trans\big\|\\
    &\leq \big\|\E_P[\xi\xi^\trans]-\E_Q[\xi\xi^\trans]\big\| + \big\|\mu_P\mu_P^\trans-\mu_Q\mu_Q^\trans\big\|\\
    &\leq \big\|\E_P[\xi\xi^\trans]-\E_Q[\xi\xi^\trans]\big\| + \|\mu_P\|\|\mu_P-\mu_Q\|+\|\mu_Q\|\|\mu_P-\mu_Q\|\\
    &\leq \Big(\sum_{i=1}^n\sum_{j=1}^n \big(\E_P[\xi_i\xi_j]-\E_Q[\xi_i\xi_j]\big)^2\Big)^{1/2} + \big(\|\mu_P\|+\|\mu_Q\|\big)\|\mu_P-\mu_Q\|\\
    &\leq \Big(\sum_{i=1}^n\sum_{j=1}^n\Big(\int x_ix_j\,(P-Q)\big(d(x_1,\ldots,x_n)\big)\Big)^2\Big)^{1/2}\\
    & \qquad+ \big(\|\mu_P\|+\|\mu_Q\|\big)\Big(\sum_{i=1}^n\Big(\int x_i\,(P-Q)\big(d(x_1,\ldots,x_n)\big)\Big)^2\Big)^{1/2}\\
    &\leq \sqrt{n^2\dd_2(P,Q)^2} + \big(\|\mu_P\|+\|\mu_Q\|\big)\dd_2(P,Q)\\
    &\leq n\,\dd_2(P,Q) +\sqrt{n}\,\dd_1(P,Q) \big(\|\mu_P\|+\|\mu_Q\|\big),
\end{align*}
where we use the Cauchy-Schwarz inequality and that the mapping $(x_1,\ldots,x_n)\mapsto x_ix_j$ is contained in ${\cal F}_2(\R^n)$ and the mapping $(x_1,\ldots,x_n)\mapsto x_i$ is contained in ${\cal F}_1(\R^n)$. This proves the first inequality in (\ref{eq:cvn-mtx-Lip-EQ}). The second inequality in (\ref{eq:cvn-mtx-Lip-EQ}) is trivial.
{{\hspace*{\fill}$\Box$\par\bigskip}}
\end{proof}

%



Proposition \ref{prop-cvn-mtx-Lip} can be combined with the following results known from the literature to obtain stability results for the eigenvalues and eigenvectors of $\Sigma_P$. Let $\Sigma_1,\Sigma_2\in\mathbb{S}^n$. For $j=1,2$, let $\lambda_j^{1},\ldots,\lambda_j^n$ be the eigenvalues of $\Sigma_j$ in decreasing order (i.e.\ $\lambda_j^{1}\geq \lambda_j^{2}\geq \cdots\geq \lambda_j^n$) and set $\Lambda_j=\mathrm{diag}(\lambda_j^1,\ldots,\lambda_j^n)$. Then, for each $i=1,\ldots,n$, we have
\begin{align}\label{stability-of-eigenvalues}
    |\lambda_1^i-\lambda_2^i| \leq \|\Lambda_1-\Lambda_2\| \leq \|\Sigma_1-\Sigma_2\|,
\end{align}
see Fact 4 on p.\,15-2 in \cite{li2006matrix}. Moreover, for each $i=1,\ldots,n$ with $\lambda_1^{i+1}<\lambda_1^i<\lambda_1^{i-1}$ (where $\lambda_1^0:=\infty$, $\lambda_1^{n+1}:=-\infty$), we have
\begin{align}\label{stability-of-eigenvectors}
    \|v_1^i-v_2^i\|_2 \leq \frac{2^{3/2}\|\Sigma_1-\Sigma_2\|_2}{\min\{\lambda_1^{i-1}-\lambda_1^i,\lambda_1^i-\lambda_1^{i+1}\}}
\end{align}
for any eigenvectors $v_1^i$ and $v_2^i$ of ($\Sigma_1$ and $\Sigma_2$) associated with $\lambda_1^i$ and $\lambda_2^i$, respectively, which satisfy $\la v_1^i,v_2^i\ra \geq 0$. This is known from Corollary 3 (to a version of the classical Davis-Kahan $\sin\theta$ theorem \cite{davis1970rotation}) in \cite{yu2015useful}. The condition $\la v_1^i,v_2^i\ra \geq 0$ is always attainable by changing the signs of eigenvectors properly.

\section{Proofs}\label{Sec:Proofs}


\subsection{Minimizer of (\ref{eq:L1-MLE}) for $\lambda=0$}\label{Appendix:motivation SN}

For the convenience of the reader, we recall here the well-known argument why $\widehat S_N$ defined by (\ref{eq:L1-MLE}) in the case $\lambda=0$ is given by the sample precision matrix $\widehat\Sigma_N^{-1}$, provided $\widehat\Sigma_N$ is contained in $\PD$. For $\lambda=0$ the expression to be minimised on the right-hand side in (\ref{eq:L1-MLE}) coincides with the log-likelihood function in the case where $\xi$ is $n$-variate Gaussian with (arbitrary mean vector and) covariance matrix $S^{-1}$, and the corresponding maximum likelihood estimator for the covariance matrix $S^{-1}$ is known to be $\widehat\Sigma_N$. That is, the minimizer of the expression on the right-hand side in (\ref{eq:L1-MLE}) over all $S^{-1}$, $S\in\PD$, is $S_*^{-1}:=\widehat\Sigma_N$, i.e.\ $S_*=\widehat\Sigma_N^{-1}$. Thus, $\widehat S_N=\widehat\Sigma_N^{-1}$.


\subsection{Proof of Theorem \ref{thm-SR-gnl-nom-space}}\label{Sec: Proof of thm-SR-gnl-nom-space}

The proof of Theorem \ref{thm-SR-gnl-nom-space} relies on the following lemma. The lemma involves the class $\Psi_{\kappa_1,\kappa_2}$ of all functions $\psi:X^{N}\to\R$ that  satisfy
\begin{align*}
    & \big|\psi(\hat{\bm{x}})-\psi(\tilde{\bm{x}})\big|\\
    & \,\le\, \frac{\kappa_1}{N}\sum_{i=1}^N L_2(\hat{x}^i,\tilde{x}^i)\|\hat{x}^i-\tilde{x}^i\|_X+\frac{\kappa_2}{N^2}\sum_{i=1}^N \big(\|\hat{x}^i\|_X+\|\tilde{x}^i \|_X\big)\sum_{k=1}^N \|\hat{x}^k-\tilde{x}^k\|_X
\end{align*}
for all $\hat{\bm x}=(\hat x^1,\ldots,\hat x^N),\,\tilde{\bm x}=(\tilde x^1,\ldots,\tilde x^N)\in  X^{N}$. Recall that $\kappa_1,\kappa_2\in\R_{+}$ are given constants.

\begin{lemma}\label{lem-SR-gnl-nrm-space}
For any $P,Q\in \mathscr{P}_2(X)$, we have
\begin{align*}
    \sup_{\psi\in\Psi_{\kappa_1,\kappa_2}} \Big|\int_{X^{N}}\psi\,\diff \P^P-\int_{X^{N}}\psi\,\diff\P^Q\Big|\le\,\max\big\{\kappa_1+\kappa_2,2m_P,2m_Q\big\}\, \dd_2(P,Q).
\end{align*}
\end{lemma}

\begin{proof}
The proof is similar to \cite[Lemma 1]{guo2021statistical}. Here we give a sketch of it to facilitate reading. Let $\bm{x}_{-j}:=(x^1,\ldots,x^{j-1},x^{j+1},\ldots,x^N)$ for any $\bm{x}=(x^1,\ldots,x^N)\in X^N$ and $j=1,\ldots,N$. Moreover, for any $\boldsymbol{P}=(P_1,\ldots,P_N)\in\mathscr{P}_2(X)^N$, $\psi\in\Psi_{\kappa_1,\kappa_2}$ and $j=1,\ldots,N$, write $\boldsymbol{P}_{-j}(\mathrm{d} \bm{x}_{-j})$ for $P_1(\mathrm{d}x^1)\cdots P_{j-1}(\mathrm{d}x^{j-1}) P_{j+1}(\mathrm{d}x^{j+1})\cdots P_{N}(\mathrm{d}x^{N})$ and set $\psi_{-j}(x^j):=\int_{X^{N-1}}\psi(\bm{x}_{-j},x^j)\,\boldsymbol{P}_{-j}(\mathrm{d} \bm{x}_{-j})$ for any $\boldsymbol{x}\in X^N$ and $x^j\in X$.
Then, for any $\boldsymbol{P}=(P_1,\ldots,P_N)\in\mathscr{P}_2(X)^N$, $\psi\in\Psi_{\kappa_1,\kappa_2}$ and $j=1,\ldots,N$, we obtain 
\begin{align*}
   & \big|\psi_{-j}(\hat{x}^j)-\psi_{-j}(\tilde{x}^j)\big| \\
   & \,\le\, \int_{X^{N}}\big|\psi(\bm{x}_{-j},\hat{x}^j)-\psi(\bm{x}_{-j},\tilde{x}^j)\big|\,\boldsymbol{P}_{-j}(\mathrm{d} \bm{x}_{-j})\\
   & \,\le\, \int_{X^{N}}\Big(\frac{\kappa_1}{N} L_2(\hat{x}^j,\tilde{x}^j)\|\hat{x}^j-\tilde{x}^j\|_X\\
   & \qquad\qquad+ \frac{\kappa_2}{N^2} \|\hat{x}^j-\tilde{x}^j\|\Big(\|\hat{x}^j\|_X+\|\tilde{x}^j\|_X+\sum_{i=1,i\neq j}^N 2\|x^i\|_X \Big)\Big)\,\boldsymbol{P}_{-j}(\mathrm{d} \bm{x}_{-j})\\
   & \,=\, \frac{\kappa_1}{N} L_2(\hat{x}^j,\tilde{x}^j)\|\hat{x}^j-\tilde{x}^j\|_X+ \frac{\kappa_2}{N^2} \|\hat{x}^j-\tilde{x}^j\|_X \big(\|\hat{x}^j\|_X+\|\tilde{x}^j\|_X\big) \\
   & \qquad\qquad+ \frac{2\kappa_2}{N^2} \sum_{i=1,i\neq j}^N\|\hat{x}^j-\tilde{x}^j\|_X \int_{X^{N}}  \|x^i\|_X\,\boldsymbol{P}_{-j}(\mathrm{d} \bm{x}_{-j})\\
   & \,=\, \frac{\kappa_1}{N} L_2(\hat{x}^j,\tilde{x}^j)\|\hat{x}^j-\tilde{x}^j\|_X+ \frac{\kappa_2}{N^2} \|\hat{x}^j-\tilde{x}^j\|_X \big(\|\hat{x}^j\|_X+\|\tilde{x}^j\|_X\big)\\
   & \qquad\qquad + \frac{2\kappa_2}{N^2} \|\hat{x}^j-\tilde{x}^j\|_X \sum_{i=1,i\neq j}^N \int_{X^{N}}  \|x^i\|_X\,P_i(\mathrm{d}x^i)\\
   & \,=\, \frac{\kappa_1}{N} L_2(\hat{x}^j,\tilde{x}^j)\|\hat{x}^j-\tilde{x}^j\|_X+ \frac{\kappa_2}{N^2} \|\hat{x}^j-\tilde{x}^j\|_X \big(\|\hat{x}^j\|_X+\|\tilde{x}^j\|_X\big)\\
   & \qquad\qquad + \frac{2\kappa_2}{N^2} \|\hat{x}^j-\tilde{x}^j\|_X \sum_{i=1,i\neq j}^N m_{P_i}\\
   & \,\le\, \frac{\kappa_1}{N} L_2(\hat{x}^j,\tilde{x}^j)\|\hat{x}^j-\tilde{x}^j\|_X+ \frac{2\kappa_2}{N^2}L_2(\hat{x}^j,\tilde{x}^j)\|\hat{x}^j-\tilde{x}^j\|_X\\
   & \qquad\qquad + \frac{2\kappa_2}{N^2} \|\hat{x}^j-\tilde{x}^j\|_X\sum_{i=1,i\neq j}^N m_{P_i})\\
   & \,=\, \frac{\kappa_1+\kappa_2}{N}L_2(\hat{x}^j,\tilde{x}^j)\|\hat{x}^j-\tilde{x}^j\|_X + \frac{2\kappa_2}{N^2} \|\hat{x}^j-\tilde{x}^j\|_X\, \sum_{i=1,i\neq j}^N m_{P_i}\\
   & \,\le\, \frac{1}{N}\max \Big\{\kappa_1+\kappa_2, \frac{2\kappa_2}{N}\sum_{i=1,i\neq j}^N m_{P_i}\Big\}\,L_2(\hat{x}^j,\tilde{x}^j) \|\hat{x}^j-\tilde{x}^j\|_X
\end{align*}
and so, for any choice of $\hat P_j,\tilde P_j\in\mathscr{P}_2(X)$, 
\begin{align*}
    & \sup_{\psi\in\Psi_{\kappa_1,\kappa_2}}\Big|\int_X \int_{X^{N-1}} \psi(\bm{x}_{-j},{x}_j)\,\boldsymbol{P}_{-j}(\mathrm{d} \bm{x}_{-j})\hat{P}_j(\mathrm{d} \xi_j)\\
    & \qquad\qquad-\int_X \int_{X^{N-1}} \psi(\bm{x}_{-j},{\xi}_j)\,\boldsymbol{P}_{-j}(\mathrm{d} \bm{x}_{-j})\,\tilde{P}_j(\mathrm{d}x_j)\Big|\\
    & \,=\, \sup_{\psi\in\Psi_{\kappa_1,\kappa_2}} \Big|\int_X \psi_{-j} (x_j)\,\hat {P}_j(\mathrm{d}x_j)-\int_X \psi_{-j}(x_j)\,\tilde {P}_j(\mathrm{d} x_j)\Big|\\
    & \,\le\, \frac{1}{N}\max \Big\{\kappa_1+\kappa_2, \frac{2\kappa_2}{N}\sum_{i=1,i\neq j}^N m_{P_i}\Big\}\,\dd_2(\hat{P}_j,\tilde{P}_j)
\end{align*}
Therefore, we obtain
\begin{align*}
    & \sup_{\psi\in\Psi_{\kappa_1,\kappa_2}} \Big|\int_{X^{N}}\psi\,\diff \P^P-\int_{X^{N}}\psi\,\diff\P^Q\Big|\,=\, \sup_{\psi\in\Psi_{\kappa_1,\kappa_2}} \Big|\int_{X^{N}}\psi\,\diff P^{\otimes N}-\int_{X^{N}}\psi\,\diff Q^{\otimes N}\Big|\\
    & \,\le\, \sup_{\psi\in\Psi_{\kappa_1,\kappa_2}} \Big|\int_{X^{N}}\psi\,\diff P^{\otimes N}-\int_{X^{N}}\psi\,\diff\big(P^{\otimes(N-1)}\otimes Q\big)\Big|\\
    & \qquad+\sup_{\psi\in\Psi_{\kappa_1,\kappa_2}} \Big|\int_{X^{N}}\psi\,\diff\big(P^{\otimes(N-1)}\otimes Q\big)-\int_{X^{N}}\psi\,\diff\big(P^{\otimes(N-2)}\otimes Q^{\otimes 2}\big)\Big|\\
    & \qquad+\cdots+\sup_{\psi\in\Psi_{\kappa_1,\kappa_2}} \Big|\int_{X^{N}}\psi\,\diff\big(P\otimes Q^{\otimes(N-1)}\big)-\int_{X^{N}}\psi\,\diff Q^{\otimes N}\Big|\\
    & \,\le\, \frac{1}{N}\max \Big\{\kappa_1+\kappa_2, \frac{2\kappa_2}{N}(N-1)m_{P}\Big\}\,\dd_2(P,Q) \\
    & \qquad + \frac{1}{N}\max \Big\{\kappa_1+\kappa_2, \frac{2\kappa_2}{N}\big((N-2)m_p+m_Q\big)\Big\}\Big\}\,\dd_2(P,Q)\\
    & \qquad + \cdots+\frac{1}{N}\max \Big\{\kappa_1+\kappa_2, \frac{2\kappa_2}{N}(N-1)m_Q\Big\}\Big\}\,\dd_2(P,Q)\\
    & \,\le\, \max\big\{\kappa_1+\kappa_2,2\kappa_2m_P,2\kappa_2m_Q\big\}\,\dd_2(P,Q)
\end{align*}
for any $P,Q\in\mathscr{P}_2(X)$.
{{\hspace*{\fill}$\Box$\par\bigskip}}
\end{proof}


We are now ready to prove Theorem \ref{thm-SR-gnl-nom-space}. We have
\begin{align}\label{dl1}
    \dd_{Y,1}\big(\P^P\circ\widehat{T}_N^{-1},\P^Q\circ\widehat{T}_N^{-1}\big)
    & \,=\,\sup_{g\in\mathcal{F}_1(Y)}\Big|\int_{Y} g\,\diff \P^P\circ\widehat{T}_N^{-1}-\int_{Y} g\,\diff \P^Q\circ\widehat{T}_N^{-1}\Big|\nonumber\\
    & \,=\,\sup_{g\in\mathcal{F}_1(Y)}\Big|\int_{X^{N}} g(\widehat{T}_N)\,\diff \P^P-\int_{X^{N}} g(\widehat{T}_N)\,\diff\P^Q\Big|\nonumber\\
    & \,=\,\sup_{g\in\mathcal{F}_1(Y)}\Big|\int_{X^{N}} \psi_g\,\diff \P^P-\int_{X^{N}}\psi_g\,\diff\P^Q\Big|,
\end{align}
where $\psi_g:=g(\widehat T_N)$. Moreover, for any $g\in{\cal F}_1(Y)$, we have 
\begin{align*}
    & \big|\psi_g(\hat{\bm{x}})-\psi_g(\tilde{\bm{x}})\big| \,=\, \big|g\big(\widehat{T}_N(\hat{\bm{x}})\big)-g\big(\widehat{T}_N(\tilde{\bm{x}})\big)\big| \,\le\, \big\|\widehat{T}_N(\hat{\bm{x}})-\widehat{T}_N(\tilde{\bm{x}})\big\|_Y\\
    & \,\le\, \frac{\kappa_1}{N}\sum_{i=1}^N L_2(\hat{x}^i,\tilde{x}^i) \|\hat{x}^i-\tilde{x}^i\|+\frac{\kappa_2}{N^2} \sum_{i=1}^N (\|\hat{x}^i\|+\|\tilde{x}^i\|)\sum_{j=1}^N \|\hat{x}^j-\tilde{x}^j\|
\end{align*}
for all $\hat{\bm{x}}=(\hat x^1,\ldots,\hat x^N),\,\tilde{\bm{x}}=(\tilde x^1,\ldots,\tilde x^N)\in X^N$, which means that $\psi_g:=g(\widehat T_N)$ belongs to the set $\Psi_{\kappa_1,\kappa_2}$. So (\ref{eq:SR-gnl-nom-space}) follows directly from (\ref{dl1}) and Lemma \ref{lem-SR-gnl-nrm-space}.

\subsection{Proof of Proposition \ref{prop:boundns-of-sln}}\label{SEC:proof-prop:boundns-of-sln}

{
Let $\lambda\in\R_{++}$ and $\Sigma\in\PSD$ be arbitrary but fixed. 
Choose $\underline{K},\overline{K}\in\R_{++}$ such that $\underline{K}<\min\{\frac{n}{2\|\Sigma\|},\frac{1}{2\lambda}\}$ and $\overline{K}>\frac{n}{\lambda}$, where we use the convention $\frac{n}{2\|\Sigma\|}:=\infty$ if $\|\Sigma\|=0$. Consider the non-empty compact set
\begin{align}\label{SEC:proof-prop:boundns-of-sln-der-comp-K}
    \K:=\big\{S\in \PD:\,\underline{K} \le \|S\| \le\overline{K}\big\},
\end{align}

We first note that the mapping $\PD\to\R$, $S\mapsto L(\lambda,\Sigma,S)$ is strictly convex. This follows from the fact that the mapping $\PD\to\R$, $S\mapsto\log(\det S)$ is strictly concave (see \cite[Theorem 7.6.6]{horn2012matrix}), noting that $S\mapsto\la\Sigma,S\ra$ is linear (thus convex) and $S\mapsto\lambda\|S\|_1$ is convex. Moreover, the mapping $\PD\to\R$, $S\mapsto L(\lambda,\Sigma,S)$ is continuous, which follows from the continuity of the mapping $\R^{n\times n}\to\R$, $A\mapsto\det(A)$. The restriction of the continuous mapping $S\mapsto L(\lambda,\Sigma,S)$ to the compact set $\K$ attains its minimum (by the well-known extreme value theorem). In the following, we show that (a) the attainable minimizer over $\K$ is globally optimal for $S\mapsto L(\lambda,\Sigma,S)$ over $\PD$, and (b) the minimizer 
over $\PD$ 
exists and is unique. Note that (b) is obvious given (a), since $S\mapsto L(\lambda,\Sigma,S)$ is strictly convex and $\PD$ is a convex set. To show (a), It suffices to prove that for any $S\in\PD\setminus\K$, there exists an $S'\in \K$ such that $L(\lambda,\Sigma, S')<L(\lambda,\Sigma, S)$. {So let $S\in\PD\setminus\K$}. We consider the following two cases:

(i) If $\|S\| > {\overline{K}}$, we choose $h\in(0,1)$ so small that $(-\log(1-h)/h)(n/\lambda)<\overline{K}$ (not that $-\log(1-h)/h\downarrow 1$ as $h\downarrow 0$) and set $S_j^{(h)}:=(1-h)^jS$ for $j=0,\ldots,J$, where $J\in\N$ is chosen such that $\|S_{J-1}^{(h)}\|>\overline{K}$ and $\underline{K} \le \|S_J^{(h)}\|\le\overline{K}$. Note that such a $J$ always exists and that $S_J^{(h)}\in\K$. Then, for any $j=0,\ldots,J-1$, we have
\begin{align}
    & L\big(\lambda, \Sigma, S_j^{(h)}\big) - L\big(\lambda, \Sigma, S_{j+1}^{(h)}\big) \nonumber\\
    & \,=\, h\la \Sigma, S_j^{(h)}\ra +n\log(1-h)+h\lambda \|S_j^{(h)}\|_1\nonumber\\ 
    & \,\ge\, n\log(1-h)+h\lambda\|S_j^{(h)}\|  \,\ge\, n\log(1-h)+h\lambda {\overline{K}}\,>\,0, \nonumber
\end{align} 
where the second step follows from $\la \Sigma, S_j^{(h)}\ra=\mathrm{tr}(\Sigma S_j^{(h)})\ge 0$ (since both $\Sigma$ and $S_j^{(h)}$ are positive semi-definite) and $\|\cdot\|\le\|\cdot\|_1$. Thus, $L(\lambda, \Sigma, S) - L(\lambda, \Sigma, S')=\sum_{j=0}^{J-1}(L(\lambda, \Sigma, S_j^{(h)}) - L(\lambda, \Sigma, S_{j+1}^{(h)}))>0$ for $S':=S_J^{(h)}\in\K$. 

(ii) If $\|S\| < \underline{K}$, we choose $h\in\R_{++}$ so small that $(\log(1+h)/h)\min\{\frac{n}{2\|\Sigma\|},\frac{1}{2\lambda}\}>\underline{K}$ (not that $\log(1+h)/h\uparrow 1$ as $h\downarrow 0$) and set $S_j^{(h)}:=(1+h)^j S$ for $j=0,\ldots,J$, where $J\in\N$ is chosen such that $\|S_{J-1}^{(h)}\|<\underline{K}$ and $\underline{K}\le\|S_J^{(h)}\|\le\overline{K}$. Again, such a $J$ always exists and $S_j^{(h)}\in\K$. For any $j=0,\ldots,J-1$, we have
\begin{align}
    & L(\lambda, \Sigma, S_j^{(h)}) - L(\lambda, \Sigma, S_{j+1}^{(h)})\nonumber\\ 
    & \,=\, -h\la \Sigma, S_j^{(h)}\ra + n\log(1+h)-h\lambda\|S_j^{(h)}\|_1\nonumber\\
    & \,\ge\, -h\|\Sigma\|\|S_j^{(h)}\|+n\log(1+h)-h\lambda n\|S_j^{(h)}\|\nonumber\\
    & \,=\, \Big(\frac{1}{2}n\log(1+h)-h\|\Sigma\|\|S_j^{(h)}\|\Big) +\Big(\frac{1}{2}n\log(1+h)-h\lambda n\|S_j^{(h)}\|\Big) >\, 0, \nonumber
\end{align}
where the second step follows from the Cauchy-Schwarz inequality $|\la \Sigma, S_j^{(h)}\ra|\le\|\Sigma\|\|S_j^{(h)}\|$ and $\|\cdot\|_1\le n\|\cdot\|$. Thus, $L(\lambda, \Sigma, S) - L(\lambda, \Sigma, S')=\sum_{j=0}^{J-1}(L(\lambda, \Sigma, S_j^{(h)}) - L(\lambda, \Sigma, S_{j+1}^{(h)}))>0$ for $S':=S_J^{(h)}\in\K$.

Finally, we note that $\|S^*(\lambda,\Sigma)\|=\|S_{\K}\| \le\overline{K}$ for any $\Sigma\in\PSD$ and $\overline{K}\in(n/\lambda,\infty)$; note that $\K$ depends on $\overline{K}$. Thus, $\|S^*(\lambda,\Sigma)\|\le n/\lambda$ . So the last statement in Proposition \ref{prop:boundns-of-sln} also holds.
}

{
\subsection{Proof of Proposition \ref{prop-L-growth-cvn-mtx}}

The proof of Proposition \ref{prop-L-growth-cvn-mtx} will be based on Lemmas \ref{lemma:convex-function-linear-growth} and \ref{lem-stn-monotn-invs-mtx} and will be given after the proof of Lemma \ref{lem-stn-monotn-invs-mtx}. 
The statement of Lemma \ref{lemma:convex-function-linear-growth} should be well-known. However, as we could not find it in the literature, we provide a proof below.

\begin{lemma}\label{lemma:convex-function-linear-growth}
Let $(X, \|\cdot\|_X)$ be a finite-dimensional normed vector space. Let $X_0$ be a nonempty convex subset of $X$, and let $f:X_0\to\mathbb{R}$ be a convex function. Suppose that $f$ has a unique minimizer, $x^\ast$, and that $x^*\in\operatorname{int} X_0$.
Set $\overline{\delta}:=\sup\{\delta\in\R_{++}: \{x\in X: \|x-x^*\|\leq \delta\}\subseteq \operatorname{int} X_0\}$. Then for every $\delta\in(0,\overline{\delta})$ there exists a constant $\beta_{\delta}>0$ depending on $\delta$ such that
\begin{align}\label{lemma:convex-function-linear-growth-eq}
    x\in X_0\,,\quad \|x - x^\ast\|_X \,\ge\, \delta
    \quad\Longrightarrow\quad
    f(x) - f(x^\ast) \,\ge\, \beta_{\delta}\,\|x-x^\ast\|_X.
\end{align}
\end{lemma}

\begin{proof}
Let $\delta\in(0,\overline{\delta})$ be arbitrary but fixed. Consider the sphere $\mathcal{X}_\delta := \{x\in X_0 : \|x - x^\ast\|_X = \delta\}$ and note that it is nonempty and compact. Furthermore, define
\begin{align}\label{eq:m_delta}
    m(\delta) := \inf_{x\in\mathcal{X}_\delta} \bigl(f(x) - f(x^\ast)\bigr).
\end{align}
By~\cite[Theorem 4.1.3]{ref:borwein2006convex}, the function $f:\operatorname{int} X_0\to\R$ is continuous. In particular, its restriction to $\mathcal{X}_\delta$ is continuous. Thus, the infimum in~\eqref{eq:m_delta} is attainable at some point in $\mathcal{X}_\delta$, denoted by $\overline{x}$. Since $x^\ast$ is the unique minimizer of $f$, and $\|x-x^*\|_X=\delta > 0$, we can conclude that $m(\delta) > 0$.

Now consider an arbitrary $x\in X_0$ with $\|x-x^\ast\|_X\ge\delta$. Set $\theta:=\delta/\|x-x^*\|_X$ ($\in(0,1]$). Then the vector $y := x^\ast + \theta(x - x^\ast)$ lies on the line segment between $x^*$ and $x$, and it is contained in $\mathcal{X}_\delta$ (since $\|y - x^\ast\|_X = \|\theta(x - x^\ast)\|_X = \theta \|x - x^\ast\|_X= \delta$). By convexity of $f$ along the line segment, we have
\[
f(y) \,=\, f\bigl((1-\theta)x^\ast + \theta x\bigr)
\,\le\, (1-\theta) f(x^\ast) + \theta f(x).
\]
Rearranging the above inequality yields that
\begin{align*}
f(x) - f(x^\ast)\,\ge\, \frac{1}{\theta}\bigl(f(y) - f(x^\ast)\bigr)\,\ge\, \frac{1}{\theta} m(\delta) \,=\, \frac{m(\delta)}{\delta} \,\|x-x^\ast\|_X,
\end{align*}
where the second step is ensured by the definition of $m(\delta)$ in (\ref{eq:m_delta}) and the fact that $y\in\mathcal{X}_{\delta}$, and the last step relies on the definition of $\theta$. 
Since $\delta$ is independent of $x$ ($\in X_0$), we arrive at (\ref{lemma:convex-function-linear-growth-eq}) for $\beta_{\delta} := m(\delta)/\delta > 0$.
{\hspace*{\fill}$\Box$\par\bigskip}
\end{proof}

}

\begin{lemma}\label{lem-stn-monotn-invs-mtx}
Let 
{$\rho:=1/C_\lambda^2$.}
Then we have
\begin{align*}
    & -\big\la S_1^{-1}-S_2^{-1}, S_1-S_2\big\ra \geq \rho\|S_1-S_2\|^2\quad\mbox{ for all }S_1, S_2\in {\cal S}_\lambda.
\end{align*}
\end{lemma}

\begin{proof}
{For any $S_1,S_2\in{\cal S}_\lambda$, we have}
\begin{align}\label{lem-stn-monotn-invs-mtx-proof-eq-00}
    -\la S_1^{-1}-S_2^{-1}, S_1-S_2\ra &= -\la S_1^{-1}(S_2-S_1)S_2^{-1},S_1-S_2\ra= \la S_1^{-1}\Delta S S_2^{-1},\Delta S\ra,
\end{align}
 where $\Delta S:= S_1-S_2$. 
By the spectral decomposition, we may write $S_1$ and $S_2$ respectively as $S_1=Q_1 \Lambda_1 Q_1^\trans$ and $S_2=Q_2\Lambda_2 Q_2^\trans$ for some diagonal matrices $\Lambda_1=\Lambda_1(S_1),\,\Lambda_2:=\Lambda_2(S_2)$ and orthogonal matrices $Q_1,Q_2$. Then
\begin{align}
    &\la S_1^{-1}\Delta S\,S_2^{-1},\Delta S\ra \,=\, \mathrm{tr}\big(Q_2\Lambda_2^{-1}Q_2^\trans
        \Delta S\, Q_1 \Lambda_1^{-1} Q_1^\trans \Delta S\big)\nonumber\\
    & \,=\, \mathrm{tr}\big(\Lambda_2^{-1}Q_2^\trans\Delta S\,Q_1 \Lambda_1^{-1} Q_1^\trans \Delta S\,Q_2\big) \,\ge\, \lambda_{\max}^{-1}(\Lambda_2)\, \mathrm{tr}\big(Q_2^\trans\Delta S\,Q_1 \Lambda_1^{-1} Q_1^\trans \Delta S\,Q_2\big)\nonumber\\
    & \,=\, \lambda_{\max}^{-1}(\Lambda_2)\, \mathrm{tr}\big(\Delta S\, Q_1 \Lambda_1^{-1} Q_1^\trans \Delta S\big) \,=\, \lambda_{\max}^{-1}(\Lambda_2)\, \mathrm{tr}\big(\Lambda_1^{-1} Q_1^\trans\, \Delta S\,\Delta S\,Q_1\big)\nonumber\\
    & \,\ge\, \lambda_{\max}^{-1}(\Lambda_2)\,\lambda_{\max}^{-1}(\Lambda_1)\,\mathrm{tr}\big(Q_1^\trans \Delta S\,\Delta S\,Q_1\big) \,=\, \lambda_{\max}^{-1}\,(\Lambda_2)\,\lambda_{\max}^{-1}(\Lambda_1)\,\mathrm{tr}\big(\Delta S\, \Delta S \big)\nonumber\\
    & \,=\,\lambda_{\max}^{-1}(\Lambda_2)\,\lambda_{\max}^{-1}(\Lambda_1)\,\|\Delta S\|^2.\label{lem-stn-monotn-invs-mtx-proof-eq}
\end{align}
By Proposition \ref{prop:boundns-of-sln} 
we have {$\|S\|\leq C_\lambda$} 
for any $S\in\mathcal{S}_{\lambda}$. In particular, $\lambda_{\max}(\Lambda(S))=\lambda_{\max}(S)\leq \|S\|\le {C_\lambda}$ 
(where $S=Q(S)^\trans\Lambda(S)Q(S)$ is the spectral decomposition of $S$) for any $S\in\mathcal{S}_{\lambda}$. So by (\ref{lem-stn-monotn-invs-mtx-proof-eq-00}) and (\ref{lem-stn-monotn-invs-mtx-proof-eq}) we obtain 
$-\la S_1^{-1}-S_2^{-1}, S_1-S_2\ra=\la S_1^{-1}\Delta S\,S_2^{-1},\Delta S\ra \geq \lambda_{\max}^{-1}(\Lambda_2)\,\lambda_{\max}^{-1}(\Lambda_1)\,\|\Delta S\|^2{\ge C_\lambda^{-2}} 
\|S_1-S_2\|^2 =\rho\|S_1-S_2\|^2$. 
{\hspace*{\fill}$\Box$\par\bigskip}
\end{proof}

{

We are now in the position to prove Proposition \ref{prop-L-growth-cvn-mtx}. Fix $\lambda\in\R_{++}$ and $\Sigma\in\PSD$. To simplify the notation in this proof, the dependence of the quantities involved in $\lambda$ and $\Sigma$ will be omitted. That is, we will write $L(S)$, $S^*$, $\alpha$ and $\beta$ instead of $L(\lambda,\Sigma,S)$, $S^*(\lambda,\Sigma)$, $\alpha_{\lambda,\Sigma}$ and $\beta_{\lambda,\Sigma}$, respectively. 

We have to show that (\ref{eq:L-growth-in-S}) holds for suitable constants $\alpha,\beta\in\R_{++}$. It is sufficient to show that there exist constants $\delta_0,\alpha\in\R_{++}$ such that 
\begin{align}\label{ineq:quadratic-L}
    S\in\PD\,,\quad\,\|S-S^*\|\leq \delta_0\quad\Longrightarrow\quad L(S)-L(S^*) \geq \alpha \|S-S^*\|^2.
\end{align}
In fact, if \eqref{ineq:quadratic-L} holds for some $\delta_0,\alpha\in\R_{++}$, then (\ref{eq:L-growth-in-S}) holds for the same $\alpha$ and some $\beta\in\R_{++}$ (depending on $\delta_0$). This can be seen as follows. Let $\delta\in(0, \min\{\delta_0, \overline{\delta}\})$ be arbitrary but fixed, where $\overline{\delta}:=\sup\{\delta>0: \{S\in \S^n: \|S-S^*\|\leq \delta\}\subseteq \operatorname{int} \PD\}$. From Proposition \ref{prop:boundns-of-sln} we know that the function $L:\PD\to\R$ is (strictly) convex and has a unique minimizer, namely $S^*$. Moreover, its domain $\PD$ is open and convex. Hence, the function $L$ meets the assumptions of Lemma~\ref{lemma:convex-function-linear-growth}. Then Lemma~\ref{lemma:convex-function-linear-growth} implies that $L(S)-L(S^*) \geq \beta\|S-S^*\|$ for all $S\in\PD$ with $\|S-S^*\| \ge \delta$, where $\beta=\beta_{\delta}$ is a constant that only depends on $\delta$ (but is independent of $S$).
By~\eqref{ineq:quadratic-L}, we have $L(S)-L(S^*) \geq \alpha \|S-S^*\|^2$ for all $S\in\PD$ with $\|S-S^*\|\leq \delta$. Thus, we have $L(S)-L(S^*) \geq \min\{\alpha\|S-S^*\|^2, \beta\|S-S^*\|\}$ for all $S\in\PD$, i.e.\  (\ref{eq:L-growth-in-S}). 

It remains to show~\eqref{ineq:quadratic-L} for some suitable constants $\delta_0,\alpha\in\R_{++}$. From (\ref{eq:Int-optimal-sln}) and~\eqref{eq:S-feasible-set-bnd} we know that $S^*\in \operatorname{int} \mathcal{S}_{\lambda}$, where $\mathcal{S}_{\lambda}=\big\{S\in\PD:\,\|S\| \leq C_\lambda\big\}$. 
Then, if $\delta_0$ is chosen so small that the closed ball centered at $S^*$ with radius $\delta_0$ is contained in $\operatorname{int}{\cal S}_\lambda$,
by \cite[Exercise 3.58]{ref:drusvyatskiy2020convex} it is sufficient to show that there exists an $\alpha\in\R_{++}$ such that the mapping $\operatorname{int}\mathcal{S}_\lambda\mapsto \R$, $S\mapsto L(S)=\la \Sigma, S\ra - \log(\det S)+\lambda \|S\|_{1}$ is $2\alpha$-strongly convex.
Since the mapping $\operatorname{int}\mathcal{S}_\lambda\to\R$, $S\mapsto\lambda \|S\|_{1}$ is convex, it is sufficient to show that the mapping $\operatorname{int}\mathcal{S}_\lambda\to\R$, $S\mapsto\widetilde L(S):=\la \Sigma, S\ra - \log(\det S)$ is $2\alpha$-strongly convex for some $\alpha\in\R_{++}$. By \cite[Theorem B.4.1.4]{ref:hiriart2004fundamentals} (its proof also works for a general finite-dimensional real Hilbert space instead of $\R^n$), the latter is equivalent to the condition that the gradient $\nabla\widetilde  L(\cdot)$ of $\widetilde L$ is $2\alpha$-strongly monotone over $\operatorname{int}\mathcal{S}_{\lambda}$.
Note here that $\widetilde L:\operatorname{int}\mathcal{S}_\lambda\to\R$ is Fr\'echet differentiable. The Fr\'echet derivative of $\widetilde L$ at $S$ is given by $D\widetilde L_S(V)= \la\Sigma,V\ra - \la S^{-1},V\ra =\la\Sigma-S^{-1},V\ra$, which follows from \cite[Appendix A.4.1]{boyd2004convex} (according to which, for any $S\in\PD$, the Frech\'et derivative of the mapping $\log\det(\cdot):\PD\to\R$ at $S$ in direction $V$ is given by $\mathrm{tr}(S^{-1}V)=\la S^{-1},V\ra$). Thus, $\nabla\widetilde L(S)=\Sigma-S^{-1}$ (see \cite[p.\,41]{Bauschke2011Convex}). 
Note that $2\alpha$-strong monotonicity of $\nabla\widetilde  L(\cdot)$ over $\operatorname{int}\mathcal{S}_{\lambda}$ means that $\la\nabla\widetilde  L(S_1)-\nabla\widetilde  L(S_2),S_1-S_2\ra\ge 2\alpha\|S_1-S_2\|^2$ for all $S_1,S_2\in{\cal S}_\lambda$, which is clearly equivalent to $-\la S_1^{-1}-S_2^{-1},S_1-S_2\ra\ge 2\alpha\|S_1-S_2\|^2$ for all $S_1,S_2\in{\cal S}_\lambda$. However, according to Lemma~\ref{lem-stn-monotn-invs-mtx}, this holds for $\alpha:=1/(2C_\lambda^2)$. This completes the proof.}

\subsection{Proof of Theorem \ref{thm-prci-mtx-Lip-hld-cnty}}

Fix $\lambda\in\R_{++}$ and $\Sigma\in\PSD$, and let $C_\lambda$ and $\alpha_{\lambda,\Sigma},\beta_{\lambda,\Sigma}$ {be as in (\ref{eq:Int-optimal-sln})--(\ref{eq:S-feasible-set-bnd})} and Proposition \ref{prop-L-growth-cvn-mtx}, respectively. Let ${\cal S}_\lambda$ be as in (\ref{eq:S-feasible-set-bnd}) and recall that problems (\ref{eq:L-PrecisionM-spse}) and (\ref{problem-with-compact-feasible-set}) are equivalent. 
{Let $R_{\lambda,\Sigma}:\R_+\to\R$ be the function defined by (\ref{prop-smth-appx-prci-mtx-lemma-eq}), i.e.\ }
$$
    {R_{\lambda,\Sigma}(\varepsilon)} = \inf_{S\in\mathcal{S}_\lambda:\, \|S- S^*(\lambda,\Sigma)\| \geq \varepsilon} L(\lambda,\Sigma,S) - L\big(\lambda,\Sigma,S^*(\lambda,\Sigma)\big)\quad\mbox{for any $\varepsilon\in\R_{++}$}.
$$

Note that $\sup_{S\in\mathcal{S}_\lambda}|L(\lambda,\Sigma,S)-L(\lambda,\Sigma',S)|=\sup_{S\in\mathcal{S}_\lambda}|\la \Sigma - \Sigma', S \ra|\le\sup_{S\in\mathcal{S}_\lambda}\break\|\Sigma - \Sigma'\|\|S\| \,\leq\, \|\Sigma-\Sigma'\|\,C_\lambda$ for all $\Sigma'\in\PSD$. Thus, for any $\varepsilon\in\R_{++}$, $S\in\mathcal{S}_\lambda$ with $\|S-S^*(\lambda,\Sigma)\|\ge\varepsilon$ and $\Sigma'\in\PSD$, we obtain
\begin{align}\label{bounded-function-L}
    & L(\lambda,\Sigma',S)-L\big(\lambda,\Sigma',S^*(\lambda,\Sigma')\big)\nonumber\\
    & \,=\, L(\lambda,\Sigma,S)-L\big(\lambda,\Sigma,S^*(\lambda,\Sigma)\big)\nonumber\\
    & \quad~+L\big(\lambda,\Sigma',S^*(\lambda,\Sigma)\big)-L\big(\lambda,\Sigma',S^*(\lambda,\Sigma')\big)\nonumber\\
    & \quad~+ L\big(\lambda,\Sigma,S^*(\lambda,\Sigma)\big)-L\big(\lambda,\Sigma',S^*(\lambda,\Sigma)\big)+L(\lambda,\Sigma',S)-L(\lambda,\Sigma,S)\nonumber\\
    & \,\ge\, L(\lambda,\Sigma,S) - L\big(\lambda,\Sigma,S^*(\lambda,\Sigma)\big)+0-2C_\lambda\|\Sigma-\Sigma'\|\nonumber\\
    & \,\geq\, {R_{\lambda,\Sigma}(\varepsilon)}-2C_\lambda\|\Sigma-\Sigma'\|.
\end{align}
From Proposition \ref{prop-L-growth-cvn-mtx}, we know that ${R_{\lambda,\Sigma}(\varepsilon)} \geq \min\{\alpha_{\lambda,\Sigma}\varepsilon^2,\beta_{\lambda,\Sigma}\varepsilon\}$ for all $\varepsilon\in\R_{++}$. Thus, for any $\Sigma'\in\PSD$ with $\Sigma'\not=\Sigma$ we have ${R_{\lambda,\Sigma}}(\varepsilon')\ge 3C_\lambda\|\Sigma-\Sigma'\|$ for $\varepsilon' := \max\{(3C_\lambda\|\Sigma-\Sigma'\|/\alpha_{\lambda,\Sigma})^{1/2},3C_\lambda\|\Sigma-\Sigma'\|/\beta_{\lambda,\Sigma}\}$ ($>0$). In particular, (\ref{bounded-function-L}) implies 
$L(\lambda,\Sigma',S)-L(\lambda,\Sigma',S^*(\lambda,\Sigma'))\ge C_\lambda\|\Sigma-\Sigma'\| > 0$ for any $\Sigma'\in\PSD$ with $\Sigma'\not=\Sigma$ and any $S\in\mathcal{S}_\lambda$ with $\|S-S^*(\lambda,\Sigma)\|\ge\varepsilon'$. Thus, $S\not=S^*(\lambda,\Sigma')$ for any $\Sigma'\in\PSD$ and $S\in\mathcal{S}_\lambda$ with $\Sigma'\not=\Sigma$ and $\|S-S^*(\lambda,\Sigma)\|\ge\varepsilon'$.  
Since $S^*(\lambda,\Sigma')\in{\cal S}_\lambda$, we can conclude that $\|S^*(\lambda,\Sigma')-S^*(\lambda,\Sigma)\|<\varepsilon'=$ right-hand side of (\ref{thm-prci-mtx-Lip-hld-cnty-EQUATION}).


\subsection{Properties of the function $R_{\lambda,\Sigma}$ defined by (\ref{prop-smth-appx-prci-mtx-lemma-eq})}\label{prop-smth-appx-prci-mtx-lemma-sec-1}

Let $\lambda\in\R_{++}$ and $\Sigma\in\PSD$ be arbitrary but fixed. The function $R_{\lambda,\Sigma}:\R_+\to\R$ defined by (\ref{prop-smth-appx-prci-mtx-lemma-eq}) is clearly non-decreasing over $\R_{+}$ and satisfies $R_{\lambda,\Sigma}(0)=0$. {Since problem (\ref{eq:L-PrecisionM-spse}) has a unique minimizer (see Proposition \ref{prop:boundns-of-sln}),} we also have that $R_{\lambda,\Sigma}(\delta)>0$ for all $\delta\in\R_{++}$. 
Moreover, the convergence $R_{\lambda,\Sigma}(\delta)\to\infty$ as $\delta\uparrow\infty$ follows from  $L(\lambda,\Sigma,S) \geq -\log(\det S)+ \lambda \|S\| = -\log(\prod_{i=1}^n\lambda_{i}(S)) + \lambda\|S\| = -\sum_{i=1}^n \log(\lambda_{i}(S)) + \lambda\|S\| \geq -n \log(\lambda_{\max}(S)) + \lambda\|S\|$, $\|S\|^2=\sum_{i=1}^n\lambda_i(S)^2\ge \lambda_{\max}(S)^2$ and $\lim_{x\uparrow\infty}(-n \log(x) + \lambda x)=\infty$.

{

To prove continuity of the function $R_{\lambda,\Sigma}:\R_+\to\R$,  
we will show that 1) $R_{\lambda,\Sigma}$ is left-continuous at any point of $\R_{++}$ and 2) $R_{\lambda,\Sigma}$ is right-continuous at any point of $\R_+$. 
Note that $R_{\lambda,\Sigma}$ is monotonically increasing, which implies that the left-sided limits of $R_{\lambda,\Sigma}$ exist at any point of $\R_{++}$ and the right-sided limits of $R_{\lambda,\Sigma}$ exist at any point of $\R_+$. 

1) We first show that $R_{\lambda,\Sigma}$ is left-continuous at every point of $\R_{++}$. Let $\delta\in\R_{++}$ and consider a sequence $(\delta_\ell)_{\ell\in\N}$ in $\R_+$ with $\delta_\ell\uparrow \delta$. By the monotonicity of $R_{\lambda,\Sigma}$, we have $R_{\lambda,\Sigma}(\delta)\geq \limsup_{\ell\to\infty} R_{\lambda,\Sigma}(\delta_\ell)$. We will now show that $R_{\lambda,\Sigma}(\delta)\leq \liminf_{\ell\to\infty}R_{\lambda,\Sigma}(\delta_\ell)$ also holds, which implies $R_{\lambda,\Sigma}(\delta)=\lim_{\ell\to\infty}R_{\lambda,\Sigma}(\delta_\ell)$. 
By the definition of $R_{\lambda,\Sigma}$ in (\ref{prop-smth-appx-prci-mtx-lemma-eq}) and the definition of the infimum of a set, we can find for every $\ell\in\N$ a matrix $S_\ell\in\PD$ such that $\|S_\ell-S^*(\lambda,\Sigma)\|\ge\delta_\ell$ and $L(\lambda,\Sigma,S_\ell)-L(\lambda,\Sigma,S^*(\lambda,\Sigma))\le R_{\lambda,\Sigma}(\delta_\ell)+1/\ell$. Since $R_{\lambda,\Sigma}$ is monotonically increasing, we can conclude that 
\begin{align}\label{eq:bounded-Sn}
    L(\lambda,\Sigma,S_\ell)-L(\lambda,\Sigma,S^*(\lambda,\Sigma)) \,\le\, R_{\lambda,\Sigma}(\delta_\ell) +1/\ell \,\le\, R_{\lambda,\Sigma}(\delta) +1
\end{align}
for all $\ell\in \N$. Together with Proposition~\ref{prop-L-growth-cvn-mtx}, this implies
\begin{align}\label{eq:R-growth-cond}
    \min\big\{\alpha_{\lambda,\Sigma} \|S_\ell-S^*(\lambda,\Sigma)\|^2,\,\beta_{\lambda,\Sigma} \|S_\ell-S^*(\lambda,\Sigma)\|\big\}\leq R_{\lambda,\Sigma}(\delta) + 1
\end{align}
for all $\ell\in\N$. In particular, $(S_\ell)_{\ell\in\N}$ is a bounded sequence in the finite-dimensional linear space $\S^n$. Thus, by the Bolzano-Weierstrass theorem, there exists a convergent subsequence $(S_{\ell_k})_{k\in\N}$ with limit $S\in\S^n$. The limit $S$ is contained in $\PD$. In fact, we clearly have $S\in \PSD$ (since $\PSD$ is the closure of $\PD$), and in the case $S\in\PSD\setminus\PD$ the convergence $S_{\ell_k}\to S\in \PSD\setminus \PD$ would imply $L(\lambda,\Sigma,S_{\ell_k})\to\infty$, which would contradict with~\eqref{eq:bounded-Sn}. Moreover, by the continuity of the norm $\|\cdot\|$, we have $\|S-S^*(\lambda,\Sigma)\|\geq \delta$. Therefore, $S$ is feasible for the infimum in the definition of $R_{\lambda,\Sigma}(\delta)$, and so we have
\[
    L(\lambda,\Sigma,S)-L(\lambda,\Sigma,S^*(\lambda,\Sigma)) \geq R_{\lambda,\Sigma}(\delta).
\]
Moreover, by~\eqref{eq:bounded-Sn} and the continuity of $S\mapsto L(\lambda,\Sigma,S)$ (see Proposition \ref{prop:boundns-of-sln}), we have
\[
    L(\lambda,\Sigma,S)-L(\lambda,\Sigma,S^*(\lambda,\Sigma))\leq \liminf_{\ell\to\infty}R_{\lambda,\Sigma}(\delta_\ell).
\]
Combining the latter two inequalities gives $R_{\lambda,\Sigma}(\delta)\leq \liminf_{\ell\to\infty}R_{\lambda,\Sigma}(\delta_\ell)$. 
This completes the proof of left-continuity of $R_{\lambda,\Sigma}$.  

2) Now we prove that $R_{\lambda,\Sigma}$ is right-continuous at every point of $\R_+$. Let $\delta\in\R_+$ and consider a sequence $(\delta_\ell)_{\ell\in\N}$ in $\R_+$ with $\delta_\ell\downarrow \delta$. By the monotonicity of $R_{\lambda,\Sigma}$, we have $\liminf_{\ell\to \infty} R_{\lambda,\Sigma}(\delta_\ell) \geq R_{\lambda,\Sigma}(\delta)$. We will now show that $R_{\lambda,\Sigma}(\delta)\geq \limsup_{\ell\to\infty}R_{\lambda,\Sigma}(\delta_\ell)$ also holds, which implies $R_{\lambda,\Sigma}(\delta)=\lim_{\ell\to\infty}R_{\lambda,\Sigma}(\delta_\ell)$. Let $\varepsilon\in\R_{++}$ be arbitrary but fixed for the moment. By the definition of $R_{\lambda,\Sigma}$ in (\ref{prop-smth-appx-prci-mtx-lemma-eq}), we can find an $S_\varepsilon\in\PD$ such that $\|S_\varepsilon-S^*(\lambda,\Sigma)\|\geq \delta$ and $L(\lambda,\Sigma,S_\varepsilon)-L(\lambda,\Sigma,S^*(\lambda,\Sigma))< R_{\lambda,\Sigma}(\delta)+\varepsilon$. If we set $S_{\varepsilon,\gamma} := (1+\gamma)S_\varepsilon-\gamma S^*(\lambda,\Sigma)$, $\gamma\in\R_{++}$, then the continuity of $S\mapsto L(\lambda,\Sigma,S)$ (see Proposition \ref{prop:boundns-of-sln}) implies that we can find a sufficiently small $\gamma_\varepsilon\in\R_{++}$ such that 
\begin{align*}
    L(\lambda,\Sigma,S_{\varepsilon,\gamma_\varepsilon})-L(\lambda,\Sigma,S^*(\lambda,\Sigma))
    & \,<\, L(\lambda,\Sigma,S_\varepsilon) + \varepsilon -L(\lambda,\Sigma,S^*(\lambda,\Sigma))\\
    & \,<\, R_{\lambda,\Sigma}(\delta)+2\varepsilon.
\end{align*}
Since $\|S_{\varepsilon,\gamma_\varepsilon}-S^*(\lambda,\Sigma)\|=(1+\gamma_\varepsilon)\|S_\varepsilon-S^*(\lambda,\Sigma)\|\ge(1+\gamma_\varepsilon)\delta$, we can find a sufficiently large $\ell_{\varepsilon}\in\N$ such that $S_{\varepsilon,\gamma_\varepsilon}$ is feasible for the infimum in the definition of $R_{\lambda,\Sigma}(\delta_{\ell})$ for all $\ell\ge\ell_\varepsilon$. For this reason we have
\[
    R_{\lambda,\Sigma}(\delta_\ell) \,\le\, L(\lambda,\Sigma,S_{\varepsilon,\gamma_\varepsilon})-L(\lambda,\Sigma,S^*(\lambda,\Sigma)) < R_{\lambda,\Sigma}(\delta)+2\varepsilon
\]
for all $\ell\ge \ell_\varepsilon$. This implies $\limsup_{\ell\to\infty}R_{\lambda,\Sigma}(\delta_\ell)\le R_{\lambda,\Sigma}(\delta)+2\varepsilon$. Since $\varepsilon$ could be any value in $\R_{++}$, we arrive at $\limsup_{\ell\to\infty}R_{\lambda,\Sigma}(\delta_\ell)\le R_{\lambda,\Sigma}(\delta)$. 
This completes the proof of the right-continuity of $R_{\lambda,\Sigma}$.

}

\subsection{Proof of Proposition \ref{prop-smth-appx-prci-mtx}}

{
Let $\lambda\in\R_{++}$, $\Sigma\in\PSD$ and  $\varepsilon\in\R_{++}$be arbitrary but fixed.

{\em Step 1}. 
We first show that problem \eqref{smoothObjective}
admits a unique minimizer. This can be shown in a similar way 
to the proof of 
problem (\ref{eq:L-PrecisionM-spse}) having  a unique minimizer
in Section \ref{SEC:proof-prop:boundns-of-sln}. 
Let 
$\mathbb{K}_\varepsilon:=\{S\in \PD: c_{\lambda, \Sigma, \varepsilon}\leq \|S\|\leq C_\lambda\}$, 
where $c_{\lambda, \Sigma, \varepsilon}:=\min\{\frac{n\log 2}{2\|\Sigma\|}, (\frac{n\sqrt{\varepsilon}\log 2}{3\lambda})^{1/2}, C_\lambda/2\}$ and $C_\lambda$ is as in (\ref{eq:S-feasible-set-bnd}), and note that $\K_\varepsilon=\{S\in{\cal S}_\lambda: c_{\lambda, \Sigma, \varepsilon}\le\|S\|\}$. 
Recall that $L_\varepsilon(\lambda,\Sigma,S)$ is defined as $L(\lambda,\Sigma,S)$ with $\lambda\|S\|_1$ replaced by $\lambda H_\varepsilon(S)$. So it is easy to see that the mapping $S\mapsto L_\varepsilon(\lambda,\Sigma,S)$ is continuous and strictly convex. By the continuity, the mapping $S\mapsto L_\varepsilon(\lambda,\Sigma,S)$ attains its minimum on the compact set $\K_\varepsilon$. 
Let $S_{\K_\varepsilon}$ denote the set of global minimizers at which 
the minimum is attained over $\K_\varepsilon$. 
We show: (a)
$S_{\K_\varepsilon}$ coincides with 
the set of minimizers of $L_\varepsilon(\lambda,\Sigma,S)$
over ${\cal S}_\lambda$, and (b) $S_{\K_\varepsilon}$ is a singleton. Note that (b) is obvious given (a), since $L_\varepsilon(\lambda,\Sigma,S)$
is strictly convex in $S$ and 
${\cal S}_\lambda$ is a convex set.
Thus we are left to prove (a).
It suffices to prove that for any $S\in {\cal S}_\lambda\setminus \K_\varepsilon$, there exists an $S'\in\K_\varepsilon$ such that $L_\varepsilon(\lambda,\Sigma,S')<L_\varepsilon(\lambda,\Sigma, S)$. So let
$S\in{\cal S}_\lambda\setminus\K_\varepsilon$ be fixed. 
Then $\|S\| < c_{\lambda, \Sigma, \varepsilon}$. Set $S^{(k)}:=2^k S$ for $k=0,\ldots,K$, where $K\in\N$ is chosen such that $\|S^{(K-1)}\|<c_{\lambda, \Sigma, \varepsilon}$ and $c_{\lambda, \Sigma, \varepsilon}\leq\|S^{(K)}\|\leq C_\lambda$. For any $k=0,\ldots,K-1$, we have
\begin{align*}
    & L_{\varepsilon}(\lambda, \Sigma, S^{(k)}) - L_{\varepsilon}(\lambda, \Sigma, S^{(k+1)})\\ 
    & \,=\, -\la \Sigma, S^{(k)}\ra + n\log 2 + \lambda\big(H_\varepsilon(S^{(k)})-H_{\varepsilon}(2S^{(k)})\big)\\ 
    & \,\ge\, -\|\Sigma\|\|S^{(k)}\|+n\log 2-\frac{3\lambda}{2\sqrt{\varepsilon}}\|S^{(k)}\|^2\\
    & \,=\, \Big(\frac{1}{2}n\log 2-\|\Sigma\|\|S^{(k)}\|\Big) +\Big(\frac{1}{2}n\log 2-\frac{3\lambda}{2\sqrt{\varepsilon}}\|S^{(k)}\|^2\Big) \,>\, 0,
\end{align*}
{where the inequality follows from the Cauchy-Schwarz inequality $|\la \Sigma, S^{(k)}\ra|\le\|\Sigma\|\|S^{(k)}\|$ and the inequality}
\begin{align*}
    & H_\varepsilon(S)-H_{\varepsilon}(2S) \,=\, \sum_{1\leq i,j\leq n}\sqrt{S_{i,j}^2+\varepsilon} - \sum_{1\leq i,j\leq n}\sqrt{4S_{i,j}^2+\varepsilon} \\
    & \,=\, \sum_{1\leq i,j\leq n} \frac{-3S_{i,j}^2}{\sqrt{S_{i,j}^2+\varepsilon}+\sqrt{4S_{i,j}^2+\varepsilon}} \,\ge\,-\frac{3}{2\sqrt{\varepsilon}}\sum_{1\leq i,j\leq n} {S_{i,j}^2} \,=\, -\frac{3}{2\sqrt{\varepsilon}}\|S\|^2
\end{align*}
{(which holds for any $S$)}, and the last inequality is ensured by the choice of $c_{\lambda,\Sigma, \varepsilon}$. Thus, $L_\varepsilon(\lambda, \Sigma, S) - L_\varepsilon(\lambda, \Sigma, S')=\sum_{k=0}^{K-1}(L(\lambda, \Sigma, S^{(k)}) - L(\lambda, \Sigma, S^{(k+1)}))>0$ for $S':=S^{(K)}\in{\mathbb{K}_\varepsilon}$. 
}

{\em Step 2}. 
We will now prove inequality \eqref{ineq:continuity-of-S_eps} in a similar way to Lemma A1 in~\cite{guo2021existence}. 
Set { $\delta:=R_{\lambda,\Sigma}^{-1}(3\lambda n^2\sqrt{\varepsilon})$} and note that 
{by the continuity of $R_{\lambda,\Sigma}$, $R_{\lambda,\Sigma}(\delta)= 3\lambda n^2\sqrt{\varepsilon}$. }If we set $\eta:=R_{\lambda,\Sigma}(\delta)/3$,
then we get
\begin{align*}
    \sup_{S\in \mathcal{S}_\lambda}\big|L(\lambda,\Sigma,S)-L_\varepsilon(\lambda,\Sigma,S)\big|
    \le{ \lambda} \sum_{1\le i,j\le n} \sup_{s_{i,j}\in \R} \Big||s_{i,j}|-\sqrt{s_{i,j}^2+\varepsilon}\Big|
    \le {\lambda} n^2\sqrt{\varepsilon} = \eta.
\end{align*}
Consequently, for any $S\in{\cal S}_\lambda$ with $\|S-S^*(\lambda,\Sigma)\|\geq \delta$ we obtain that
\begin{align*}
    L_\varepsilon(\lambda,\Sigma, S) - L_\varepsilon\big(\lambda,\Sigma, S^*(\lambda,\Sigma)\big)
    & \,\ge\, L(\lambda,\Sigma,S) - L\big(\lambda,\Sigma, S^*(\lambda,\Sigma)\big) - 2\eta\\
    & \,\ge\, {R_{\lambda,\Sigma}(\delta)-2R_{\lambda,\Sigma}(\delta)/3\,=\,} R_{\lambda,\Sigma}(\delta)/3 \,>\, 0,
\end{align*}
where the second step relies on the definition of $R_{\lambda,\Sigma}(\delta)$ in (\ref{prop-smth-appx-prci-mtx-lemma-eq}). This means that $S$ is not a minimizer of the mapping ${\cal S}_\lambda\to\R$, $S'\mapsto L_\varepsilon(\lambda,\Sigma, S')$, which in turn implies that $\|S^*(\lambda,\Sigma)-S_\varepsilon^*(\lambda,\Sigma)\|\leq \delta=R_{\lambda,\Sigma}^{-1}({3\lambda n^2\sqrt{\varepsilon}})$.

\subsection{Proof of Theorem \ref{thm-glb-Lip-spars-precin-mtx}}\label{Sec: Proof of Proof of thm-glb-Lip-spars-precin-mtx}

{

We will prove (\ref{eq:S^*sigma-Lip}) for $\kappa:=C_\lambda^2$, where $C_\lambda$ is as introduced before (\ref{eq:S-feasible-set-bnd}). Recall that $C_\lambda>n/\lambda$ can be arbitrarily close to $n/\lambda$. 

In Step 1 below we will prove {the following: For any $\bar\Sigma\in\PSD$ and $\varepsilon\in\R_{++}$ with {$S_\varepsilon^*(\lambda,\bar\Sigma)\in\inmat{int}\,{\cal S}_\lambda$}} there exists a neighborhood ${\cal N}_{\varepsilon,0}(\bar\Sigma)$ of $\bar\Sigma$ in $\S^n$ such that 
\begin{align}\label{eq:S^*ep-sigma-Lip-pre}
    & \big\|S^*_\varepsilon(\lambda,\Sigma_1) - S^*_\varepsilon(\lambda,\Sigma_2)\big\| \le \kappa\|\Sigma_1-\Sigma_2\|\quad\mbox{ for all }\Sigma_1,\Sigma_2\in{\cal N}_{\varepsilon,0}(\bar\Sigma).
\end{align}
In Step 2 below we will use this to show that for any $\Sigma_1, \Sigma_2\in \PSD$ there exists an $\overline{\varepsilon}_{\Sigma_1,\Sigma_2}\in\R_{++}$ such that
\begin{align}\label{eq:S^*_eps-Lip}
    \big\|S_\varepsilon^*(\lambda,\Sigma_1)-S_\varepsilon^*(\lambda,\Sigma_2)\big\| \leq \kappa \|\Sigma_1-\Sigma_2\|\quad\mbox{ for all }\varepsilon\in (0, \overline{\varepsilon}_{\Sigma_1,\Sigma_2}]. 
\end{align}
Combining~\eqref{eq:S^*_eps-Lip} with Proposition~\ref{prop-smth-appx-prci-mtx} (according to which $S^*_\varepsilon(\lambda,\Sigma_i)\to S^*(\lambda,\Sigma_i)$ as $\varepsilon\downarrow 0$, $i=1,2$), we can conclude that 
\begin{align}\label{eq:S^*-Lip}
    \big\|S^*(\lambda,\Sigma_1)-S^*(\lambda,\Sigma_2)\big\| \leq \kappa \|\Sigma_1-\Sigma_2\|
\end{align}
for any $\Sigma_1,\Sigma_2\in \PSD$. This proves \eqref{eq:S^*sigma-Lip}.



{\em Step 1}. {Let $\bar\Sigma\in\PSD$ and $\varepsilon\in\R_{++}$ with {$S_\varepsilon^*(\lambda,\bar\Sigma)\in\inmat{int}\,{\cal S}_\lambda$}.} To prove (\ref{eq:S^*ep-sigma-Lip-pre}), we intend to apply the implicit function theorem (in the form of Theorem \ref{thm-implt-fnct-thm}) to the mapping $F_{\lambda,\varepsilon}:\S^n\times\PD\to\S^n$ defined by 
\begin{align*}
    F_{\lambda,\varepsilon}(\Sigma,S)\,:=\,G_\varepsilon(\lambda,\Sigma,S)\,=\,\Sigma - S^{-1} + \lambda \nabla H_\varepsilon(S);
\end{align*}
recall that $G_\varepsilon(\lambda,\Sigma,S)$ was introduced after \eqref{eq:G_ep=0}. To this end, we note that the following three assertions hold true: 
\begin{itemize}
    \item[(a)] The mapping $F_{\lambda,\varepsilon}:\S^n\times\PD\to\S^n$ is continuous. 
    \item[(b)] For any $\Sigma\in\S^n$, the mapping $\PD\to\S^n$, $S\mapsto F_{\lambda,\varepsilon}(\Sigma,S)$ is $\rho$-strongly monotone on ${\cal S}_\lambda$, where {$\rho:=1/C_\lambda^2$}. 
    \item[(c)] For any $S\in{\cal S}_\lambda$, the mapping $\S^n\to\S^n$, $\Sigma\mapsto F_{\lambda,\varepsilon}(\Sigma,S)$ is $1$-Lipschitz continuous (on all of $\S^n$).
\end{itemize}

Assertion (a) is obvious. Assertion (b) is true since, for any $\Sigma\in\S^n$, $\rho_\lambda$-strong monotonicity of the mapping $\PD\to\S^n$, $S\mapsto F_{\lambda,\varepsilon}(\Sigma,S)= \Sigma-S^{-1} +\lambda \nabla H_\varepsilon(S)$ follows from Lemma \ref{lem-stn-monotn-invs-mtx} and the fact that the mapping ${\cal S}_\lambda\to\S^n$, $S\mapsto\lambda\nabla H_\varepsilon(S)$ is a monotone operator. {The monotonicity of the operator $\lambda\nabla H_\varepsilon(\cdot)$ holds by \cite[Theorem B.4.1.4]{ref:hiriart2004fundamentals} (its proof also works for a general finite-dimensional real Hilbert space instead of $\R^n$),} {since $H_\varepsilon$ is convex}. Assertion (c) is true since, for any $S\in{\cal S}_\lambda$, the mapping $\S^n\to\S^n$, $\Sigma\mapsto F_{\lambda,\varepsilon}(\Sigma,S)=\Sigma-S^{-1} +\lambda \nabla H_\varepsilon(S)$ is clearly $1$-Lipschitz continuous.

Now let $\bar\Sigma\in\S_+^n$ be arbitrary but fixed. By assumption 
we know that $S_\varepsilon^*(\lambda,\bar\Sigma)\in\inmat{int}\,{\cal S}_\lambda$. 
Thus, we can find a neighborhood ${\cal N}_{\varepsilon,0}(S_\varepsilon^*(\lambda,\bar\Sigma))$ of $S_\varepsilon^*(\lambda,\bar\Sigma)$ in $\S^n$ that is contained in ${\cal S}_\lambda$ ($\subseteq\PD$). Also note that $F_{\lambda,\varepsilon}(\bar\Sigma,S_\varepsilon^*(\lambda,\bar\Sigma))={\bm 0}$. So, in view of the above assertions (a)--(c), we can conclude that the assumptions in Theorem \ref{thm-implt-fnct-thm} (with $X=Y=\S^n$, $Y_0=\PD$, $f=F_{\lambda,\varepsilon}$, $(\bar x,\bar y)=(\bar\Sigma,S_\varepsilon^*(\lambda,\bar\Sigma))$, ${\cal N}(\bar x)={\cal N}_\varepsilon(\bar\Sigma):=\S^n$, ${\cal N}_0(\bar y)={\cal N}_{\varepsilon,0}(S_\varepsilon^*(\lambda,\bar\Sigma))$) are satisfied. So Theorem \ref{thm-implt-fnct-thm} implies that there exists a neighborhood ${\cal N}_{\varepsilon,0}(\bar\Sigma)$ of $\bar\Sigma$ in $\S^n$ such that the mapping $\S^n\to\S^n$, $\Sigma\mapsto S_\varepsilon^*(\lambda,\Sigma')$ is $\frac{1}{\rho}$-Lipschitz continuous (i.e.\ $\kappa$-Lipschitz continuous for {$\kappa:=C_\lambda^2$}) on ${\cal N}_{\varepsilon,0}(\bar\Sigma)$, i.e.\ we have (\ref{eq:S^*ep-sigma-Lip-pre}).

{\em Step 2}. 
We now show that the inequalities (\ref{eq:S^*ep-sigma-Lip-pre}), $\bar\Sigma\in\PSD$, imply (\ref{eq:S^*_eps-Lip}). Let $\Sigma_1,\Sigma_2\in\PSD$ be arbitrary but fixed. 
Set $\Sigma(\theta):= \Sigma_1 + \theta(\Sigma_2-\Sigma_1)$, $\theta\in[0,1]$, and note that $\Gamma:=\{\Sigma(\theta):\theta\in[0,1]\}$ is the line segment that connects $\Sigma_1$ and $\Sigma_2$.

We first show that there exists an $\overline{\varepsilon}_{\Sigma_1,\Sigma_2}\in\R_{++}$ such that $S^*_\varepsilon(\lambda,\Sigma(\theta))\in \inmat{int}\,{\cal S}_\lambda$ for all $\varepsilon\in(0,\overline{\varepsilon}_{\Sigma_1,\Sigma_2}]$ and $\theta\in[0,1]$. For this it is sufficient to show that there exists a $\delta\in\R_{++}$ such that the $\delta$-hull $\mathfrak{S}_\delta:=\{S\in \PD:\|S-S^*(\lambda,\Sigma(\theta))\|\le\delta\mbox{ for at least one }\theta\in[0,1]\}$ of $\mathfrak{S}:=\{S^*(\lambda,\Sigma(\theta)):\theta\in[0,1]\}$ is contained in the interior $\inmat{int}\,{\cal S}_\lambda$ of ${\cal S}_\lambda$. In fact, by noting that the line segment $\Gamma$ is compact, the assertion then follows from Lemma~\ref{lem-for-proof-of-thm-glb-Lip-spars-precin-mtx} below. 
From Theorem \ref{thm-prci-mtx-Lip-hld-cnty} we know that the mapping $\PSD\to\PD$, $\Sigma\mapsto S^*(\lambda,\Sigma)$ is continuous. So the mapping $[0,1]\to\PD$, $\theta\mapsto S^*(\lambda,\Sigma(\theta))$ is continuous as well. In particular, its image, i.e.\ $\mathfrak{S}$, is compact in $\PD$. From Proposition \ref{prop:boundns-of-sln} and the definition of ${\cal S}_\lambda$ in (\ref{eq:Int-optimal-sln}), we also know that $S^*(\lambda,\Sigma(\theta))\in \inmat{int}\,{\cal S}_\lambda$ for any $\theta\in[0,1]$. So $\mathfrak{S}$ is a compact and thus closed set contained in the interior of ${\cal S}_\lambda$. To show that there exists a $\delta\in\R_{++}$ such that $\mathfrak{S}_\delta\subseteq\inmat{int}\,{\cal S}_\lambda$, we assume by way of contradiction that there is no such $\delta\in\R_{++}$. Then we can find a sequence $(S_n)_{n\in\N}$ such that $S_n\in\mathfrak{S}_{1/n}$ and $S_n\in\partial{\cal S}_\lambda$ for any $n\in\N$. Since $(S_n)_{n\in\N}$ is a bounded sequence in the finite-dimensional linear space $\S^n$, the Bolzano-Weierstrass theorem implies that there exists a convergent subsequence $(S_{n_k})_{k\in\N}$. Its limit (in $\S^n$), denoted by $S_0$, is clearly contained in the closed set $\partial{\cal S}_\lambda$.
It is also contained in the closed set $\mathfrak{S}$, since $S_{n_k}\in\mathfrak{S}_{1/n_k}$ for all $k\in\N$. Thus, $S_0\in\mathfrak{S}\cap\partial{\cal S}_\lambda$, which implies $\mathfrak{S}\cap\partial{\cal S}_\lambda\not=\varnothing$. But that is in contradiction to $\mathfrak{S}\subseteq\inmat{int}\,{\cal S}_\lambda$.

Now let $\varepsilon\in(0,\overline{\varepsilon}_{\Sigma_1,\Sigma_2}]$ be arbitrary but fixed. Let $M_\varepsilon\in\R_{++}$ be a constant (to be concretized later) and $\delta_\varepsilon\in(0,1/(2M_\varepsilon))$. Consider the line segments
\begin{align*}
    I_1^{\varepsilon} & \,:=\, \big\{\Sigma(\theta):\,\theta\in \big[0,1/M_\varepsilon+\delta_\varepsilon\big]\big\},\\
    I_i^{\varepsilon} & \,:=\, \big\{\Sigma(\theta):\,\theta\in \big[(i-1)/M_\varepsilon-\delta_\varepsilon,i/M_\varepsilon+\delta_\varepsilon\big]\big\},\quad i=2,\ldots,M_\varepsilon-1,\\
    I_{M_\varepsilon}^{\varepsilon} & \,:=\, \big\{\Sigma(\theta):\, \theta\in \big[1-1/M_\varepsilon-\delta_\varepsilon,1\big]\big\}.
\end{align*}
Then, $\bigcup_{i=1}^{M_\varepsilon} I_i^\varepsilon=\{\Sigma(\theta):\theta\in [0,1]\}$ and $I_i^\varepsilon\cap I_{i+1}^\varepsilon = \{\Sigma(\theta):\theta\in [i/M_\varepsilon-\delta_\varepsilon,i/M_\varepsilon+\delta_\varepsilon]\}$ for $i=1,\ldots,M_\varepsilon-1$. 

By the choice of $\varepsilon$ and Step 1, we know that for any $\bar\Sigma\in\{\Sigma(\theta):\theta\in [0,1]\}$ there exists a neighborhood $\mathcal{N}_{\varepsilon,0}(\bar\Sigma)$ of $\bar\Sigma$ in $\S^n$ such that the mapping $\PSD\to\S^n$, $\Sigma\mapsto S_\varepsilon^*(\lambda, \Sigma)$ is $\kappa$-Lipschitz continuous on $\mathcal{N}_{\varepsilon,0}(\bar\Sigma)$ (see (\ref{eq:S^*ep-sigma-Lip-pre})). Therefore, and in view of the finite covering theorem, we can assume that $M_\varepsilon$ is chosen so large that for any $i\in\{1,\ldots,M_\varepsilon\}$ we have 
$$
    \big\|S_\varepsilon^*(\lambda,\Sigma) - S_\varepsilon^*(\lambda,\Sigma^\prime)\big\| \leq \kappa \|\Sigma-\Sigma^\prime\|\quad\mbox{ for all }\Sigma,\Sigma^\prime\in I_i^\varepsilon.
$$
For $K_\varepsilon\in\mathbb{N}$, let $\Sigma^{(k)}:=\Sigma_1+\frac{k}{2^{K_\varepsilon}}(\Sigma_2-\Sigma_1)$, $k=0,\ldots,2^{K_\varepsilon}$, and assume that $K_\varepsilon$ is large enough that for any $k=1,\ldots,K_\varepsilon$, there exists an $i_k\in\{1,\ldots,M_\varepsilon\}$ such that both $\Sigma^{(k)}$ and $\Sigma^{(k+1)}$ lie in $I_{i_k}^\varepsilon$ (this can be achieved by choosing $K_\varepsilon$ such that $\frac{1}{2^{K_\varepsilon}}<\delta_\varepsilon$). Then we have $\|S_\varepsilon^*(\lambda,\Sigma^{(k)}) - S_\varepsilon^*(\lambda,\Sigma^{(k+1)})\| \leq \kappa \|\Sigma^{(k)}-\Sigma^{(k+1)}\|$ for $k=0,\ldots,2^{K_\varepsilon}-1$, and it follows that
\begin{align*}
    & \big\|S_\varepsilon^*(\lambda,\Sigma_1) - S_\varepsilon^*(\lambda,\Sigma_2)\big\| 
    \,\le\,  \sum_{k=0}^{2^{K_\varepsilon}-1}\big\|S_\varepsilon^*(\lambda,\Sigma^{(k)}) - S_\varepsilon^*(\lambda,\Sigma^{(k+1)})\big\|\\
    & \,\le\, \sum_{k=0}^{2^{K_\varepsilon}-1} \kappa\big\|\Sigma^{(k)}-\Sigma^{(k+1)}\big\| \,=\, \kappa \sum_{k=0}^{2^{K_\varepsilon}-1} \frac{1}{2^{K_\varepsilon}} \|\Sigma_1-\Sigma_2\| \,=\, \kappa\|\Sigma_1-\Sigma_2\|.
\end{align*}
This proves Theorem \ref{thm-glb-Lip-spars-precin-mtx}. However, we still need to prove the following two lemmas.


\begin{lemma}\label{lem-for-proof-of-thm-glb-Lip-spars-precin-mtx}
Let $\Gamma\subseteq \PSD$ be a compact set. Then for any $\delta\in\R_{++}$ there exists an $\varepsilon_{\Gamma,\delta}$ such that 
\begin{align*}
    \label{eq:S-S_i-nu-net-a}
    \big\|S^*_\varepsilon(\lambda,\Sigma^\prime)-S^*(\lambda,\Sigma^\prime)\big\|\leq \delta\quad\text{ for all } \Sigma^\prime\in\Gamma\text{ and }\varepsilon\in(0,\varepsilon_{\Gamma,\delta}].
\end{align*}
\end{lemma}

\begin{proof}
Let $\Gamma$ be a compact subset of $\PSD$. By Proposition~\ref{prop-smth-appx-prci-mtx}, we have 
    \[
        \|S^*_\varepsilon(\lambda,\Sigma')-S^*(\lambda,\Sigma')\|\leq R_{\lambda,\Sigma'}^{-1}\big(3\lambda n^2\sqrt{\varepsilon}\big)
    \]
for any $\varepsilon\in\R_{++}$ and $\Sigma'\in \Gamma$, where $R_{\lambda,\Sigma'}^{-1}(t):=\inf\{\sigma{\in\R_+}: R_{\lambda,\Sigma'}(\sigma)\geq t\}$, $t\in{\R_+}$, is the generalized inverse of the function $R_{\lambda,\Sigma'}:\R_+\to\R$ defined by~(\ref{prop-smth-appx-prci-mtx-lemma-eq}). Since $R_{\lambda,\Sigma'}^{-1}$ is non-decreasing, this shows that it suffices to show that for any $\delta\in \R_{++}$ there exists an $\varepsilon_{\Gamma,\delta}\in \R_{++}$ such that 
\begin{align}
    R_{\lambda,\Sigma'}^{-1}\big(3\lambda n^2\sqrt{\varepsilon_{\Gamma,\delta}}\big) \leq \delta\quad\mbox{ for all }\Sigma'\in \Gamma.\label{ineq:R-inv-delta}
\end{align}  

Let $\delta\in\R_{++}$ be arbitrary but fixed. Note that $R_{\lambda,\Sigma'}(\delta)>0$ for any $\Sigma'\in\Gamma$. Moreover, {Lemma~\ref{sec:M-continuity} below shows} that the function $\PSD\to\R$, $ \Sigma'\mapsto R_{\lambda,\Sigma'}(\delta)$ is continuous. Thus, {since $\Gamma$ is assumed to be compact,} it follows that $\inf_{\Sigma'\in \Gamma} R_{\lambda,\Sigma'}(\delta) > 0$. 
So we can choose $\varepsilon_{\Gamma,\delta}\in\R_{++}$ so small that 
\begin{align}
    3\lambda n^2\sqrt{\varepsilon_{\Gamma,\delta}}\leq R_{\lambda,\Sigma'}(\delta)\quad\mbox{ for all }\Sigma'\in \Gamma.
\label{ineq:eps-delta}
\end{align} 
Taking $R_{\lambda,\Sigma}^{-1}(\cdot)$ on both sides of~\eqref{ineq:eps-delta} gives
\[
    R_{\lambda,\Sigma'}^{-1}\big(3\lambda n^2\sqrt{\varepsilon_{\Gamma,\delta}}\big)\leq R_{\lambda,\Sigma'}^{-1}\left(R_{\lambda,\Sigma'}(\delta)\right) \leq \delta\quad\mbox{ for all }\Sigma'\in \Gamma,
\]
where the first inequality is due to the the monotonicity of $R_{\lambda,\Sigma}^{-1}(\cdot)$ and the second inequality is due to the definition of $R_{\lambda,\Sigma}^{-1}(\cdot)$. The above inequalities is~\eqref{ineq:R-inv-delta}, and this completes the proof. \hfill$\Box$
\end{proof}

}

{
\begin{lemma}\label{sec:M-continuity}
For any fixed $\lambda,\delta\in\R_{++}$, the mapping $\PSD\to\R$, $\Sigma\mapsto R_{\lambda,\Sigma}(\delta)$ is continuous, where $R_{\lambda,\Sigma}(\delta)$ is defined by (\ref{prop-smth-appx-prci-mtx-lemma-eq}).
\end{lemma}

\begin{proof}
Let $\lambda,\delta\in\R_{++}$ be arbitrary but fixed. For any $\Sigma\in\PSD$, we set $M_{\lambda,\delta}(\Sigma):=R_{\lambda,\Sigma}(\delta)$ with $R_{\lambda,\Sigma}(\delta)$ defined by (\ref{prop-smth-appx-prci-mtx-lemma-eq}). We have to show that the mapping $\PSD\to\R$, $\Sigma\mapsto M_{\lambda,\delta}(\Sigma)$ is continuous. To this end, it is sufficient to show that it is both lower and upper semi-continuous.

{\em Step 1}. We first show that $M_{\lambda,\delta}:\PSD\to\R$ is upper semi-continuous. { As a consequence of Theorem \ref{thm-prci-mtx-Lip-hld-cnty}, the mapping $\PSD\to\R$, $\Sigma\mapsto L(\lambda,\Sigma,S) - L(\lambda,\Sigma, S^*(\lambda,\Sigma))$ is continuous for any fixed $S\in\PD$. Then by~\cite[Proposition 1.26(a)]{ref:rockafellar1998variational}, as a pointwise infimum of a family of continuous functions, $M_{\lambda,\delta}:\PSD\to\R$ is upper semi-continuous.
}


{\em Step 2}. Now we show that $M_{\lambda,\delta}:\PSD\to\R$ is lower semi-continuous. Let $\Sigma\in\PSD$ and consider a sequence $(\Sigma_\ell)_{\ell\in\N}$ in $\PSD$ with $\Sigma_\ell\to \Sigma$. We have to show that $\liminf_{\ell\to\infty} M_{\lambda,\delta}(\Sigma_\ell)\ge M_{\lambda,\delta}(\Sigma)$.
By the definition of $M_{\lambda,\delta}(\Sigma)$
(see \eqref{prop-smth-appx-prci-mtx-lemma-eq}),
we can find, for every $\ell\in \N$, 
an ${1}/{\ell}$-optimal solution,
written $S_\ell\in \PD$, such that 
\begin{align}\label{ineq:M-n:0}
   \|S_\ell-S^*(\lambda,\Sigma_\ell)\|\geq \delta
\end{align}
and
\begin{align}\label{ineq:M-n}
   0\leq L(\lambda,\Sigma_\ell,S_\ell)-L(\lambda,\Sigma_\ell,S^*(\lambda,\Sigma_\ell)) \leq M_{\lambda,\delta}(\Sigma_\ell) + 1/\ell.
\end{align}
In what follows, we show that the sequence $(S_{\ell})_{\ell\in\N}$ is bounded and 
all of its cluster points lie in $\PD$.
Since $\limsup_{\ell\to\infty}M_{\lambda,\delta}(\Sigma_\ell)\le M_{\lambda,\delta}(\Sigma)$ (see Step 1), taking $\limsup$ in (\ref{ineq:M-n}) gives 
\begin{align*}
    \limsup_{\ell\to\infty}L(\lambda,\Sigma_\ell,S_\ell) 
    & \,\le\, M_{\lambda,\delta}(\Sigma) + \limsup_{\ell\to\infty} L(\lambda,\Sigma_\ell,S^*(\lambda,\Sigma_\ell))\\ 
    & \,=\,  M_{\lambda,\delta}(\Sigma) + L(\lambda, \Sigma, S^*(\lambda, \Sigma)),
\end{align*}
where the equality follows from the continuity of the mappings $\Sigma'\mapsto S^*(\lambda,\Sigma')$ and {$(\Sigma',S')\mapsto L(\lambda,\Sigma',S')$}. 
Thus, we can conclude that the sequence $(L(\lambda, \Sigma_\ell, S_\ell))_{\ell\in\N}$ is bounded above. 
Furthermore, if $\lambda_{\max}(S_\ell)$ is used to denote the largest eigenvalue of $S_\ell$, we have
\begin{align*}
    L(\lambda,\Sigma_\ell,S_\ell) 
    & \,=\, \la \Sigma_\ell, S_\ell\ra -\log(\det S_\ell) + \lambda \|S_\ell\|_1  \\
    & \,\ge\, -\log(\det S_\ell) + \lambda \|S_\ell\| \,\ge\, - n \log (\lambda_{\max}(S_\ell)) +\lambda\,\lambda_{\max}(S_\ell),
\end{align*}
where we used that $-\log(\det S_\ell) = -\log(\prod_{i=1}^n\lambda_i(S_\ell))= -\sum_{i=1}^n \log(\lambda_i(S_\ell)) \geq - n \log \lambda_{\max}(S_\ell)$ (with $\lambda_1(S_\ell),\ldots,\lambda_n(S_\ell)$ the eigenvalues of $S_\ell$) and $\la \Sigma_\ell, S_\ell\ra\ge 0$ (note that $\Sigma_\ell, S_\ell\in\PSD$). 
Now we can conclude that the sequence $(S_\ell)_{\ell\in \N}$ is also bounded above. 
Assume for the sake of a contradiction that it is unbounded. Then there exists a subsequence $(S_{\ell_k})_{k\in\N}$ with $\|S_{\ell_k}\|\to\infty$, which would imply that $\lambda_{\max}(S_{\ell_k})\to\infty$ (note that $\|S_{\ell_k}\|^2=\sum_{i=1}^n\lambda_i(S_{\ell_k})^2$). This would imply $L(\lambda,\Sigma_{\ell_k},S_{\ell_k})\to\infty$ (since $-n\log (x)+\lambda x\to\infty$ as $x\to\infty$), which is in contradiction to the boundness of the sequence $(L(\lambda, \Sigma_\ell, S_\ell))_{\ell\in \N}$. 

Since $(S_\ell)_{\ell\in\N}$ is a bounded sequence in the finite-dimensional linear space $\S^n$, we can conclude with the help of the Bolzano-Weierstrass theorem that there exists a convergent subsequence $(S_{\ell_k})_{k\in\N}$ with limit $S\in\S^n$. The limit $S$ is contained in $\PD$. In fact, 
$S\in \PSD$ (since $\PSD$ is the closure of $\PD$), and in the case $S\in\PSD\setminus\PD$ the convergence $S_{\ell_k}\to S\in \PSD\setminus \PD$ would imply $L(\lambda,{\Sigma_{\ell_k}},S_{\ell_k})\to\infty$ {(since $-\log\det(S_{\ell_k})\to\infty$)}, which would contradict 
the boundedness of $(L(\lambda, \Sigma_\ell, S_\ell))_{\ell\in \N}$. 
Moreover, by {(\ref{ineq:M-n:0}), the continuity of $\Sigma'\mapsto S^*(\lambda,\Sigma')$ and} the continuity of the norm $\|\cdot\|$, we have $\|S-S^*(\lambda,\Sigma)\|\geq \delta$. Therefore, $S$ is 
a feasible solution to 
the minimization problem (\ref{prop-smth-appx-prci-mtx-lemma-eq}).
Consequently 
we have
\begin{align}
\label{ineq:M-n-a}
    L(\lambda,\Sigma,{S})-L(\lambda,\Sigma,S^*(\lambda,\Sigma)) \geq M_{\lambda,\delta}(\Sigma).
\end{align}
Moreover, from~\eqref{ineq:M-n} and utilizing the continuity of the mappings $\Sigma'\mapsto S^*(\lambda,\Sigma')$
and  $(\Sigma',S')\mapsto L(\lambda,\Sigma',S')$, we can conclude that
\[
    L(\lambda,\Sigma,S)-L(\lambda,\Sigma,S^*(\lambda,\Sigma))\leq \liminf_{\ell\to\infty}M_{\lambda,\delta}(\Sigma_\ell).
\]
Combining the latter two inequalities gives $M_{\lambda,\delta}(\Sigma)\leq \liminf_{\ell\to\infty}M_{\lambda,\delta}(\Sigma_\ell)$. 
This completes the proof of lower semi-continuity of $M_{\lambda,\delta}$.  
{\hspace*{\fill}$\Box$\par\bigskip}
\end{proof} 
}


\subsection{Proof of Theorem \ref{thm-SR-Covn-mtx}}\label{Sec:Proof:thm-SR-Covn-mtx}

{Inequality (\ref{consistencyofT}) holds true since
\begin{align*}
    & \dd_{\rm K}\big(\P^P\circ \widehat\Sigma_N^{-1},\delta_{\Sigma_P}\big)
    =\sup_{f\in{\cal F}_1(\PSD)}\Big|\int f\,\diff\P^P\circ\widehat\Sigma_N^{-1}-\int f\,\diff\delta_{\Sigma_P}\Big|\\
    & =\sup_{f\in{\cal F}_1(\PSD)}\Big|\int \big(f(\widehat\Sigma_N)-f(\Sigma_P)\big)\,\diff\P^P\Big| 
    \le  \int \big\|\widehat\Sigma_N-\Sigma_P\big\|\,\diff\P^P
    = \E^P\big[\|\widehat\Sigma_N-\Sigma_P\|\big]
\end{align*}    
and Inequality in (\ref{consistencyofT-1}) can be obtained analogously. The proof of Inequalities (\ref{robustnessofT}) and (\ref{robustnessofT-1})} relies on the following lemma, where as before $L_2(\hat{x},\tilde{x})=\max\{1,\|\hat{x}\|,\|\tilde{x}\|\}$.

\begin{lemma}\label{eq:sampl-cvn-mtx-grwth}
For any $\hat{\bm{x}}=(\hat x^1,\ldots,\hat x^N),\,\tilde{\bm{x}}=(\tilde x^1,\ldots,\tilde x^N)\in(\R^n)^N$, we have
\begin{align}\label{eq:sampl-cvn-mtx-grwth-EQ}
    & \|\widehat\Sigma_N(\hat{\bm{x}})-\widehat\Sigma_N(\tilde{\bm{x}})\|\\
    &  \quad\leq \frac{2}{N}\sum_{i=1}^N L_2({\hat x}^i,{\tilde x}^i)\|{\hat x}^i-{\tilde x}^i\|+\frac{1}{N^2}\sum_{i=1}^N\big(\|{\hat x}^i\|+\|{\tilde x}^i\|\big)\sum_{j=1}^N \|{\hat x}^j-{\tilde x}^j\|,\nonumber\\
    & \|\widetilde\Sigma_N({\bm{x}})-\widetilde\Sigma_N({\bm{y}})\|\leq \frac{2}{N}\sum_{i=1}^N L_2({\hat x}^i,{\tilde x}^i)\|{\hat x}^i-{\tilde x}^i\|.\label{eq:sampl-cvn-mtx-grwth-EQ 2}
    \end{align}
\end{lemma}

\begin{proof}
For any $\hat{\bm{x}}=(\hat x^1,\ldots,\hat x^N)$, $\tilde{\bm{x}}=(\tilde x^1,\ldots,\tilde x^N)\in(\R^n)^N$, we obtain
\begin{align*}
    & \big\|\widehat\Sigma_N(\hat{\bm{x}})-\widehat\Sigma_N(\tilde{\bm{x}})\big\|\\
    & \leq \Big\|\frac{1}{N}\sum_{i=1}^N {\hat x}^i({\hat x}^i)^{\trans} - \Big(\frac{1}{N}\sum_{i=1}^N \hat x^i\Big)\Big(\frac{1}{N}\sum_{i=1}^N x^i\Big)^\trans\\
    & \qquad\quad-\frac{1}{N}\sum_{i=1}^N \tilde x^i(\tilde x^i)^\trans + \Big(\frac{1}{N}\sum_{i=1}^N \tilde x^i\Big)\Big(\frac{1}{N}\sum_{i=1}^N \tilde x^i\Big)^\trans\Big\|\\
    &\leq \Big\|\frac{1}{N}\sum_{i=1}^N \hat x^i(\hat x^i)^\trans-\frac{1}{N}\sum_{i=1}^N \tilde x^i(\tilde x^i)^\trans\Big\|\\
    & \qquad\quad+\Big\|\frac{1}{N^2}\Big(\sum_{i-1}^n \hat x^i\Big)\Big(\sum_{i-1}^n \hat x^i\Big)^\trans -\frac{1}{N^2}\Big(\sum_{i-1}^n \tilde x^i\Big)\Big(\sum_{i-1}^n \tilde x^i\Big)^\trans\Big\|\\
    &\leq \frac{1}{N}\sum_{i=1}^N \Big\|x^i(\hat x^i)^\trans-\tilde x^i(\tilde x^i)^\trans\Big\|+\frac{1}{N^2}\Big\|\Big(\sum_{i-1}^n\hat x^i\Big)\Big(\sum_{i-1}^n \hat x^i\Big)^\trans -\Big(\sum_{i-1}^n \tilde x^i\Big)\Big(\sum_{i-1}^n \tilde x^i\Big)^\trans\Big\|\\
    &\leq \frac{1}{N}\sum_{i=1}^N 2L_2(\hat x^i,\tilde x^i)\|\hat x^i-\tilde x^i\|+\frac{1}{N^2}\sum_{i=1}^N \big(\|\hat x_i\|+\|\tilde x_i \|\big)\sum_{i=1}^N \|\hat x_i-\tilde x_i\|.
\end{align*}
This proves Inequality (\ref{eq:sampl-cvn-mtx-grwth-EQ}). Inequality (\ref{eq:sampl-cvn-mtx-grwth-EQ 2}) can be shown analogously.
{{\hspace*{\fill}$\Box$\par\bigskip}}
\end{proof}

We are now ready to prove (\ref{robustnessofT}) and (\ref{robustnessofT-1}). By Inequality \eqref{eq:sampl-cvn-mtx-grwth-EQ}, the estimator $\widehat\Sigma_N$ satisfies condition \eqref{eq:general-SR-condition} with $\kappa_1=2$ and $\kappa_2=1$, which leads to (\ref{robustnessofT}) immediately by Theorem~\ref{thm-SR-gnl-nom-space}.
By Inequality \eqref{eq:sampl-cvn-mtx-grwth-EQ 2}, the estimator $\widetilde\Sigma_N$ satisfies condition \eqref{eq:general-SR-condition} with $\kappa_1=2$ and $\kappa_2=0$, which leads to (\ref{robustnessofT-1}) immediately by Theorem~\ref{thm-SR-gnl-nom-space}.


\subsection{Proof of Theorem \ref{thm-SR-Covn-mtx-with-mean}}

The proof of Theorem \ref{thm-SR-Covn-mtx-with-mean} can be carried out in the same way as the proof of ((\ref{robustnessofT}) and (\ref{consistencyofT}) in) Theorem \ref{thm-SR-Covn-mtx} (see Section \ref{Sec:Proof:thm-SR-Covn-mtx}). For this reason, we will not give the details here.


\subsection{Proof of Theorem \ref{thm-SR-Sprs-Preci-mtx}}

{
Inequality \eqref{eq:statistical-robustness-precision-matrix-1} is an immediate consequence of Theorem \ref{thm-glb-Lip-spars-precin-mtx}, Inequality \eqref{robustnessofT} in Theorem \ref{thm-SR-Covn-mtx} and Proposition \ref{prop:chain-rule}. Analogously, Inequality \eqref{eq:statistical-robustness-precision-matrix-2} follows from Theorem \ref{thm-glb-Lip-spars-precin-mtx}, Inequality \eqref{robustnessofT-1} in Theorem \ref{thm-SR-Covn-mtx} and Proposition \ref{prop:chain-rule}.}
Inequality (\ref{eq:consistency-precision-matrix-1}) holds true since
\begin{align*}
    & \dd_{\rm K}\big(\P^P\circ S^*(\lambda,\widehat\Sigma_N)^{-1},\delta_{S^*(\lambda,\Sigma_P)}\big)\\
    & \,=\,\sup_{f\in{\cal F}_1(\PSD)}\Big|\int f\,\diff\P^P\circ S^*(\lambda,\widehat\Sigma_N)^{-1}-\int f\,\diff\delta_{S^*(\lambda,\Sigma_P)}\Big|\\
    & \,=\,\sup_{f\in{\cal F}_1(\PSD)}\Big|\int \big(f(S^*(\lambda,\widehat\Sigma_N))-f(S^*(\lambda,\Sigma_P))\big)\,\diff\P^P\Big| \\
    & \,\le\,  \int \big\|S^*(\lambda,\widehat\Sigma_N)-S^*(\lambda,\Sigma_P)\big\|\,\diff\P^P
    \,=\, \E^P\big[\|S^*(\lambda,\widehat\Sigma_N)-S^*(\lambda,\Sigma_P)\|\big]\\
    &\,\le\, \frac{1}{\rho_\lambda}\E^P\big[\|\widehat\Sigma_N-\Sigma_P\|\big],
\end{align*}    
where the last step is ensured by Theorem \ref{thm-glb-Lip-spars-precin-mtx}, and Inequality in (\ref{eq:consistency-precision-matrix-2}) can be obtained analogously. 


\subsection{Proof of Proposition \ref{thm-SR-Covn-mtx:REMARK}}

Let $r\in[1,2)$. {We will only show the convergence in (\ref{thm-SR-Covn-mtx:REMARK-EQ2}). The convergences in (\ref{thm-SR-Covn-mtx:REMARK-EQ1}) and (\ref{thm-SR-Covn-mtx:REMARK-EQ3}) can be shown analogously. 
For any $P\in\mathscr{P}_{2r}(\R^N)$, we have}
\begin{align*}
    & \E^P\big[\|\widehat\Sigma_N-\Sigma_P\|_1\big]\\
    & \,=\,\E^P\Big[\sum_{j=1}^n\sum_{k=1}^n\Big|\frac{1}{N}\sum_{i=1}^N\xi_j^i\xi_k^i-\Big(\frac{1}{N}\sum_{i=1}^N\xi_j^i\Big)\Big(\frac{1}{N}\sum_{i=1}^N\xi_k^i\Big)-\Sigma_P(j,k)\Big|\Big]\\
    & \,\le\,\sum_{j=1}^n\sum_{k=1}^n\E^P\Big[\Big|\frac{1}{N}\sum_{i=1}^N\xi_j^i\xi_k^i-\E^P\big[\xi_j^1\xi_k^1\big]\Big|\Big]\\
    & \qquad +\sum_{j=1}^n\sum_{k=1}^n\E^P\Big[\Big|\Big(\frac{1}{N}\sum_{i=1}^N\xi_j^i\Big)\Big(\frac{1}{N}\sum_{i=1}^N\xi_k^i\Big)-\E^P\big[\xi_j^1\big]\E^P\big[\xi_k^1\big]\Big|\Big]\\
    & \,\le\,\sum_{j=1}^n\sum_{k=1}^n\E^P\big[\big|\overline{X}_{j,k;N}-\E^P[X_{j,k}]\big|\big] +\sum_{j=1}^n\sum_{k=1}^n\E^P\big[\big|\overline{\xi}_{j;N}-\E^P[\xi_{j}^1]\big||\overline{\xi}_{k;N}|\big]\\
    & \qquad +\sum_{j=1}^n\sum_{k=1}^n\E^P\big[|\xi_j^1|\big]\E^P\big[\big|\overline{\xi}_{k;N}-\E^P[\xi_{k}^1]\big|\big]\\
    & =:\, \sum_{j=1}^n\sum_{k=1}^nT_{1,N}(j,k)+\sum_{j=1}^n\sum_{k=1}^nT_{2,N}(j,k)+\sum_{j=1}^n\sum_{k=1}^nT_{3,N}(j,k),
\end{align*}
where $\overline{X}_{j,k,N}:=\frac{1}{N}\sum_{i=1}^N\xi_j^i\xi_k^i$, $X_{j,k}:=\xi_j^1\xi_k^1$ and $\overline{\xi}_{\ell;N}:=\frac{1}{N}\sum_{i=1}^N\xi_\ell^i$ for $j,k=1\ldots,n$. Now, let $(j,k)\in\{1,\ldots,n\}^2$ be arbitrary but fixed.

Since $X_{j,k}$ is $r$-fold $\P^P$-integrable (take into account that the random variables $\xi_j^1$ and $\xi_k^1$ are both $2r$-fold $\P^P$-integrable since $P\in\mathscr{P}_{2r}(\R^n)$ and $r\ge 1$), we have by Theorem 3.1 in \cite{Gut1978} that $\lim_{N\to\infty}N^{(r-1)/r}\,\E^P[|\overline{X}_{j,k;N}-\E^P[X_{j,k}]|^r]^{1/r}=0$. Thus, 
\[
\lim_{N\to\infty}N^{(r-1)/r}\,T_{1,N}(j,k)=\lim_{N\to\infty}N^{(r-1)/r}\,\E^P[|\overline{X}_{j,k;N}-\E^P[X_{j,k}]|]=0
\]
since $r\ge 1$. By Theorem 3.1 in \cite{Gut1978}, we also have $\lim_{N\to\infty}N^{(r-1)/r}\,\E^P[|\overline{\xi}_{\ell;N}-\E^P[\xi_k^1]|^r]^{1/r}=0$. Therefore, since $r\ge 1$, we get $\lim_{N\to\infty}N^{(r-1)/r}\,T_{3,N}(j,k)=\E^P[|\xi_j^1|]\limsup_{N\to\infty}N^{(r-1)/r}\,\E^P[|\overline{\xi}_{k;N}-\E^P[\xi_k^1]|]=0$. By H\"older's inequality, we have 
\[
\E^P[|\overline{\xi}_{j;N}-\E^P[\xi_{j}^1]||\overline{\xi}_{k;N}|]\le\E^P[|\overline{\xi}_{j;N}-\E^P[\xi_{j}^1]|^2]^{1/2}\E^P[|\overline{\xi}_{k;N}|^{2}]^{1/2}.
\]
The second factor in the latter upper bound is bounded above uniformly in $N\in\mathbb{N}$ since $\E^P[|\overline{\xi}_{k;N}|^{2}]^{1/2}\le\frac{1}{N}\sum_{i=1}^N\E^P[|\xi_{k}^i|^2]^{1/2}=\E^P[|\xi_{k}^1|^2]^{1/2}$ for all $N\in\mathbb{N}$ by Minkowski's inequality. For the first factor we have $N^{1/2}\,\E^P[|\overline{\xi}_{j;N}-\E^P[\xi_{j}^1]|^2]^{1/2}={\cal O}_N(1)$ by (2.11) in \cite{Gut1978}. Thus, we can conclude that $\lim_{N\to\infty}N^{(r-1)/r}\,\E^P[|\overline{\xi}_{j;N}-\E^P[\xi_{j}^1]|^2]^{1/2}=0$ since $(r-1)/r<1/2$ for our choice of $r$ (recall that $r<2$). Consequently, we obtain $\lim_{N\to\infty}N^{(r-1)/r}\,T_{2,N}(j,k)=0$.

In summary, we have shown that {(\ref{thm-SR-Covn-mtx:REMARK-EQ2})} does indeed hold.


\subsection{Proof of Theorem \ref{Statistical-robustness-of-eigenvalue}}\label{Sec:Proof:Statistical-robustness-of-eigenvalue}

{
Inequality (\ref{robustnessofeigens}) is an immediate consequence of \eqref{stability-of-eigenvalues}, Inequality \eqref{robustnessofT} in Theorem \ref{thm-SR-Covn-mtx} and Proposition \ref{prop:chain-rule}. Analogously, Inequality \eqref{robustnessofeigens-2} follows from \eqref{stability-of-eigenvalues}, Inequality \eqref{robustnessofT-1} in Theorem \ref{thm-SR-Covn-mtx} and Proposition \ref{prop:chain-rule}.}
Inequality (\ref{robustnessofeigens-3}) holds true since
\begin{align*}
    & \dd_{\rm K}\big(\P^P\circ\widehat\lambda_{i,N}^{-1},\delta_{\lambda_i^P}\big)\\
    & \,=\,\sup_{f\in{\cal F}_1(\R)}\Big|\int f\,\diff\P^P\circ\widehat\lambda_{i,N}^{-1}-\int f\,\diff\delta_{\lambda_i^P}\Big| \,=\,\sup_{f\in{\cal F}_1(\PSD)}\Big|\int \big(f(\widehat\lambda_{i,N})-f(\lambda_i^P)\big)\,\diff\P^P\Big| \\
    & \,\le\,  \int \big|\widehat\lambda_{i,N}-\lambda_i^P\big|\,\diff\P^P
    \,=\, \E^P\big[|\widehat\lambda_{i,N}-\lambda_i^P|\big] \,\le\, \E^P\big[\|\widehat\Sigma_N-\Sigma_P\|\big],
\end{align*}    
where the last step is ensured by \eqref{stability-of-eigenvalues}, and Inequality in (\ref{robustnessofeigens-4}) can be obtained analogously.

\subsection{Proof of Proposition \ref{thm:statistical-robustness-of-value-function-of-Markowitz-model-PROPOSITION}}\label{Sec:Proof:thm:statistical-robustness-of-value-function-of-Markowitz-model}

The assertion of Proposition \ref{thm:statistical-robustness-of-value-function-of-Markowitz-model-PROPOSITION} follows from Steps 1 and 2 below. Throughout this section, $z\in\R_+$ is a given constant. For any $(\mu,\Sigma)\in\R^n\times\PSD$, we use ${\cal W}(\mu, \Sigma)$ to denote the set of optimal solutions of problem \eqref{Markowitz-model}. 

{\em Step 1}. We first show that, for any $\mu\in\R^n$, the mapping $\PSD\to\R$, $\Sigma'\mapsto v(\mu,\Sigma')$ is (globally) Lipschitz continuous with Lipschitz constant 
$L_1:=1/2$. Indeed: For any $\mu\in\R^n$, we have by the well-known Danskin's theorem that the mapping $\Sigma':\PSD\to\R$, $\Sigma'\mapsto v(\mu,\Sigma')$ is convex and its subdifferential at $\Sigma\in\PSD$ can be represented as $\partial_\Sigma v(\mu, \Sigma) = \mathrm{conv}\{\tfrac{1}{2}ww^\trans:w\in {\cal W}(\mu, \Sigma)\}$. 
Since each $w$ from the feasible set of \eqref{Markowitz-model} satisfies $w^\trans \bm 1=1$ and $w\geq 0$, 
we have $\|\Sigma'\|\leq 1/2$ for all $\Sigma'\in\partial_\Sigma v(\mu, \Sigma)$. This means that, for any $\mu\in\R^n$, the mapping $\PSD\to\R$, $\Sigma'\mapsto v(\mu,\Sigma')$ is indeed Lipschitz continuous with Lipschitz constant $L_1:=1/2$. 

{\em Step 2}. We next show that, for any $C_1,C_2\in\R_{++}$ with $C_1<C_2$ and any $\Sigma\in\PSD$ with {$\|\Sigma\|\le 2C_2^2$}, the mapping $\R^n_{C_1,C_2}\to\R$, $\mu\mapsto v(\mu,\Sigma)$ is Lipschitz continuous with Lipschitz constant $L_{2,C_1,C_2}:=L_{C_1,C_2}-\frac{1}{2}$, where $\R^n_{C_1,C_2}:=\{m\in\R^n:(C_1^2+\la m,\bm 1\ra^2/n)^{1/2}\le\|m\|\le  C_2\}$  and the constant $L_{C_1,C_2}$ is defined as in Proposition \ref{thm:statistical-robustness-of-value-function-of-Markowitz-model-PROPOSITION}. 
Without loss of generality we may and do assume $\Sigma\not=\bm 0$ (since for $\Sigma=\bm 0$ we have $v(\mu,\Sigma)=0$ for all $\mu\in\R^n$). Compared to the subdifferential w.r.t.\ $\Sigma$ (see Step  1), it is slightly more sophisticated to derive, for any given $\Sigma\in\PSD$ with $0<\|\Sigma\|\le 2C_2^2$,
the subdifferential of the mapping $\R^n_{C_1,C_2}\to\R$, $\mu'\mapsto v(\mu',\Sigma)$ at $\mu\in\R^n_{C_1,C_2}$ (take into account that $\mu$ appears in the constraints of (\ref{Markowitz-model})). Our idea is to derive the Lagrange dual of the program (\ref{Markowitz-model}) and then apply the well-known Danskin's theorem.

Let $\Sigma\in\PSD$ with ${0<}\|\Sigma\|\le 2C_2^2$
and
$\mu\in\R^n_{C_1,C_2}$ be arbitrary but fixed. 
Problem (\ref{Markowitz-model}) satisfies Slater's condition, i.e.\ there exists a feasible solution $w_0$ such that the inequality in the constraints is strict at $w_0$ (i.e.\ $w_0>0$).
This means that strong duality holds. 
Consequently, we have
\begin{align}
    v(\mu,\Sigma)
    & \,=\, \max_{{\lambda_1,\lambda_2\in\R,}\,s\in\R_+^n} \min_{w\in\R^n} \Big(\frac{1}{2}w^\trans \Sigma w + \lambda_1(w^\trans\mu-z)+\lambda_2(w^\trans\bm 1-1)-s^\trans w\Big)\nonumber\\
    & \,=\, \max_{{\lambda_1,\lambda_2\in\R,}\,s\in\R_+^n} \min_{w\in\R^n} \Big(\frac{1}{2}w^\trans \Sigma w + (\lambda_1\mu+\lambda_2\bm 1-s)^\trans w-\lambda_1 z-\lambda_2\Big)\label{eq:dual-of-Markwitz-2}
\end{align}
Let $u_{\Sigma,1},\ldots,u_{\Sigma,n}$ be the eigenvalues of $\Sigma$ in decreasing order (i.e.\ $ u_{\Sigma,1}\geq u_{\Sigma,2}\ge\cdots\ge u_{\Sigma,r} > 0 = u_{\Sigma,r+1}=\cdots=u_{\Sigma,n}$) and $\Sigma=Q_\Sigma^\trans U_\Sigma Q_\Sigma$ be the spectral decomposition of $\Sigma$ with diagonal matrix $U_\Sigma:=\mathrm{diag}(u_{\Sigma,1},\ldots,u_{\Sigma,n})$. Then, if we set $U_\Sigma^{\dag}:=\mathrm{diag}(1/u_{\Sigma,1},\ldots,1/u_{\Sigma,r},0,\ldots,0)$, we have
\begin{align}
    & v(\mu,\Sigma)\nonumber\\
    & \,=\,\max\limits_{{\lambda_1,\lambda_2\in\R,}\,s\in\R_+^n} \Big(-\frac{1}{2}(\lambda_1\mu+\lambda_2\bm 1-s)^\trans Q_\Sigma^\trans U_\Sigma^{\dag} Q_\Sigma(\lambda_1\mu+\lambda_2\bm 1-s)-\lambda_1z-\lambda_2\Big)\nonumber\\
    & \qquad \text{ subject to }~ \lambda_1\mu+\lambda_2\bm 1-s\in\mathrm{span}(\Sigma), \label{eq:dual-of-Markwitz-3}
\end{align}
where $\mathrm{span}(\Sigma)$ is the space spanned by the rows of $\Sigma$ and the constraint $\lambda_1\mu+\lambda_2\bm 1-s\in\mathrm{span}(\Sigma)$ is obtained from the first order optimality condition with respect to $w$. {By the way:} In the case when $\Sigma$ is of full rank, \eqref{eq:dual-of-Markwitz-3} reduces to $v(\mu,\Sigma)=\max_{\lambda_1,\lambda_2,s\in\R_+^n}(-\frac{1}{2}(\lambda_1\mu+\lambda_2\bm 1-s)^\trans \Sigma^{-1}(\lambda_1\mu+\lambda_2\bm 1-s)-\lambda_1z-\lambda_2)$. 

Let ${\cal W}(\mu,\Sigma)$ and $\Lambda(\mu,\Sigma)$ denote the sets of optimal solutions 
of the max-min problem \eqref{eq:dual-of-Markwitz-2}. 
Applying Danskin's theorem to \eqref{eq:dual-of-Markwitz-2}, we have
\begin{eqnarray}\label{eq:v-mu-Lip}
    \partial_\mu v(\mu,\Sigma)=\mathrm{conv}\big\{\lambda_1^*w^*:\,
    (\lambda_1^*,\lambda_2^*,s^*)\in \Lambda(\mu,\Sigma),\,w^* \in {\cal W}(\mu,\Sigma)\big\}.
    \label{eq:Danskin-v}
\end{eqnarray}
For $L_{2,C_1,C_2}$-Lipschitz continuity of the mapping $\R^n_{C_1,C_2}\to\R$, $\mu'\mapsto v(\mu',\Sigma)$ at $\mu$, it is sufficient to show that $\|\mu'\|\le L_{2,C_1,C_2}$ for all $\mu'\in\partial_\mu v(\mu,\Sigma)$. 
In view of (\ref{eq:Danskin-v}) and $\|\lambda_1^* w^*\| \leq |\lambda_1^*|\|w^*\|\leq |\lambda_1^*|$, it is sufficient to show that $|\lambda_1^*|\le L_{2,C_1,C_2}$ for all $\lambda_1^*\in\Lambda(\mu,\Sigma)$. But this is a simple consequence of the following lemma. {Recall that $z\in\R_+$ is an arbitrary constant.}

\begin{lemma}\label{boundness-of-dual-variable-of-Markowitz}
Let $\Sigma\in\PSD\setminus\{\bm 0\}$
and $\mu\in\R^n\setminus\{a\bm 1:a\in\R\}$.
Then, for any $(\lambda_1^*,\lambda_2^*,s^*)\in\Lambda(\mu,\Sigma)$, we have
\begin{align}
    |\lambda_1^*| & \,\le\,\frac{16\|\Sigma\|_2}{9(\|\mu\|^2-{(\mu^\trans \bm 1)^2}/{n})}\Big(z+\frac{\|\mu\|}{\sqrt{n}}\Big)+\frac{8\|\Sigma\|_2}{9\sqrt{n}(\|\mu\|^2-{(\mu^\trans \bm 1)^2}/{n})}\,.\label{eq:lambda_1*-up-bnd}
    \end{align}
\end{lemma}

\begin{proof}
Let $\Sigma\in\PSD\setminus\{\bm 0\}$
and $\mu\in\R^n\setminus\{a\bm 1:a\in\R\}$ 
be arbitrary but fixed, and define a function $f:\R\times \R\times \R_+^{n}\to\R$ by $f(\lambda_1,\lambda_2,s):=-\frac{1}{2}(\lambda_1\mu+\lambda_2\bm 1-s)^\trans\Sigma^{\dag}(\lambda_1\mu+\lambda_2\bm 1-s)-\lambda_1z-\lambda_2$, where $\Sigma^{\dag}:=Q_\Sigma^\trans U_\Sigma^{\dag} Q_\Sigma$ with $Q_\Sigma$ and $U_\Sigma^\dag$ as introduced after (\ref{eq:dual-of-Markwitz-2}).
Let $(\lambda_1,\lambda_2,s)\in\{{(\lambda_1',\lambda_2',s')\in\R\times\R\times\R_+^n}:\lambda_1'\mu+\lambda_2'\bm 1-s'\in\mathrm{span}(\Sigma)\}$.  
Then we have
\begin{align}
    & f\big(\lambda_1/2,\lambda_2/2,s/2\big)-f\big(\lambda_1,\lambda_2,s\big)\nonumber\\
    & \,=\, \frac{3}{8}(\lambda_1\mu+\lambda_2\bm 1-s)^\trans \Sigma^{\dag}(\lambda_1\mu+\lambda_2\bm 1-s)+\frac{1}{2}\lambda_1z+\frac{1}{2}\lambda_2 \nonumber\\
    & \,\ge\, \frac{3}{8u_1}(\lambda_1\mu+\lambda_2\bm 1-s)^\trans(\lambda_1\mu+\lambda_2\bm 1-s) + \frac{1}{2}\lambda_1 z+\frac{1}{2}\lambda_2 \nonumber\\
    & \,=\,  \frac{3}{8u_1}\Big(\lambda_1^2 \|\mu\|^2+\lambda_2^2\|\bm 1\|^2+\|s\|^2+2\lambda_1\lambda_2\mu^\trans \bm 1-\lambda_1\mu^\trans s-\lambda_2 \bm 1^\trans s\Big)+\frac{1}{2}\lambda_1 z+\frac{1}{2}\lambda_2 \nonumber\\
    & {\,\ge\, \frac{3}{8u_1}\Big(\lambda_1^2 \|\mu\|^2+\lambda_2^2\|\bm 1\|^2+2\lambda_1\lambda_2\mu^\trans\bm 1 -\frac{1}{4}\|\lambda_1\mu +\lambda_2\bm 1\|^2\Big)+\frac{1}{2}\lambda_1 z+\frac{1}{2}\lambda_2}\nonumber\\
    & {\,=\, \frac{3}{8u_1}\Big(\frac{3}{4}\lambda_1^2 \|\mu\|^2+\frac{3}{4}\lambda_2^2\|\bm 1\|^2+\frac{3}{2}\lambda_1\lambda_2\mu^\trans\bm 1\Big)+\frac{1}{2}\lambda_1 z+\frac{1}{2}\lambda_2}\nonumber\\
    & {\,\ge\, \frac{9}{32u_1}\Big(\|\mu\|^2-\frac{(\mu^\trans \bm 1)^2}{n}\Big)\lambda_1^2+\Big(\frac{1}{2}z-\frac{\|\mu\|}{2\sqrt{n}}\Big)\lambda_1-\frac{2u_1}{9n}}\,,\label{boundness-of-dual-variable-of-Markowitz-PROOF-10}
\end{align}
where the second and third inequalities are satisfied by minimizing the corresponding terms with respect to $s$ and $\lambda_2$ and the first inequality can be justified as follows. Since $\lambda_1\mu+\lambda_2\bm 1-s\in\mathrm{span}(\Sigma)$, there exists a $w\in\R^n$ such that $\lambda_1\mu+\lambda_2\bm 1-s=\Sigma w$. Thus, we have $Q_\Sigma(\lambda_1\mu+\lambda_2\bm 1-s)=U_\Sigma Q_\Sigma w$, which means that only the first $r$ components of $Q_\Sigma(\lambda_1\mu+\lambda_2\bm 1-s)$ are non-zero, i.e.\ $Q_\Sigma(\lambda_1\mu+\lambda_2\bm 1-s)=(x_1,\ldots,x_r,0,\ldots,0)^\trans=:x$ for some $x_1,\ldots,x_r\in\R$.  
Thus,
\begin{align*}
    & (\lambda_1\mu+\lambda_2\bm 1-s)^\trans \Sigma^{\dag}(\lambda_1\mu+\lambda_2\bm 1-s)\\
    & \,=\, \big(Q_\Sigma(\lambda_1\mu+\lambda_2\bm 1-s)\big)^\trans U_\Sigma^{\dag}Q_\Sigma(\lambda_1\mu+\lambda_2\bm 1-s) \,=\, x^\trans U^{\dag} x \,=\, \sum_{i=1}^r \frac{1}{u_i}x_i^2\\
    & \,\ge\, \frac{1}{u_1}\sum_{i=1}^r x_i^2 = \frac{1}{u_1}\, x^\trans x \,=\, \frac{1}{u_1}(\lambda_1\mu+\lambda_2\bm 1-s)^\trans Q_\Sigma^\trans Q_\Sigma(\lambda_1\mu+\lambda_2\bm 1-s)\\
    & \,=\, \frac{1}{u_1}(\lambda_1\mu+\lambda_2\bm 1-s)^\trans (\lambda_1\mu+\lambda_2\bm 1-s),
\end{align*}
which shows that the first inequality in (\ref{boundness-of-dual-variable-of-Markowitz-PROOF-10}) does indeed hold.

Now, we can conclude {from \eqref{boundness-of-dual-variable-of-Markowitz-PROOF-10}} that for any $\lambda_1{\in\R}$ with 
\begin{align*}
    |\lambda_1|> \frac{16u_1}{9(\|\mu\|^2-{(\mu^\trans \bm 1)^2}/{n})}\Big(z+\frac{\|\mu\|}{\sqrt{n}}\Big)+\frac{8u_1}{9\sqrt{n}(\|\mu\|^2-{(\mu^\trans \bm 1)^2}/{n})}\,,
\end{align*}
the term $f(\lambda_1/2,\lambda_2/2,s/2)-f(\lambda_1,\lambda_2,s)$ is strictly positive. This implies (\ref{eq:lambda_1*-up-bnd}).
{{\hspace*{\fill}$\Box$\par\bigskip}}
\end{proof}

\section{Approximation of Fortet-Mourier metrics}\label{sec:app-fm-metric}

When $P$ and $Q$ are both discrete probability measures on $\R^n$ {with finite supports}, it turns out that the $p$-th order Fortet-Mourier distance $\dd_p(P,Q)$ $(p\ge 1)$ can be calculated precisely by solving a linear programming problem. In fact, in \cite[p.\,732]{heitsch-roemisch-2007} it is shown that for any two discrete probability measures $P$ and $Q$ on $\R^n$ {with finite supports $\{\xi^1,\ldots,\xi^{N}\}$ and $\{\zeta^1,\ldots,\zeta^{N}\}$} and probability point masses $\{p_i\}_{i=1}^{N}$ and $\{q_i\}_{i=1}^{N}$, respectively, the Fortet-Mourier distance $\dd_{p}(P,Q)$ is equal to the optimal value of the following linear programming problem:
\begin{align}
    \dd_p(P,Q) \,=\, \min_{\eta_{i,j}}\;\;&\sum_{i=1}^{{2N}}\sum_{j=1}^{{2N}} L_p(\xi^i,\xi^j) \eta_{i,j}\nonumber\\
    \mbox{s.t.}\;\;&\sum_{\ell=1}^{{2N}} \eta_{i,\ell} - \sum_{\ell=1}^{{2N}}\eta_{\ell,i} = p_i-q_i~\mbox{ for }i=1,\ldots,{2N},\nonumber\\
    &\eta_{i,j} \geq 0~\mbox{ for }i,j=1,\ldots,{2N},\label{prob:dl_p-discrete}
\end{align}
{where $\xi^{N+i}:=\zeta^i$ for $i=1,\ldots,N$, $p_i:=0$ for $i=N+1,\ldots,2N$, and $q_i:=0$ for $i=1,\ldots,N$.}

For general Borel probability measures $P$ and $Q$ on $\R^n$, the Fortet-Mourier distance $\dd_p(P,Q)$ can be approximated by a Monte Carlo simulation. In fact, let $\xi^1,\ldots,\xi^N$ and $\zeta^1,\ldots,\zeta^N$ be i.i.d.\ samples from $P$ and $Q$, respectively, {and define the corresponding empirical probability measures by $\widehat P_N:=\frac{1}{N}\sum_{i=1}^N\delta_{\xi^i}$ and $\widehat Q_N:=\frac{1}{N}\sum_{i=1}^N\delta_{\zeta^i}$, respectively. Moreover, set $\xi^{N+i}:=\zeta^{i}$ for $i=1,\ldots,N$.} Then by~\eqref{prob:dl_p-discrete}, {we get for large $N$ that $\dd_p(P,Q)\approx\widehat{\dd}_p(P,Q):=\dd_p(\widehat P_N,\widehat Q_N)$,} where
\begin{align}
     \dd_p(\widehat P_N,\widehat Q_N) \,=\, \min_{\eta_{i,j}}\;\;&\sum_{i=1}^{2N}\sum_{j=1}^{2N} L_p{(\xi^i,\xi^j)} \eta_{i,j},\nonumber\\
    \mbox{s.t.}\;\;&\sum_{\ell=1}^{{2N}} \eta_{i,\ell} - \sum_{\ell=1}^{{2N}}\eta_{\ell,i} = 1/N~\mbox{ for }i=1,\ldots,N,\nonumber\\
    &\sum_{\ell=1}^{{2N}} \eta_{i,\ell} - \sum_{\ell=1}^{{2N}}\eta_{\ell,i} = -1/N~\mbox{ for }i=N+1,\ldots,2N,\nonumber\\
    &\eta_{i,j} \geq 0~\mbox{ for }i,j=1,\ldots,2N.\label{prob:app-dl_p}
\end{align}

Specifically, in Section~\ref{sec:app-numerical}, when $P$ and $Q$ are assumed to be multivariate normal distributions {with positive definite covariance matrices}, i.e.\ when $P={\rm N}_{\mu_1,\Sigma_1}$ and $Q={\rm N}_{\mu_2,\Sigma_2}$ {for some $\mu_1,\mu_2\in\R^n$ and $\Sigma_1,\Sigma_2\in\PD$}, we draw samples and approximate $\dd_2(P,Q)$ as follows. Let $\varsigma^1,\ldots,\varsigma^N$ be i.i.d.\ samples from an $n$-variate standard normal distribution ${\rm N}_{0,I}$, and let $\Sigma_1=C_1C_1^\trans$ and $\Sigma_2=C_2C_2^\trans$ be the Cholesky decompositions of $\Sigma_1$ and $\Sigma_2$, respectively. If we define $\xi^i:={\mu_1+} C_1\varsigma^i$ and $\zeta^i:={\mu_2+}C_2\varsigma^i$ for $i=1,\ldots,N$, then $\xi^1,\ldots,\xi^N$ and $\zeta^1,\ldots,\zeta^N$ can be viewed as i.i.d.\ samples drawn from $P$ and $Q$, respectively. Inserting these samples into \eqref{prob:app-dl_p}, we obtain an approximation of $\dd_p(P,Q)$ by $\widehat{\dd}_p(P,Q):=\dd_p(\widehat P_N,\widehat Q_N)$. One of the advantages of this approach is that when $P=Q$, it gives $\widehat{\dd}_p(P,Q)=0$, which is consistent with the property of a metric. This approach is also applicable when $P$ and $Q$ are assumed to be (multivariate) log-normal distributions.





\bibliography{bibfile}
\bibliographystyle{plain}

\end{document}